\documentclass [leqno,11pt]{amsart}    
\usepackage{amssymb,amsmath,latexsym,amsfonts,amsbsy,amsthm,mathtools,graphicx,color}    
\usepackage{float}                 
\usepackage{hyperref}     
\usepackage{scalerel,stackengine}    
\stackMath       
\newcommand\reallywidehat[1]{%   
\savestack{\tmpbox}{\stretchto{% 
  \scaleto{%
    \scalerel*[\widthof{\ensuremath{#1}}]{\kern-.6pt\bigwedge\kern-.6pt}%
    {\rule[-\textheight/2]{1ex}{\textheight}}%WIDTH-LIMITED BIG WEDGE
  }{\textheight}% 
}{0.5ex}}% 
\stackon[1pt]{#1}{\tmpbox}%  
}

\definecolor{myback}{RGB}{204,232,207}

\usepackage{cancel}
\usepackage[dvipsnames]{xcolor}
\numberwithin{equation}{section}

\allowdisplaybreaks

\numberwithin{equation}{section}

\allowdisplaybreaks
\let\b=\beta
\let\g=\gamma
\let\d=\delta

\let\la=\lambda

\let\s=\sigma
\let\f=\frac

\let\om=\omega

\let\Om=\Omega

\let\pa=\partial
\let\ep=\epsilon

\def\no{\noindent}

\def\eqdef{\buildrel\hbox{\footnotesize def}\over =}

\def\bbT{\mathbb{T}}

\def\fc{{\mathfrak c}}
\newcommand{\beq}{\begin{equation}}
\newcommand{\eeq}{\end{equation}}
\newcommand{\ben}{\begin{eqnarray}}
\newcommand{\een}{\end{eqnarray}}
\newcommand{\beno}{\begin{eqnarray*}}
\newcommand{\eeno}{\end{eqnarray*}}

\newtheorem{theorem}{Theorem}[section]

\newtheorem{lemma}[theorem]{Lemma}
\newtheorem{proposition}[theorem]{Proposition}
\newtheorem{corol}[theorem]{Corollary}
\newtheorem{remark}[theorem]{Remark}
\begin{document}

\title{A dynamical approach to the study of instability near Couette flow}

\author{Hui Li}
\address{Department of Mathematics, New York University Abu Dhabi, Saadiyat Island, P.O. Box 129188, Abu Dhabi, United Arab Emirates.}
\email{lihuiahu@126.com, lihui@nyu.edu}

\author{Nader Masmoudi}
\address{NYUAD Research Institute, New York University Abu Dhabi, Saadiyat Island, Abu Dhabi, P.O. Box 129188, United Arab Emirates\\
Courant Institute of Mathematical Sciences, New York University, 251 Mercer Street New York, NY 10012 USA}
\email{masmoudi@cims.nyu.edu}

\author{Weiren Zhao}
\address{Department of Mathematics, New York University Abu Dhabi, Saadiyat Island, P.O. Box 129188, Abu Dhabi, United Arab Emirates.}
\email{zjzjzwr@126.com, wz19@nyu.edu}

\begin{abstract}
In this paper, we obtain the optimal instability threshold of the Couette flow for Navier-Stokes equations with small viscosity $\nu>0$, when the perturbations are in the critical spaces $H^1_xL_y^2$. More precisely, we introduce a new dynamical approach to prove the instability for some perturbation of size $\nu^{\frac{1}{2}-\delta_0}$ with any small $\delta_0>0$, which implies that $\nu^{\frac{1}{2}}$ is the sharp stability threshold. In our method, we prove a transient exponential growth without referring to eigenvalue or pseudo-spectrum. As an application, for the linearized Euler equations around shear flows that are near the Couette flow, we provide a new tool to prove the existence of growing modes for the corresponding Rayleigh operator and give a precise location of the eigenvalues. 

\end{abstract}
\maketitle
\section{Introduction} 
\subsection{Hydrodynamic stability problems} 
Hydrodynamic stability is an active field of fluid mechanics that deals with the stability and instability of fluid flows. The field of hydrodynamic stability started in the nineteenth century with Stokes, Helmholtz, Reynolds, Rayleigh, Kelvin, Orr, Sommerfeld, and many others. The study of (in)stability of shear flows dates back to Rayleigh \cite{Rayleigh1880}, Kelvin \cite{Kelvin1887}, and Sommerfeld \cite{Sommerfeld1908}.

We consider the instability problem of shear flows for both viscous and inviscid fluids. Let us first introduce the two-dimensional incompressible Navier-Stokes and Euler equations in $\Omega=\bbT\times\mathbb{R}$:
\begin{equation}\label{eq-nl-NS}
  \left\{
    \begin{array}{ll}
      \pa_tU+U\cdot\nabla U+\nabla P-\nu\Delta U=0,\\
      \nabla\cdot U=0,
    \end{array}
  \right.
\end{equation}
where $\nu\ge0$ is viscosity. We denote by $U=(U^{(1)},U^{(2)})$ the velocity and $P$ the pressure. Let $W=-\pa_yU^{(1)}+\pa_xU^{(2)}$ be the vorticity, which satisfies
\begin{align*}
  \pa_t W+U\cdot\nabla W-\nu\Delta W=0.
\end{align*}
Let $b(t,y)$ solve the heat equation:
\begin{equation}\label{eq:heat}
  \left\{
    \begin{array}{l}
      \pa_tb(t,y)-\nu\pa_{yy}b(t,y)=0,\\
      b(0,y)=b_{in}(y).
    \end{array}
  \right.
\end{equation}
Then the shear flow $(b(t,y),0)$ is a solution of \eqref{eq-nl-NS} with vorticity $W=-\pa_yb$. For the Euler case ($\nu=0$), the shear flows $\big(b(y),0\big)$ are steady solutions to Euler equations. The special case, $b(t,y)=y$, namely the Couette flow $(y,0)$ is a steady solution of \eqref{eq-nl-NS} with $W=-1$ for both $\nu>0$ and $\nu=0$.

In this paper, we focus on the (in)stability of the shear flow $(b(t,y),0)$. It is natural to introduce the perturbation. Let $u=(u^{(1)},u^{(2)})=U-(b(t,y), 0)$ and $\om=W-(-\pa_yb)$, then $\omega$ satisfies
\begin{equation}\label{eq:b-perturbation}
  \left\{
    \begin{array}{ll}
      \partial_{t} \omega+b(t,y) \partial_{x} \omega-\nu \Delta \omega-\pa_{yy}b(t,y)u^{(2)}=-u \cdot \nabla \omega , \\
      u=-\nabla^{\perp}(-\Delta)^{-1} \omega=\big(\pa_y(-\Delta)^{-1} \omega,-\pa_x(-\Delta)^{-1} \omega\big),\\
      \omega|_{t=0}=\omega_{i n}.
    \end{array}
  \right.
\end{equation}
If $b(t,y)=y$, the equation is simpler:
\begin{equation}\label{eq-nl-om-chan}
  \left\{
    \begin{array}{ll}
      \partial_{t} \omega+y \partial_{x} \omega-\nu \Delta \omega=-u \cdot \nabla \omega , \\
      u=-\nabla^{\perp}(-\Delta)^{-1} \omega=\big(\pa_y(-\Delta)^{-1} \omega,-\pa_x(-\Delta)^{-1} \omega\big),\\
      \omega|_{t=0}=\omega_{i n}.
    \end{array}
  \right.
\end{equation}
 
The traditional starting point of the study of flow (in)stability usually consists of two stages. The first one is to linearize the system around the background shear flow. For Couette flow, the linearized equation is a simple transport-diffusion equation:
\begin{equation}\label{eq:LNS2}
  \left\{
    \begin{array}{ll}
      \partial_{t} \omega+y \partial_{x} \omega-\nu \Delta \omega=0,\\
      \psi=\Delta^{-1}\omega,\\
      \omega |_{t=0}=\omega_{in},
    \end{array}
  \right.
\end{equation}
where $\psi$ is the stream function. In \cite{Kelvin1887} Kelvin considered system \eqref{eq:LNS2} and showed that if $\hat{\om}(t,k,\eta)$ is the Fourier transform of $\om(t,x,y)$, then the solution of \eqref{eq:LNS2} can be written as
\beq\label{eq: Lin-sol}
\begin{split}
&\hat{\om}(t,k,\eta)=\hat{\om}_{in}(k,\eta+kt)\exp\left(-\nu\int_0^t|k|^2+|\eta-ks+kt|^2ds\right),\\
&\hat{\psi}(t,k,\eta)=\f{-\hat{\om}_{in}(k,\eta+kt)}{k^2+\eta^2}\exp\left(-\nu\int_0^t|k|^2+|\eta-ks+kt|^2ds\right),
\end{split}
\eeq
which gives that
\begin{equation}\label{eq: ED_and_ID}
  \begin{aligned}    
&\|\pa_yP_{\neq}\psi\|_{L^2}+\langle t\rangle \|\pa_xP_{\neq}\psi\|_{L^2}\leq C\langle t\rangle^{-1}\|P_{\neq}\om_{in}\|_{H^{2}},\\
&\|P_{\neq}\om\|_{L^2}\leq C\|P_{\neq}\om_{in}\|_{L^2}e^{-c\nu t^3},
  \end{aligned}
\end{equation}
where we denote by $P_{\neq}f=f(x,y)-\frac{1}{2\pi}\int_{\mathbb{T}}f(x,y)dx$ the projection to the nonzero mode of $f$. The first inequality in \eqref{eq: ED_and_ID} is the {\bf inviscid damping} and the second one is the {\bf enhanced dissipation}. 
The study of nonlinear stability is more difficult, and experiments show that any small perturbation to the linear shear can lead to the transition from the laminar shear flow to turbulence when the Reynolds number is large enough, which is the so-called Couette-Sommerfeld paradox \cite{LiLin2011,FalGiaMul2019} (also called the turbulence paradox \cite{Birkhoff1960}). It was suggested by Lord Kelvin \cite{Kelvin1887} that indeed the flow may be stable, but the stability threshold is decreasing as $\nu\to 0$, resulting in transition at a finite Reynolds number in any real application. In \cite{BGM2017}, Bedrossian, Germain, and Masmoudi formulated the following stability threshold problem:

{\it Given a norm $\|\cdot\|_X$, find a $\beta=\beta(X)$ so that
\begin{align*}
  &\|\omega_{in}\|_{X}\leq \nu^{\beta} \Rightarrow \text{stability},\\
&\|\omega_{in}\|_{X}\gg \nu^{\beta} \Rightarrow \text{instability}.
\end{align*}
}
One can reformulate this in terms of the enhanced dissipation and inviscid damping for the nonlinear system which yield asymptotic stability:

{\it Given a norm $\|\cdot\|_{X}$ $(X\subset L^2)$, determine a $\beta=\beta(X)$ so that if the initial vorticity satisfies $\|\om_{in}\|_{X}\ll \nu^{\beta}$, then for $t>0$
\begin{equation}\label{eq: enha-invis}
  \|\omega_{\neq}\|_{L^2_{x,y}}\leq C\|\omega_{in}\|_{X}e^{-c\nu^{\frac{1}{3}}t}\quad \text{and}\quad
\|u_{\neq}\|_{L^2_{t,x,y}}\leq C\|\omega_{in}\|_{X},
\end{equation}
hold for the Navier-Stokes equation \eqref{eq-nl-om-chan}.}

We summarize some of the recent stability results for the 2D Couette flow in the following tables: 
\begin{table}[H]
\centering
\caption{2D Couette flow} 
\medskip
 \begin{tabular}{|c|c|c|c| }
\hline
Space & $\beta$ &Boundary&Reference \\
\hline
$H_x^{log}L_y^2$ & $\b\geq\frac{1}{2}$ &No&\cite{BVW2018,MasmoudiZhao2020cpde} \\
\hline
$H^{\s}$ & $\b\geq \frac{1}{3}$ &No&\cite{MasmoudiZhao2019} \\
\hline
Gevrey-${2_-}$ & $\b=0$ &No&\cite{BMV2016} \\
\hline
Gevrey-${\f1s}$ & $s=\f{1-3\b}{2-3\b},\ \b\in[0,\f13]$ &No&\cite{LMZ-G2022} \\
\hline
$H^{1}$ & $\b\geq \frac{1}{2}$ &Non-slip&\cite{CLWZ2020} \\
\hline
\end{tabular} 
\end{table} 

To consider the instability problem, the linearization around the Couette flow fails to give growth. A modified linearization is needed, namely, one may consider the linearized equation around a well-chosen shear flow near the Couette flow. More precisely, we consider the linearized equation
\begin{equation}\label{eq:b-perturbation-linear}
  \left\{
    \begin{array}{ll}
      \partial_{t} \omega+b(t,y) \partial_{x} \omega-\nu \Delta \omega-\pa_{yy}b(t,y)u^{(2)}=0, \\
      u=-\nabla^{\perp}(-\Delta)^{-1} \omega=\big(\pa_y(-\Delta)^{-1} \omega,-\pa_x(-\Delta)^{-1} \omega\big)\\
      \omega |_{t=0}=\omega_{i n}. 
    \end{array}
  \right.
\end{equation}
We also introduce the linearized operator $\mathcal{L}_{\nu,b}=b(t,y) \partial_{x} -\pa_{yy}b(t,y)\pa_x(\Delta)^{-1}-\nu \Delta$. 
Then we can rewrite \eqref{eq:b-perturbation-linear} as
\begin{align*}
  \pa_t\omega+\mathcal L_{\nu,b}\omega=0.
\end{align*}

The second stage of the traditional approach is to look for unstable eigenvalues of the linearized problem or to study the resolvent, namely to study its eigenvalue problem (consider first the case $b(y)$ is independent of time $t$)
\ben\label{eq:eigenvalu-pb}
\mathcal{L}_{\nu,b}\Om(\la,x,y)=\la\Om(\la,x,y).
\een
It is also related to the corresponding resolvent $(\la-\mathcal{L}_{\nu,b})^{-1}$. 
If $\nu=0$, then $\mathcal{L}_{0,b}$ is the Rayleigh operator for the linearized Euler equation \cite{Lin2003,Howard1961}, and \eqref{eq:eigenvalu-pb} can be written as the Rayleigh equation, and if $\nu>0$, then $\mathcal{L}_{\nu,b}$ is the Orr–Sommerfeld operator for the linearized Navier-Stokes equation \cite{GGN2016,LiLin2011}, and \eqref{eq:eigenvalu-pb} can be written as the Orr–Sommerfeld equation \cite{Sommerfeld1908}. Both equations are elliptic equations with degenerate or singular coefficients. Suppose that there exists $(\la,\Om)$ with $Re\, \la>0$, then the solution with initial data $\Om(\la,x,y)$
\beno
\om(t,x,y)=\Om(\la,x,y)e^{\la t}
\eeno 
grows exponentially. The eigenvalue problem is helpful in the study of normal operators or operators that are close to normal. For the non-normal operator, the growth may not be due to the existence of eigenvalue. In a remarkable paper entitled `Hydrodynamic Stability Without Eigenvalues' \cite{TTRD1993}, the authors used the $\ep$-pseudo-spectrum ($\s_{\ep}(\mathcal L)=\{\la\in \mathbb{C} \, \|(\la-\mathcal L)^{-1}\|\geq \ep^{-1}\}$) and `pseudoresonance' to explain the growth. More precisely, for a nonnormal operator $\mathcal L$, due to the non-normality, `pseudoresonance' can occur even if $\la$ is far from the spectrum. This resonance leads to a transient growth of the semigroup $e^{-t\mathcal L}$. There are many numerical results using pseudo-spectrum to explain the `subcritical transition to turbulence' phenomenon \cite{HRST2002, RSH1993, TTS1999}. However, there are very few mathematical rigorous results, since the $\ep$-pseudo-spectrum set $\s_{\ep}(\mathcal L)$ is difficult to calculate in a mathematically rigorous way. 

For the threshold problem, the recent stability results take advantage of the change of coordinate, which allows us to get a very precise long-time dynamic. Here we are at the interface between stability and instability, which allows us to borrow some ideas from the study of stability problems. We develop a dynamical approach to obtain the exponential lower bound estimate for the {{}solution operator} generated by $\mathcal{L}_{\nu,b}$ without referring to eigenvalues or studying the resolvent. More precisely, we use the change of coordinate to absorb the transport term and study the resonance caused by the nonlocal term at each critical time. As a consequence, such  {{}estimate for the solution operator} gives the optimal instability threshold. Here the word resonance refers to a strong effect caused by one frequency on another one.

\subsection{Optimal instability threshold}
We consider the viscous flow and study the optimality of the size of the initial perturbation. More precisely, we consider the case $X=H^1_xL_y^2$ and show that $\nu^{\f12}$ is the stability threshold, i.e., 
\beno
\|\omega_{in}\|_{H^1_xL_y^2}\gg \nu^{\frac{1}{2}-} \Rightarrow \text{instability}.
\eeno
Here we recall that the stability part, namely
\beno
\|\omega_{in}\|_{H^{log}_xL_y^2}\le \varepsilon\nu^{\frac{1}{2}} \Rightarrow \text{stability}
\eeno
is proved in \cite{MasmoudiZhao2020cpde}.

Our first result states as follows: 
\begin{theorem}\label{thm-low-exp}
Let $\nu>0$ be small enough. For any small $\delta_0>0$, there exist $M>0$ independent of $\nu$ and shear flows $(b_\nu(t,y),0)$ such that $\pa_{t}b_\nu-\nu\pa_{yy}b_\nu=0$, with $b_{\nu}(0,y)=b_{in}(y)$ satisfying
\begin{align}
 CM\nu^{\f12-\d_0}\ge  \|b_{in}(y)-y\|_{L^{\infty}\cap \dot{H}^1}\ge cM\nu^{\f12-\d_0}
\end{align}
and the linear and nonlinear enhanced dissipation \eqref{eq: enha-invis} fail for the shear flow $(b_\nu(t,y),0)$. 

More precisely, there exist $\omega_{in}(x,y)$ with $\int_{\mathbb{T}}\om_{in}(x,y)dx=0$ such that 
the solution of the linear system \eqref{eq:b-perturbation-linear} with initial data $\omega_{in}$ satisfies: 
for $t\in \big[0,T\big]$ with $T=\varepsilon_1\nu^{-\frac{1}{3}+\frac{2}{3}\delta_0}\ln (\nu^{-1})$ and $\varepsilon_1>0$ small enough depending on $\delta_0$ and independent of $\nu$, 
\begin{align}\label{eq-est-linear-viscous}
   C\|\omega_{in}\|_{H^1_xL_y^2}e^{C\nu^{\frac{1}{3}-\frac{2}{3}\delta_0}t}\ge\|\omega_{\pm 1}(t)\|_{L^2_{x,y}}\ge c\|\omega_{in}\|_{H^1_xL_y^2}e^{c\nu^{\frac{1}{3}-\frac{2}{3}\delta_0}t}.
\end{align}
If the initial data $\omega_{in}(x,y)$ satisfies $\int_{\mathbb{T}}\om_{in}(x,y)dx=0$ and $\|\omega_{in}\|_{H^1_xL_y^2}\approx \varepsilon_0\nu^{\frac{1}{2}+\delta_1}$, then the solution of the nonlinear system \eqref{eq:b-perturbation} with initial data $\omega_{in}$ satisfies: 
for $t\in \big[0,T\big]$ with $T=\varepsilon_1\nu^{-\frac{1}{3}+\frac{2}{3}\delta_0}\ln (\nu^{-1})$ and $\varepsilon_1>0$ small enough depending on $\delta_0,\d_1$ and independent of $\nu$, 
\begin{align}\label{eq-est-nonlinear-viscous}
  C\|\omega_{in}\|_{H^1_xL_y^2}e^{C\nu^{\frac{1}{3}-\frac{2}{3}\delta_0}t}\ge\|\omega_{\pm 1}(t)\|_{L^2_{x,y}}\ge c\|\omega_{in}\|_{H^1_xL_y^2}e^{c\nu^{\frac{1}{3}-\frac{2}{3}\delta_0}t}.
\end{align}
Here $C, c>0$ are constants independent of $t,\nu,\varepsilon_1$ and $f_{\pm 1}(x,y)=\frac{1}{2\pi}\int_{\mathbb{T}}f(x,y)e^{-(\pm)ix}dxe^{\pm ix}$ is the $\pm 1$ Fourier mode.

In particular at $t=T$, for both cases, 
\begin{align}\label{eq-growth-T}
   \|\omega_{\pm 1}(T)\|_{L^2_{x,y}}\ge \frac{c}{\nu^{c\varepsilon_1}}\|\omega_{in}\|_{H^1_xL_y^2}.
\end{align} 
\end{theorem}
Let us compare this result with the previous stability threshold and explain the optimality. We first introduce a corollary of the results in \cite{MasmoudiZhao2020cpde}: 
\begin{corol}\label{col:easy-corol}
  Let $\nu>0$ be small enough, there exist $\varepsilon_0>0$ independent of $\nu$, such that for every shear flow $(b(t,y),0)$ solving $\pa_{t}b-\nu\pa_{yy}b=0$, with $b(0,y)=b_{in}(y)$ satisfying
\begin{align*}
  \|b_{in}(y)-y\|_{L^\infty\cap \dot{H}^1}\leq \varepsilon_0\nu^{\frac{1}{2}},
\end{align*}
the linear and nonlinear enhanced dissipation holds for the shear flow $(b(t,y),0)$. 

More precisely, for every $\omega_{in}(x,y)$ the solution of the linear system \eqref{eq:b-perturbation-linear} with initial data $\omega_{in}$ satisfies: 
\begin{align*}
  \|\omega_{\neq}(t)\|_{L^2_{x,y}}\le C\|\omega_{in}\|_{H^{log}_xL_y^2}e^{-c\nu^{\frac{1}{3}}t}.
\end{align*}
If the initial data satisfies $\|\omega_{in}\|_{H^1_xL_y^2}\le\varepsilon_0\nu^{\frac{1}{2}}$, then the solution of the nonlinear system \eqref{eq:b-perturbation} with initial data $\omega_{in}$ satisfies: 
\begin{align*}
  \|\omega_{\neq}(t)\|_{L^2_{x,y}}\le C\|\omega_{in}\|_{H^{log}_xL_y^2}e^{-c\nu^{\frac{1}{3}}t}.
\end{align*}
\end{corol}
Corollary \ref{col:easy-corol} seems stronger than the results stated in \cite{MasmoudiZhao2020cpde}, as the Couette flow is the special case with $b_{in}(y)-y=0$. However, both results are equivalent. The main reason is that the difference between shear flow $b(t,y)$ and Couette flow is the same size as the perturbation $\omega_{in}$. Linearizing around the Couette flow and the shear flow $\big(b(t,y),0\big)$ are the same. However, if the deviation of the shear flow $b(t,y)$ from the Couette flow is slightly larger, then the linearization around the Couette flow is not accurate anymore. Instead of studying \eqref{eq:LNS2}, it is better to study \eqref{eq:b-perturbation-linear}, which can be regarded as a secondary linearization, see \cite{BOH1988,gill1965, LiLin2011} for some linear instability results. Due to the dissipation effect, such linear growth is a transient growth, which could trigger nonlinear instability, and lead to the transition to turbulence \cite{Waleffe1995, KarpCohen2017}. The secondary linearization gives a possible resolution to the Couette-Sommerfeld paradox.

\begin{remark}
  We emphasize that the constants $c$ and $\varepsilon_1$ in \eqref{eq-growth-T} are independent of $\nu$. Since we consider the small viscosity problem, at time $T$, the amplification can be made as large as we want by taking the limit $\nu\to 0+$.
\end{remark}

\begin{remark}
In \cite{gill1965}, Gill considers the instability of the shear flow $\big(b(y),0\big)$ of the form
\begin{align*}
  b(y)=y+\gamma^2M \Upsilon (\frac{y}{\gamma}).
\end{align*}
He points out that if the shear flow is to become turbulent, then $\gamma^3/\nu$ must necessarily be large. Generally, nonlinear growth happens before the transition. By a formal argument, he obtained the instability and gave the growth rate $e^{C\gamma t-c\nu \gamma^{-1}t^2}$. This is transient growth. The exponential decay part $e^{-c\nu \gamma^{-1}t^2}$ is due to the enhanced dissipation which stabilizes the linear system. Based on such a rate, he conjectured that the time taken for the growth should be of order $\gamma^{-1}\ln(\gamma^{-1})$. 
In our paper, we take $\gamma=\nu^{\frac{1}{3}-\frac{2}{3}\delta_0}$, and give a rigorous proof of the instability with a more precise growth rate as well as a control of the growth time. Moreover, we prove the enhanced dissipation of the solution to the linearized equation \eqref{eq:b-perturbation-linear}, see Appendix \ref{appendix-C}.
\end{remark}
\begin{remark}
  In Theorem \ref{thm-low-exp} we find a flow that gradually deviates from a shear flow. Indeed, the $L^2$ norm of the nonzero mode of the total vorticity $\mathcal P_{\neq}W(t,x,y)=\mathcal P_{\neq}\omega(t,x,y)$ grows exponentially in time, and the enstrophy transfers from the zero mode to the nonzero modes. We track the evolution of $\omega(t,x,y)$ till $T=\varepsilon_1\nu^{-\frac{1}{3}+\frac{2}{3}\delta_0}\ln (\nu^{-1})$. Due to the enstrophy conservation law for the Couette flow i.e. $\|W(t)+1\|_{L^2}\le\|W(0)+1\|_{L^2}$, the exponential growth should stop at some finite time for the nonlinear problem. The growing time $T$ in Theorem \ref{thm-low-exp} is optimal in terms of $\nu$ up to a constant.

  After time $T$, there are two possibilities. One is that the nonzero modes decay back to $0$, the flow first approaches a shear flow and then the Couette flow. The second one is that the nonzero modes do not decay immediately. One may expect that the laminar flow transits to turbulence or that the flow forms some cat's eye structure. In both cases, due to the dissipation effect, the flow will approach the Couette flow as $t\to+\infty$.
\end{remark}

\subsection{Instability of the inviscid flow}
The traditional approach of studying the (in)stability problem is to get the semigroup estimate by studying the eigenvalue problem \cite{Lin2003} or the resolvent \cite{HelSjo2010}. Now we can get the upper and lower bound of the semigroup directly. Moreover, we can deduce some information about the existence and the location of eigenvalues from the lower bound estimate of the semigroup.

We study the linearized Euler equation around the shear flow $\big(b_{0}(y),0\big)$:
\begin{align}\label{eq-Euler}
  \pa_t\omega+\mathcal R_{M,\gamma}\omega=0,
\end{align}
 where
\begin{align}\label{eq-b0}
  b_{0}(y)=\int^y_01+2\sqrt\pi M\gamma e^{-\frac{y'^2}{\gamma^2}}dy',
\end{align}
and
  \begin{align}\label{eq-rayleigh-ori}
   \mathcal R_{M,\gamma}=b_{0}(y)\pa_x-b_{0}''(y)\pa_x(\Delta)^{-1}
  \end{align}
 is the Rayleigh operator.  
We have the following instability result.
\begin{theorem}\label{thm-invisid}
  Let $M_0>0$ be big enough. For each $M\ge M_0$ there exists $0<\gamma_0=\gamma_0(M)$ such that for $0<\gamma\le\gamma_0$, the Rayleigh operator $\mathcal R_{M,\gamma}$ has an unstable eigenvalue $\lambda=\lambda_r+i\lambda_i$ such that $-CM\gamma\le \lambda_r<-\gamma$ and $|\lambda_i|\leq 4\gamma\sqrt{\ln\big(\ln(\gamma^{-1})\big)}$. As a consequence, there exists $\omega_{in}(x,y)\in L^2_{x,y}$ such that for $\forall t>0$
  \begin{align}
    \|e^{-t\mathcal R_{M,\gamma}}\omega_{in}\|_{L^2_{x,y}}\ge C^{-1}e^{\gamma t}\|\omega_{in}\|_{L^2_{x,y}}.
  \end{align}
  Here the constant $C>0$ is independent of $M$ and $\gamma$.
\end{theorem}

\begin{remark}
  In \cite{Lin2003}, Lin proved the existence of (un)stable eigenvalues for the shear flow in the $\mathcal K^+$ class (see \cite{Lin2003} for the definition) with additional spectral assumption on the corresponding Schr{\"o}dinger type operator. By our dynamical approach, to obtain the existence of a growing mode, it is not necessary to check the spectral assumption. 
\end{remark}
\begin{remark}
  For the nonlinear instability, we refer to \cite{Grenier2000}. We remark that such nonlinear instability also holds for the shear flow $\big(b_{0}(y),0\big)$. 
\end{remark}
\subsection{New ideas and potential application}
Let us now highlight some new ideas in our dynamic approach. On one hand, if the instability problem is studied by finding unstable eigenvalues of $\mathcal{L}$, and the existence of the solution to the eigenvalue problem $\mathcal{L}\Om(\la,x,y)=\la\Om(\la,x,y)$ is obtained by a bifurcation or a fixed point argument, then the growth mechanism remains somehow unclear. Thus if the operator $\mathcal{L}$ is perturbed by the non-compact operator $-\nu\Delta$, then it is difficult to approach the instability problem, since the spectrum may change.  From Theorem \ref{thm-invisid} we can see that the growing mode $\la$ of the linearized operator $\mathcal{L}=y\pa_x-b''\pa_x\Delta^{-1}$ is small. There is a big cancellation between the transport term $y\pa_x\Om(\la,x,y)$ and the nonlocal term $b''\pa_x\Delta^{-1}\Om(\la,x,y)$, which could be broken when the diffusion term $-\nu \Delta\Om(\la,x,y)$ is added to the system. 

On the other hand, to absorb the transport part $y\pa_x$, the change of coordinate $(x,y)\to (z,y)$ with $z=x-ty$ (or its nonlinear modification) was introduced in many proofs of stability results \cite{BM2015,BGM2017,BVW2018}. Our method uses this idea to get rid of the transport term, thus the linearized operator after the change of coordinates becomes $\tilde{\mathcal{L}}_{\nu,b}=-b''\pa_z\Delta_{L}^{-1}-\nu\Delta_{L}$ with $\Delta_{L}=\pa_{zz}+(\pa_v-t\pa_z)^2$. So we only need to compare the effects of the nonlocal term $b''\pa_z\Delta_{L}^{-1}$ and the new diffusion $\nu\Delta_{L}$, which are both time-dependent operators. It turns out that for $t\ll\nu^{-\f13}$, the growth due to $b''\pa_z\Delta_{L}^{-1}$ is stronger than the diffusion effect. This time scale can be seen easily in the $(z,y)$ coordinate, and such information is hidden in the $(x,y)$ coordinate.  Moreover, after taking Fourier transform, the resonance between different frequencies becomes clear. We benefit from these good structures. 

Let us also mention the robustness of our method. 1. It can treat the time-dependent operator. 2. It provides a new way of proving the growth without solving \eqref{eq:eigenvalu-pb}. Moreover, such growth can be obtained for a large class of initial data instead of taking the initial data from some eigenspace. 3. It shows how the information transfers from one frequency to another and how the resonance happens. 4. It provides a new way of proving the existence of eigenvalues.

At last, let us mention some potential applications. There are a lot of linearized operators in fluid dynamics containing a transport term, some nonlocal terms, and a diffusion term. The transport term leads to the mixing which together with the diffusion term stabilizes the system. The nonlocal term causes potential growth. One can use the change of coordinate which is usually used in the stability problem to absorb the transport terms. In the new coordinate system, the information will move along the time. Linearized operators with such structure also appear in plasma physics, such as the Vlasov equation. Also, parallel to the transport structure, one can apply the idea to the study of the eigenvalues of the Schr\"odinger operator with (complex) potential, where the free  Schr\"odinger operator plays the same role as the transport term that offers scattering which stabilizes the system, and the potential plays the same role as the nonlocal term that creates some growth. 

\subsection{Historical comments of (in)stability results}
In this subsection, we present some historical comments. 
In studying the stability of Couette flow of inviscid flow, Orr \cite{Orr1907} observed an important phenomenon that the velocity tends to 0 as $t\to \infty$. This phenomenon is called inviscid damping, which is the analog in hydrodynamics of Landau damping found by Landau \cite{Landau1946}, which predicted the rapid decay of the electric field of the linearized Vlasov equation around the homogeneous equilibrium. Mouhot and Villani \cite{MouhotVillani2011} made a breakthrough and proved nonlinear Landau damping for the perturbation in Gevrey class(see also \cite{BMM2016}). For the inviscid damping, the mechanism leading to the damping is the vorticity mixing driven by shear flows or Orr mechanism \cite{Orr1907}. See \cite{MSZ2020, RenZhao2017, Ryutov1999} for similar phenomena in various systems. The nonlinear inviscid damping was first proved by Bedrossian and Masmoudi \cite{BM2015} for the perturbations in the Gevrey-$m$  class ($1\leq m<2$). We also refer to \cite{IonescuJia2020cmp, IonescuJia2021} and references therein for other related interesting results.

Due to the presence of the nonlocal term, the inviscid damping for general shear flows is a challenging problem even at the linear level. In the case of the finite channel, Case \cite{Case1960} gave a
formal proof of $t^{-1}$ decay for the velocity. 
Lin and Zeng \cite{LinZeng2011} gave the optimal linear decay estimates of the velocity
for data in Sobolev spaces. Zillinger \cite{Zillinger2017} proved the linear inviscid damping for a class of monotone shear flows that are close to Couette flow. Wei, Zhang, and Zhao \cite{WeiZhangZhao2018} proved the linear inviscid damping for general monotone shear flows. We also refer to \cite{Jia2020siam, Jia2020arma} for a simplified proof and the linear inviscid damping in Gevrey class. 
For non-monotone flows such as the Poiseuille flow and the Kolmogorov flow, another phenomenon should be taken into consideration, which is the so-called vorticity depletion phenomenon, predicted by Bouchet and Morita \cite{BM2010} and later proved by Wei, Zhang, and Zhao \cite{WeiZhangZhao2019, WeiZhangZhao2020}. See also \cite{BCV2017, I-Jiahao2021, LMZZ2021} for similar phenomena in vortex dynamics and MHD. Very recently, Ionescu-Jia \cite{IJ2020}, and Masmoudi-Zhao \cite{MasmoudiZhao2020} proved that the nonlinear inviscid damping holds for general linear stable monotone shear flows. 

The instability of shear flow $(b(y),0)$ for the Euler equation is also well-studied. Rayleigh \cite{Rayleigh1880} proved that if the shear flow $(b(y),0)$ is linearly unstable, then $b''(y)$ must change sign. Howard \cite{Howard1961} proved the semicircle theorem which describes the possible location of eigenvalues. Lin \cite{Lin2003} proved that if the shear flow $(b(y),0)$ is in some class then this shear flow is unstable. We refer to \cite{belenkaya1999,friedlander1997, Grenier2000, Lin2003, Vishik2003} for the instability results of different shear flows. All the results are obtained by studying the Rayleigh equations. For the asymptotic instability, 
Lin and Zeng \cite{LinZeng2011} proved that nonlinear inviscid damping is not true for perturbations of the Couette flow in $H^{s}$ ($s<\f32$). 
Deng and Masmoudi \cite{DM2018}  proved some instability for initial perturbations in Gevrey-$m$ class ($m>2$). 
We also refer to \cite{DZ2020, DZ2021}, where the instability of some toy models related to linearized Euler equations was studied.

For the viscous fluid, there is the enhanced dissipation phenomenon, namely, the decay rate is much faster than the diffusive decay rate of $e^{-\nu t}$. The mechanism leading to enhanced dissipation is also due to vorticity mixing. Generally speaking, the sheared velocity sends information to a higher frequency, enhancing the effect of diffusion. This is a common phenomenon in the viscous fluid. Beside the 2D Couette flow, we refer to \cite{BGM2015,BGM2017,BGM2020,WeiZhang2020,ChenWeiZhang2020} for the enhanced dissipation of the 3D Couette flow, to  \cite{LinXu2019,IMM2019,WeiZhangZhao2020,LiWeiZhang2020} for Kolmogorov flow, and to \cite{CotiElgindiWidmayer2020,del2021,ding2020enhanced} for Poiseuille flow.  We also refer to \cite{DRM2021,BC2017,Coti2020,Wei2021,He2021,li2021metastability,GNRS2020} for the enhanced dissipation in different models.

\no{\bf Notations}: Through this paper, we will use the following notations. 
We use $C$ (or, $c$) to denote a positive big (or, small) enough constant which may be different from line to line. We also use $C_M$ (or, $c_M$) to emphasize that such a constant depends on a variable $M$. 

We use $f\lesssim g$ ($f\approx g$) to denote
\begin{align*}
  f\le C g\quad (C^{-1}g\le f\le C g).
\end{align*}

Given a function $f(t,y)$, we denote its derivation in $y$ by
\begin{align*}
  f'(t,y)=\pa_yf(t,y),\qquad f''(t,y)=\pa_y^2f(t,y),
\end{align*}
and denote its Fourier transform in $x$ by
\begin{align*}
  \tilde f(k,y)=\frac{1}{2\pi}\int_{\mathbb T} f(x,y)e^{- i\xi x} dx,
\end{align*}
and denote its Fourier transform in $(x,y)$ by
\begin{align*}
  \hat f(k,\xi)=\frac{1}{4\pi^2}\int_{\mathbb T}\int_{\mathbb R} f(x,y)e^{-ikx}e^{- i\xi y} dy dx.
\end{align*}
We denote the projection to the $k$th mode of $f(x,y)$ by
\begin{align*}
  P_kf(x,y)=f_k(x,y)=\frac{1}{2\pi}\left(\int_{\mathbb{T}}f(x',y)e^{-ikx'}dx'\right)e^{ikx},
\end{align*}
and denote the projection to the non-zero mode by
\begin{align*}
  P_{\neq}f(x,y)=f_{\neq}(x,y)=\sum_{k\in\mathbb Z/0}f_k(x,y).
\end{align*}
We also use $\hat f_k(\xi)$ to denote $\hat f(k,\xi)$ to emphasize it is the Fourier transform of the $k$ mode. For a function $f(t,x,y)$ we introduce the following function spaces which are of the same spirit as the Chemin-Lerner's Besov space \cite{CheminLerner1995},
\begin{align*}
  \|f(t)\|_{\mathcal F L^1_kL^q_y}=\sum_{k\in\mathbb Z}\left(\int_{\mathbb R} |\tilde f_k(t,y)|^qd y\right)^{\frac{1}{q}},
\end{align*}
and
\begin{align*}
  \|f\|_{\tilde L^p_t\left([0,T];\mathcal F L^1_kL^q_y\right)}=\sum_{k\in\mathbb Z}\left(\int_{0}^T\left(\int_{\mathbb R} |\tilde f_k(t,y)|^qd y\right)^{\frac{p}{q}} d t\right)^{\frac{1}{p}}.
\end{align*}
\section{Main ideas and sketch of the proof}
In this section, we present the main ideas of the dynamical approach and the proof of instability. We study the linearized system around a shear flow $\big(b_{\nu}(t,y),0\big)$ where  
\begin{align}\label{eq-shear-b}
  b'_\nu(t,y)=1+\frac{2\sqrt\pi M\gamma^2}{\sqrt{4\nu t+\gamma^2}}e^{-\frac{y^2}{4\nu t+\gamma^2}},\quad b_\nu(t,0)=0.
\end{align}
It is easy to check that $b_\nu(t,y)$ solves \eqref{eq:heat} with initial data
\begin{align*}
  b_{in}(y)=\int^y_01+2\sqrt\pi M\gamma e^{-\frac{y'^2}{\gamma^2}}dy'.
\end{align*}
For the viscous problem ($\nu>0$), $b_\nu(t,y)$ varies with time. And for the inviscid problem ($\nu=0$), we study the  time-independent shear flow $\big(b_0(y),0\big)$ with 
\begin{align}\label{eq-Euler-shear}
  b'_0(y)=1+2\sqrt\pi M\gamma e^{-\frac{y^2}{\gamma^2}},\quad b_0(0)=0.
\end{align}
It holds for $\nu\geq 0$ that
\begin{align*}
  \sup_{t\geq 0}\|b_\nu(t,y)-y\|_{L^\infty_y}&=\pi M\gamma^2,\\
  \|b_\nu(t,y)-y\|_{\dot H^1_y}&=\frac{(2\pi)^{\frac{3}{4}}M\gamma^2}{(4\nu t+\gamma^2)^{\frac{1}{4}}},
\end{align*}
which means that the shear flow $\big(b_{\nu}(t,y),0\big)$ is a perturbation of Couette flow. 

We construct the shear flow based on the Gaussian function as its time evolution through the heat equation has a precise formula. We remark that in \cite{LinZeng2011} the authors also chose Gaussian-related functions to prove Kelvin's cat's eye structure near Couette flow for the inviscid problem.

For the viscous problem, we consider the case $\nu=\gamma^{\frac{3}{1-2\delta_0}}$ for any small $\delta_0>0$, so that this shear flow $\big(b_\nu(t,y),0\big)$ is close to the Couette flow in the sense:
\begin{align*}
   \sup_{t\in [0,\gamma^2/\nu]}\|b_\nu(t,y)-y\|_{L^\infty\cap \dot H^1}\approx M\nu^{\frac{1}{2}-\delta_0}.
\end{align*} 
\subsection{The dynamical approach}
Usually, the key step for studying (non)linear instability is finding the growing mode of the linearized operators. This approach is effective in dealing with time-independent linearized operators \cite{GGN2016, Lin2003}. However, it is hard to reduce the instability problem to the eigenvalue problem if the linearized operator is time-dependent, especially, for problems where the distribution of eigenvalue varies in time. For the viscous problem, the background shear flow $b_\nu(t,y)$ varies with time, so the corresponding linearized operator is time-dependent. 

 We first introduce the modified linearized equation for $\nu\ge0$:
 \begin{equation}\label{eq-g-linear}
    \partial_{t}  g+ y\partial_{x}  g-\pa_y^2b_\nu(t,y)\pa_x(\Delta)^{-1} g-\nu \Delta  g=0,
 \end{equation}
where we replace the transport term $b_\nu(t,y)\pa_xg$ by $y\pa_x g$ and keep the nonlocal term $-\pa_y^2b_\nu(t,y)\pa_x(\Delta)^{-1} g$, since $b_\nu(t,y)-y$ is small. The control of the difference $(b_\nu(t,y)-y)\partial_{x}  g$ (between \eqref{eq:b-perturbation-linear} and \eqref{eq-g-linear})  will be carefully studied in Section 4. 

It is natural to introduce the linear change of coordinate $z=x-ty$ and let
\begin{align*}
  h(t,z,y)=h(t,x-ty,y)=g(t,x,y).
\end{align*}
Then $h$ satisfies
\begin{equation}\label{eq-lin-h}
  \partial_{t}  h(t,z,y)+\frac{4\sqrt\pi M\gamma^2y}{(4\nu t+\gamma^2)^{\frac{3}{2}}}e^{-\frac{y^2}{4\nu t+\gamma^2}}\pa_z(\Delta_L)^{-1}h(t,z,y)-\nu \Delta_L h(t,z,y)=0,
\end{equation}
where $\Delta_L=\pa_{zz}+(\pa_y-t\pa_z)^2$. By taking Fourier transform in $(z,y)$, we get that
\begin{equation}\label{eq-h-ft-ori}
  \begin{aligned}    
      \partial_{t} \hat h_k(t,\xi) &+\nu\big(k^2+(\xi-kt)^2\big)\hat h_k(t,\xi)\\
  &-\int_{\mathbb{R}}M\gamma^2(\xi-\eta)e^{-(\nu t+\frac{\gamma^2}{4})|\xi-\eta|^2}\frac{k\hat h_k(t,\eta)}{(\eta-kt)^2+k^2} d\eta =0,
  \end{aligned}
\end{equation}
For the viscous problem $(\nu=\gamma^{\frac{3}{1-2\delta_0}})$, the possible growth only happens for low frequencies in $z$ and $t\lesssim \nu^{-\f13-\f{16}{9}\d_0}$, see Appendix \ref{appendix-C}. This shows that the growth obtained for the viscous flow is transient growth. However, for the inviscid problem, the growth could be sustained all the time, and this allows us to prove the existence of unstable eigenvalues.

To capture the growth, we choose the initial data that has only $\pm1$ modes,
\begin{equation}\label{eq-h-ini0}
  h_{in}(z,y)=2\cos (z)\mathfrak h_{in}(y),\quad \hat {\mathfrak h}_{in}(\xi)=\left\{
    \begin{array}{ll}
      1,&|\xi|\le 2\varepsilon_1\gamma^{-1}\ln(\gamma^{-1}),\\
      0,&|\xi|> 2\varepsilon_1\gamma^{-1}\ln(\gamma^{-1}).
    \end{array}
  \right.
\end{equation}
The corresponding initial data of the linear system \eqref{eq-g-linear} is
\begin{equation}\label{eq-ini-g}
  g_{in}(x,y)=2\cos (x)\mathfrak h_{in}(y).
\end{equation}
From \eqref{eq-h-ft-ori} and the choice of $h_{in}$, one can easily check that $\hat h_k(t,\xi)\equiv0$ for $k\neq\pm1$, $\hat h_{-1}(t,-\xi)=\hat h_1(t,\xi)$ for all $t\ge0$ and $\xi\in \mathbb R$. Here the constant $M>0$ is large enough determined in \eqref{eq-M}, and the constant $\varepsilon_1$ is small enough determined in \eqref{eq-epsilon1}. 

For the viscous problem $(\nu=\gamma^{\frac{3}{1-2\delta_0}})$, we can see that $b_\nu(t,y)$ is close to $ b_{in}(y)$ for $t\in[0,\varepsilon_1\nu^{-\frac{1}{3}+\frac{2}{3}\delta_0}\ln (\nu^{-1})]$, and the dissipation effect is weak. Thus for both viscous and inviscid problems, we only need to catch the growth of $\hat h_1(t,\xi)$ which satisfies
\begin{equation*}
  \left\{
    \begin{array}{l}
      \partial_{t} \hat h_1(t,\xi)=\int_{\mathbb{R}}M\gamma^2(\xi-\eta)e^{-\gamma^2|\xi-\eta|^2}\frac{\hat h_1(t,\eta)}{(\eta-t)^2+1} d\eta+\text{lower order terms},\\
      \hat h_1(0,\xi)=\hat {\mathfrak h}_{in}(\xi).
    \end{array}
  \right.
\end{equation*}
The kernel 
\begin{align*}
  M\gamma^2(\xi-\eta)e^{-{\gamma^2}|\xi-\eta|^2}\frac{1}{(\eta-t)^2+1}\lesssim \gamma
\end{align*}
reaches its maximum at $\eta=t,\ \xi=\eta+ \frac{\sqrt2}{2}\gamma^{-1}$. The scenario we are most interested in is a low-to-high transition. On one hand, at time $t$, those frequencies $\eta$ which are close to $t$ have a strong effect that excites those frequencies $\xi\in[t+\gamma^{-1},t+3\gamma^{-1}]$:
\begin{align*}
  \pa_t\hat h_1(t,\xi)|_{\xi\in[t+\f1{10}\gamma^{-1},t+3\gamma^{-1}]}\gtrsim \gamma \inf_{|\eta-t|\le C}\hat h_1(t,\eta).
\end{align*}
On the other hand, those frequencies $\xi$ which are close to $t$ do not change too much. 
%\begin{align*}
%\left|\pa_t\hat h_1(t,\xi)|_{|\xi-t|\le \f1{10}\g^{-1}}\right|\lesssim \g.
%\end{align*}

For the formal argument, it is enough to follow the growth of two frequencies. Formally, let us assume 
\beq\label{eq:formal}
\inf_{|\eta-t|\le C}\hat h_1(t,\eta)\gtrsim \hat h_1(0,0),\quad \text{for}\quad 0\leq t\leq \f12\g^{-1}
\eeq
then for $t$ close to $0$, the two frequencies $\xi=\g^{-1}$ and $\xi=2\g^{-1}$ are forced by $\hat h_1(0,0)$. 
From $t=0$ to $t=\gamma^{-1}$ we ignore the effect when $t$ and $\xi$ are close. Thus we get 
\begin{align*}
\hat h_1(\gamma^{-1},\gamma^{-1})
&\gtrsim \hat h_1(0,\gamma^{-1})+\int_0^{\f1{2}\g^{-1}}\pa_t\hat h_1(s,\g^{-1})ds\\
&\gtrsim \hat h_1(0,\gamma^{-1})+\hat h_1(0,0),
\end{align*}
and
\begin{align*}
h_1(\gamma^{-1},2\gamma^{-1})\gtrsim \hat h_1(0,2\gamma^{-1})+\int_0^{\g^{-1}}\pa_t\hat h_1(s,\g^{-1})ds\gtrsim \hat h_1(0,2\gamma^{-1})+\hat h_1(0,0).
\end{align*}
Let us use the following lower bound instead:
\begin{align*}
  &\hat h_1(\gamma^{-1},\gamma^{-1})=\hat h_1(0,\gamma^{-1})+\hat h_1(0,0),\quad \hat h_1(\gamma^{-1},2\gamma^{-1})=\hat h_1(0,2\gamma^{-1})+\hat h_1(0,0).\\
  &\hat h_1(\gamma^{-1},(1+m)\gamma^{-1})=\hat h_1(0,(1+m)\gamma^{-1}),\quad m\geq 2
\end{align*}
 In general, from $t=(j-1)\gamma^{-1}$ to  $t=j\gamma^{-1}$ with $j\geq 1$,
 \begin{align*}
   \hat h_1\big(j\gamma^{-1},j\gamma^{-1}\big)=&\hat h_1\big((j-1)\gamma^{-1},j\gamma^{-1}\big)+\hat h_1\big((j-1)\gamma^{-1},(j-1)\gamma^{-1}\big),\\
   \hat h_1\big(j\gamma^{-1},(j+1)\gamma^{-1}\big)=&\hat h_1\big((j-1)\gamma^{-1},(j+1)\gamma^{-1}\big)+\hat h_1\big((j-1)\gamma^{-1},(j-1)\gamma^{-1}\big),\\
   \hat h_1\big(j\gamma^{-1},(j+m)\gamma^{-1}\big)=&\hat h_1\big((j-1)\gamma^{-1},(j+m)\gamma^{-1}\big),\quad m\geq 2.
 \end{align*}
 Note that $\hat h_1(0,\xi)=1$ for $|\xi|\le \varepsilon_1\gamma^{-1}\ln{(\gamma^{-1})}$, so it is easy to check that 
 \begin{align*}
 \hat h_1\big(j\gamma^{-1},(j+m)\gamma^{-1}\big)=&1,\quad j\geq 0,\ m\geq 2,\\
 \hat h_1\big((j-1)\gamma^{-1},j\gamma^{-1}\big)=&1+\hat h_1\big((j-2)\gamma^{-1},(j-2)\gamma^{-1}\big),\\
 \hat h_1\big(j\gamma^{-1},(j+1)\gamma^{-1}\big)=&1+\hat h_1\big((j-1)\gamma^{-1},(j-1)\gamma^{-1}\big)\\
 =&1+\hat h_1\big((j-2)\gamma^{-1},(j-1)\gamma^{-1}\big)+\hat h_1\big((j-2)\gamma^{-1},(j-2)\gamma^{-1}\big)\\
 =&\hat h_1\big((j-2)\gamma^{-1},(j-1)\gamma^{-1}\big)+\hat h_1\big((j-1)\gamma^{-1},j\gamma^{-1}\big),
 \end{align*}
 and thus $\hat h_1(j\gamma^{-1},(j+1)\gamma^{-1})=b_j$ where $\{b_j\}$ is the Fibonacci sequence, which grows exponentially. Indeed, the numerical evidence (see Figure \ref{figure1}) gives the above growth mechanism. 
\begin{figure}
\centering
\includegraphics[height=9.0cm,width=12cm]{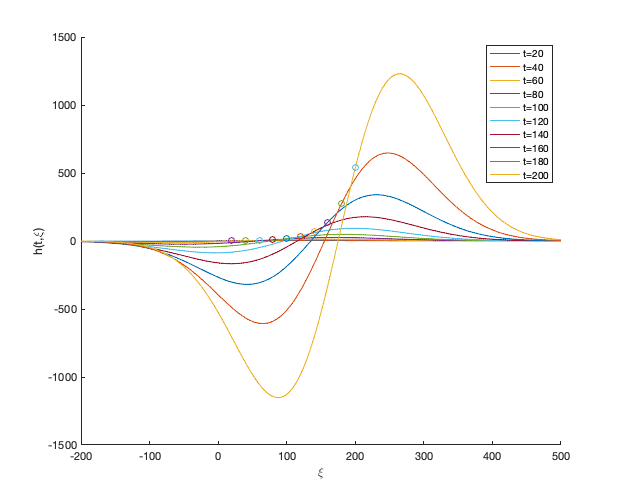}
\caption{Profile of $\hat h_1(t,\xi)$ for viscous problem with $M=5$, $\gamma=0.01$, $\delta_0=\frac{1}{20}$}\label{figure1}
\end{figure}
The formal assumption \eqref{eq:formal} is not very accurate. To estimate $\inf_{|\eta-t|\le C}\hat h_1(t,\eta)$, we will subdivide the time interval $[j\g^{-1},(j+1)\g^{-1}]$ into smaller intervals of size $\g^{-\f23}$, see the definition of $I_m$ and Lemma \ref{lem-evo-h-case2} in Section 3. 
Mathematically, we prove that
\begin{align*}
   \hat h_1(j\gamma^{-1},\xi)|_{\xi\in[(j+1)\gamma^{-1},(j+2)\gamma^{-1}]}\ge b_j,\quad \text{for}\quad 1\leq j\lesssim \varepsilon_1\ln(\g^{-1})
\end{align*}
which describes the exponential growth of the positive parts for each graph shown in the picture. See Section 3 for more details.

Based on the linear instability proved by the dynamical approach, we solve two important questions, namely, the optimal instability of the Couette flow for the viscous problem and the existence of a growing mode for the inviscid problem.   
\subsection{Nonlinear instability of the viscous problem}
In Section 4, we study the nonlinear viscous system \eqref{eq:b-perturbation}, we first get a priori estimates by using the classical energy method via the ghost type weight \cite{Alinhac2001} and bootstrap argument, then prove the solution of system \eqref{eq:b-perturbation} has exponential growth from the initial data 
\begin{equation}
  \omega_{in}(x,y)=\frac{\varepsilon_0\nu^{\frac{2}{3}+\delta_1-\frac{1}{3}\delta_0}}{\sqrt{\varepsilon_1}\sqrt{-\ln{\nu}}}g_{in}(x,y),
\end{equation}
where $g_{in}$ is given in \eqref{eq-ini-g}, the constant $\delta_1$ is small enough determined in \eqref{eq-delta1}, and $\varepsilon_0$ is a positive small constant which is determined in the proof.

\subsection{Inviscid unstable shear flow}
In Section 5, we study the linear instability and study the existence and location of eigenvalues for the Rayleigh operator $\mathcal{R}=b(y)-b''(y)(\pa_y^2-1)^{-1}$ corresponding to mode $1$.  
We prove the existence result by using a contradiction argument. 
We first show that $\mathcal{R}$ has no embedded eigenvalue (see Lemma \ref{lem-no-emb}). 
If $\mathcal{R}$ has no eigenvalues, we prove that for any $\mathfrak{w}\in L^2$ and $t\ge 0$, it holds that 
\begin{equation}\label{eq-up-aa}
\|e^{it\mathcal{R}}\mathfrak{w}\|_{L^2}\le C\|\mathfrak{w}\|_{L^2}
\end{equation}
where $C$ is a constant independent of $\gamma,t$. 

However, by the dynamical approach introduced in Section 3, we prove that 
\begin{equation}\label{eq-low-aa}
\|e^{it\mathcal{R}}\mathfrak{w}\|_{L^2}\ge ce^{c\gamma t}\|\mathfrak{w}\|_{L^2}
\end{equation}
holds for all $t\le \varepsilon_1\gamma^{-1}\ln (\gamma^{-1})$ with a well-chosen data $\mathfrak{w}\in L^2$. Here $c$ is independent of $\gamma, t$. Thus by taking $\gamma>0$ small enough, the estimates of \eqref{eq-up-aa} and \eqref{eq-low-aa} lead to a contradiction, which gives the existence of an eigenvalue. 

To get the precise location of eigenvalues, we study the Rayleigh equation. The result uses energy methods and an ODE argument.

\section{The dynamical approach: lower bound estimate}
In this section, we use the dynamical approach to give lower bound estimates for the solution of the linear system \eqref{eq-g-linear} for both viscous and inviscid case $(\nu=\gamma^{\frac{3}{1-2\delta_0}} \text{ and } \nu=0)$  from the initial data \eqref{eq-ini-g}. {{}Now we introduce the solution operator
\begin{align*}
  S_\nu(t,s): L^2_{x,y}\to L^2_{x,y},\text{ for }\nu\ge0,\ t\ge s\ge0
\end{align*}
which satisfies that $g(t,x,y)=S_\nu(t,s)f(x,y)$ solves the linear  system \eqref{eq-g-linear} with $g(s,x,y)=f(x,y)$. Specifically, for $s=0$, $g(t,x,y)=S_\nu(t,0)f(x,y)$ solves \eqref{eq-g-linear} with initial data $g(0,x,y)=f(x,y)$. We remark that for the inviscid case $\nu=0$, the coefficients in \eqref{eq-g-linear} are independent of time, and $S_0(t,0)$ is a semi-group.} By taking the well-chosen initial data, we prove the following propositions.
\begin{proposition}\label{prop-lower}
There exists $M_0>0$, for any $M>M_0$, and any small $\delta_0>0$, there exist constants $c_0,c_1,\varepsilon_1,\gamma_0>0$ such that for $0<\gamma\le\gamma_0$, $\nu=\gamma^{\frac{3}{1-2\delta_0}}$, and $0\le t\le T=\varepsilon_1\gamma^{-1}(\ln{\nu^{-1}})$, it holds that
   \begin{align*}
     \left\|P_{\pm1}\left({S_\nu(t,0)}g_{in}\right)\right\|_{L^2_{x,y}}\ge c_0e^{c_1\gamma t}\|g_{in}\|_{L^2_{x,y}},
   \end{align*}
   and in particular
   \begin{align*}
     \left\|P_{\pm1}\left({S_\nu(T,0)}g_{in}\right)\right\|_{L^2_{x,y}}\ge c_0\nu^{-c_1\varepsilon_1}\|g_{in}\|_{L^2_{x,y}},
   \end{align*}
   where $g_{in}(x,y)$ is given in \eqref{eq-ini-g}.
\end{proposition}
\begin{proposition}\label{prop-lower-Euler}
There exists $M_0>0$ and for any $M>M_0$, there exist constants $c_0,c_1,\varepsilon_1,\gamma_0>0$ such that for $0<\gamma\le\gamma_0$, and $0\le t\le T=\varepsilon_1\gamma^{-1}(\ln{\gamma^{-1}})$, it holds that
   \begin{align*}
     \left\|P_{\pm1}\left({S_0(t,0)}g_{in}\right)\right\|_{L^2_{x,y}}\ge c_0e^{c_1\gamma t}\|g_{in}\|_{L^2_{x,y}},
   \end{align*}
   and in particular
   \begin{align*}
     \left\|P_{\pm1}\left({S_0(T,0)}g_{in}\right)\right\|_{L^2_{x,y}}\ge c_0\gamma^{-c_1\varepsilon_1}\|g_{in}\|_{L^2_{x,y}},
   \end{align*}
   where $g_{in}(x,y)$ is given in \eqref{eq-ini-g}.
\end{proposition}

Recall that $h(t,z,y)$ is the solution of \eqref{eq-lin-h} satisfying
\begin{align*}
  h(t,z,y)=h(t,x-ty,y)=g(t,x,y)={S_\nu(t,0)}g_{in}(x,y),\text{ for }\nu\ge0.
\end{align*}
Accordingly,
\begin{align*}
  2\pi\|\hat h(t)\|_{L^2_{k,\xi}}=\| h(t)\|_{L^2_{z,y}}=\| g(t)\|_{L^2_{x,y}}=\|{S_\nu(t,0)}g_{in}\|_{L^2_{x,y}},\text{ for }\nu\ge0.
\end{align*}
It is sufficient to prove Propositions \ref{prop-lower} and \ref{prop-lower-Euler} by giving the lower bound estimate of $\|\hat h(t)\|_{L^2_{k,\xi}}$. Next, we focus on the system \eqref{eq-h-ft-ori}.
\subsection{Upper bound estimate in $L^\infty_{\xi}$}
We first give the upper bound estimate of $\|\hat h_k(t,\xi)\|_{L^\infty_\xi}$.
\begin{lemma}\label{lem-h-up-inf}
  Let $\hat h_k(t,\xi)$ be the solution to \eqref{eq-h-ft-ori}. For any $\nu\ge0$, we have the following upper bound estimate:
  \begin{align}\label{ineq-h-inf}
    \|\hat h_k(t,\xi)\|_{L^\infty_\xi}\le e^{M\pi\gamma t}\|\hat h_k(0,\xi)\|_{L^\infty_\xi}.
\end{align}
\end{lemma}
\begin{proof}
From \eqref{eq-h-ft-ori}, we deduce that
\begin{align*}
  &\left|\pa_t\hat h_k(t,\xi)+\nu\big(k^2+(\xi-kt)^2\big)\hat h_k(t,\xi)\right|\\
  \le& \int_{\mathbb{R}}M\gamma^2|\xi-\eta|e^{-(\nu t+\frac{\gamma^2}{4})|\xi-\eta|^2}\frac{k\|\hat h_k(t,\xi)\|_{L^\infty_\xi}}{(\eta-kt)^2+k^2}d\eta\\
  \le&M\|\hat h_k(t,\xi)\|_{L^\infty_\xi}\int_{\mathbb{R}}\frac{k\sup_\eta\left|\gamma^2|\xi-\eta|e^{-(\nu t+\frac{\gamma^2}{4})|\xi-\eta|^2}\right|}{(\eta-kt)^2+k^2}d\eta\le M\pi\gamma\|\hat h_k(t,\xi)\|_{L^\infty_\xi}.
\end{align*}
Here we use the fact that
\begin{align}\label{eq-est-ex-max}
  \sup_\eta\left|\gamma^2|\xi-\eta|e^{-(\nu t+\frac{\gamma^2}{4})|\xi-\eta|^2}\right|\le\frac{\gamma^2}{\sqrt{2e(\nu t+\gamma^2/4)}}\le \gamma.
\end{align}
Then by using Gronwall's inequality, we have
\begin{equation}\label{eq-est-hk}
  \begin{aligned}    
      &\left|\hat h_k(t,\xi)\right|\\
      \le& e^{-\nu\big(k^2t+\xi^2t-k\xi t^2+\frac{k^2t^3}{3}\big)}\Big(\left|\hat h_k(0,\xi)\right|+\int^t_0e^{\nu\big(k^2s+\xi^2s-k\xi s^2+\frac{k^2s^3}{3}\big)}M\pi\gamma \|\hat h_k(s,\xi)\|_{L^\infty_\xi}ds\Big)\\
  =&e^{-\nu k^2\big(t+t(\frac{t}{2}-\frac{\xi}{k})^2+\frac{t^3}{12}\big)}\left|\hat h_k(0,\xi)\right|\\
  &+\int^t_0e^{-\nu k^2\big((t-s)+(t-s)(\frac{\xi}{k}-\frac{(t+s)}{2})^2+\frac{(t-s)^3}{12}\big)}M\pi\gamma \|\hat h_k(s,\xi)\|_{L^\infty_\xi}ds\\
  \leq&\left|\hat h_k(0,\xi)\right|+\int^t_0M\pi\gamma \|\hat h_k(s,\xi)\|_{L^\infty_\xi}ds,
  \end{aligned}
\end{equation}
which gives $\|\hat h_k(t,\xi)\|_{L^\infty_\xi}\le e^{M\pi\gamma t}\|\hat h_k(0,\xi)\|_{L^\infty_\xi}$. 
\end{proof}

\subsection{Lower bound estimate for the viscous problem}
In this subsection, we use the dynamical approach to study the lower bound for the viscous problem ($\nu=\gamma^{\frac{3}{1-2\delta_0}}$). Note that the initial data $h_{in}(z,y)$ (see \eqref{eq-h-ini0}) has only $\pm1$ modes with $\hat h_{in,1}(\xi)=\hat h_{in,-1}(\xi)$. So the solution to the linear system satisfies $\hat h_k(t,\xi)\equiv0$ for $k\neq\pm1$ and  $\hat h_{-1}(t,-\xi)=\hat h_1(t,\xi)$ for all $t\ge0$ and $\xi\in \mathbb R$. Therefore, we focus on the $k=1$ mode, and study the following system:
\begin{equation}\label{eq-h-ft}
  \left\{
    \begin{array}{l}
      \partial_{t} \hat h(t,\xi)=\int_{\mathbb{R}}M\gamma^2(\xi-\eta)e^{-(\nu t+\frac{\gamma^2}{4})|\xi-\eta|^2}\frac{\hat h(t,\eta)}{(\eta-t)^2+1} d\eta-\nu\big(1+(\xi-t)^2\big)\hat h(t,\xi) ,\\
      \hat h(0,\xi)=\hat {\mathfrak h}_{in}(\xi)\ge0.
    \end{array}
  \right.
\end{equation}
For simplicity of notation, here we use $\hat h(t,\xi)$ instead of $\hat h_1(t,\xi)$. 

Analyzing the evolution from the initial time $t=0$, we have the following observations.

The kernel in \eqref{eq-h-ft} reaches its maximum $\frac{\sqrt2M\gamma^2}{\sqrt e\sqrt{4\nu t+\gamma^2}}\approx \gamma$ at $\eta= t$, $\eta-t=\frac{\sqrt2}{\sqrt{4\nu t+\gamma^2}}\approx\gamma^{-1}$, and satisfies
\begin{align*}
  M\gamma^2(\xi-\eta)e^{-(\nu t+\frac{\gamma^2}{4})|\xi-\eta|^2}\frac{1}{(\eta-t)^2+1}\ge c\gamma \text{ for }\xi\in[t+c\gamma^{-1},t+C\gamma^{-1}],\ \eta\in[t-C,t+C].
\end{align*}
Roughly speaking, for fixed $t\ge0$, we call $[t-C,t+C]$ the excitation region, $[t+c\gamma^{-1},t+C\gamma^{-1}]$ the region of growth, and $[t+C\gamma^{-1},\varepsilon_1\gamma^{-1}\ln (\gamma^{-1})]$ the waiting region, see Figure \ref{figure2}. One can see that, at $t=0$, the value of the integral in \eqref{eq-h-ft} is bigger than $c\gamma $ for $\xi$ in the region of growth.  Meanwhile, for such $\xi$, as $\nu=\gamma^{\frac{3}{1-2\delta_0}}$, the dissipation effect is very weak. As a result, $\hat h(0,\eta)$ with $\eta$ in the excitation region will excite $\hat h(0+,\xi)$ for $\xi$ in the region of growth. For $\xi$ in the waiting region, both the integral and the dissipation term are extremely small, which leads to the fact that $\hat h(0+,\xi)$ will stay close to $1$.  The region of growth is far in front of the excitation region, and these two regions move in the same direction as time changes. For each fixed frequency $\eta>C\gamma^{-1}$, it first belongs to the waiting region, then enters into the region of growth, and at last, enters into the excitation region. In a word, $\hat h(t,\eta)$ will first be excited and then excites the one with higher frequency in the new region of growth, which shows that the function $\hat h(t,t)$ is increasing with respect to $t$. In the end, a cascade effect generated from the huge scale of the region of growth causes exponential growth.
\begin{figure}[H]
\centering
\includegraphics[height=6.0cm,width=14cm]{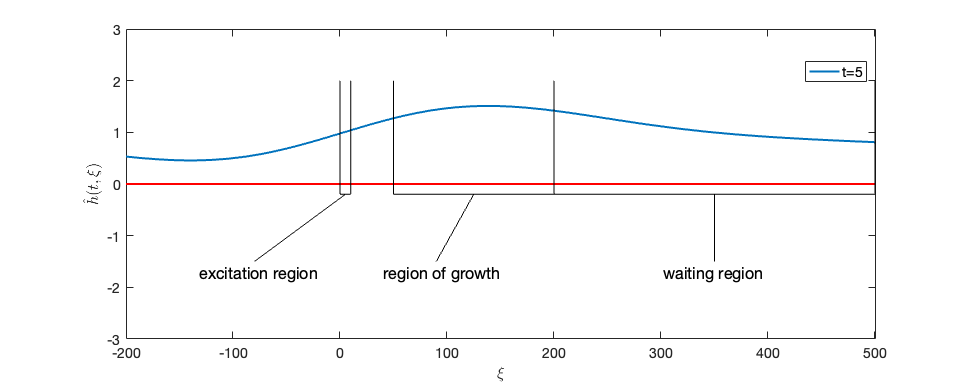}  
\caption{excitation region, region of growth and waiting region}\label{figure2}
\end{figure}
To clarify the growth mechanism, we divide the time-frequency into small intervals and study the evolution of $\hat h(t,\xi)$ during each small time interval. Let 
\begin{align*}
  \mathbb R^+\times\mathbb R=\bigcup_{m\in\mathbb Z^+,n\in \mathbb Z} I_m\times I_n,
\end{align*}
where $I_n=(\frac{n}{N}\gamma^{-1},\frac{n+1}{N}\gamma^{-1}]$, and $N=\lfloor\gamma^{-\frac{1}{3}}\rfloor$. Here we denote by $\lfloor a\rfloor$ the biggest integer that is not greater than $a$. We also define $T_m=\frac{m}{N}\gamma^{-1}$ which are the endpoints for the time intervals. Now we give precise definitions of the excitation region, the region of growth, and the waiting region.   For $t\in I_m$, we divide into 4 cases depending on the frequency $\xi\in \mathbb{R}=\mathfrak{I}^{r}_m\cup \mathfrak{I}^{e}_m\cup \mathfrak{I}^{g}_m\cup \mathfrak{I}^{w}_m$, where 
\begin{itemize}
\item[Case 1.] $\xi\in \mathfrak{I}^{r}_m=\left(\bigcup_{n\le m-3}I_n\right)\cup\left(\bigcup_{n>2\varepsilon_1\ln{(\gamma^{-1})}N}I_n\right)$;
\item[Case 2.] $\xi\in \mathfrak{I}^{e}_m=\bigcup_{|n-m|\le 2}I_n$;
\item[Case 3.] $\xi\in \mathfrak{I}^{g}_m=\bigcup_{m+3\le n\le m+2+4N}I_n$;
\item[Case 4.] $\xi\in \mathfrak{I}^{w}_m=\bigcup_{m+3+4N\le n\le 2\varepsilon_1\ln{(\gamma^{-1})}N}I_n$.
\end{itemize}
For Case 1, $\xi\in \mathfrak{I}^{r}_m$ is not in the excitation region nor the region of growth, we only need the upper bound estimate given in Remark \ref{rmk-basic}. For Case 2, the interval $\mathfrak{I}^{e}_m$ contains the excitation region, we prove that $\hat h(t,\xi)$ barely change during the time interval $I_m$, see Lemma \ref{lem-evo-h-case2}. For Case 3, the interval $\mathfrak{I}^{g}_m$ is the region of growth, $\hat h(t,\xi)$  keeps growing during the time interval $I_m$, see Lemma \ref{lem-evo-h-case3}.  For Case 4, $\mathfrak{I}^{w}_m$ is basically the waiting region, we prove that $\hat h(t,\xi)\ge \frac{19}{20}$, see Lemma \ref{lem-evo-h-case4}. The growth accumulated in Case 3 will eventually produce exponential growth.

\begin{remark}\label{rmk-basic}
   It holds for $t\ge0$ that
   \begin{align}\label{eq-h-up-basic}
  \|\hat h(t)\|_{L^\infty_\xi}\le e^{M\pi\gamma t}\|\hat {\mathfrak h}_{in}\|_{L^\infty_\xi}=e^{M\pi\gamma t}.
\end{align}
 \end{remark} 
 \begin{proof}
   The estimate \eqref{eq-h-up-basic} follows directly from Lemma \ref{lem-h-up-inf} by taking the well-chosen initial data \eqref{eq-h-ini0}.
 \end{proof}

Now, let us focus on Case 2.
\begin{lemma}\label{lem-evo-h-case2}
   Let $\hat h$ be the solution of \eqref{eq-h-ft}. For $t\in I_m$, $\xi\in I_n$, with $m\le \frac{3}{2}\varepsilon_1\ln(\gamma^{-1}) N-1$ and $\ |n-m|\le2$, it holds that
   \begin{align*}
     \hat h(T_{m},\xi)-\frac{6M\pi}{N^2}e^{\frac{M\pi(m+1)}{N}}\le\hat h(t,\xi)\le\hat h(T_{m},\xi)+\frac{6M\pi}{N^2}e^{\frac{M\pi(m+1)}{N}}.
   \end{align*}
\end{lemma}
\begin{proof}
It follows from \eqref{eq-h-ft} that
\begin{equation}\label{eq-h-divide}
  \begin{aligned}    
      \pa_t\hat h(t,\xi)
 % =&\int_{\mathbb{R}}M\gamma^2(\xi-\eta)e^{-(\nu t+\frac{\gamma^2}{4})|\xi-\eta|^2}\frac{1}{(\eta-t)^2+1}\hat h(t,\eta)d\eta-\nu\big(1+(\xi-t)^2\big)\hat h(t,\xi)\\
  =&\int_{\cup_{i\in \mathbb Z,|i-m|\le1}I_{i}}M\gamma^2(\xi-\eta)e^{-(\nu t+\frac{\gamma^2}{4})|\xi-\eta|^2}\frac{1}{(\eta-t)^2+1} \hat h(t,\eta)d\eta\\
  &+\int_{\cup_{i\in \mathbb Z,|i-m|>1}I_{i}}M\gamma^2(\xi-\eta)e^{-(\nu t+\frac{\gamma^2}{4})|\xi-\eta|^2}\frac{1}{(\eta-t)^2+1} \hat h(t,\eta)d\eta\\
  &-\nu\big(1+(\xi-t)^2\big)\hat h(t,\xi)\\
  \eqdef&\mathcal A_m+\mathcal B_m+\mathcal C.
  \end{aligned}
\end{equation}
As $\xi\in I_n,|n-m|\le2$, and $\eta\in\cup_{i\in \mathbb Z,|i-m|\le1}I_{i}$, it holds that $|\xi-\eta|\le \frac{4}{N}\gamma^{-1}$. Then from \eqref{eq-h-up-basic}, we have
\begin{align*}
  |\mathcal A_m|\le M\gamma^2\frac{4}{N}\gamma^{-1}\int_{\mathbb R} \frac{1}{(\eta-t)^2+1}d \eta \|\hat h(t,\xi)\|_{L^\infty_\xi}\le \frac{4M\pi\gamma}{N}e^{M\pi\gamma t}.
\end{align*}
Then by using \eqref{eq-est-ex-max}, we have
\begin{align*}
  |\mathcal B_m|\le&M\gamma \int_{|\eta-t|\ge \frac{1}{N}\gamma^{-1}}\frac{1}{(\eta-t)^2+1} d \eta\|\hat h(t,\xi)\|_{L^\infty_\xi}\le 2MN\gamma^2e^{M\pi\gamma t},
\end{align*}
As $t\in I_m$, and $\xi\in I_n,|n-m|\le2$, we have $|\xi-t|\le \frac{3}{N}\gamma^{-1}$, then we have
\begin{align*}
  |\mathcal C|\le& \frac{10\nu\gamma^{-2}}{N^2}\|\hat h(t,\xi)\|_{L^\infty_\xi}\le \frac{10\nu\gamma^{-2}}{N^2}e^{M\pi\gamma t}.
\end{align*}
Recalling that $N=\lfloor\gamma^{-\frac{1}{3}}\rfloor$, $\nu=\gamma^{\frac{3}{1-2\delta_0}}$, and $\gamma>0$ small enough, we have
\begin{align*}
  \|\pa_t\hat h(t,\xi)\chi_{\xi\in \cup_{i\in \mathbb Z,|i|\le2}I_{i}}\|_{L^\infty_\xi}\le |\mathcal A_m|+|\mathcal B_m|+|\mathcal C|\le \frac{5}{4}|\mathcal A_m|\le\frac{5M\pi\gamma}{N}e^{M\pi\gamma t}.
\end{align*}
Therefore, for $\xi\in I_n, |n-m|\le2$, it holds that
\begin{align*}
  \hat h(T_{m},\xi)-\frac{6M\pi}{N^2}e^{\frac{M\pi(m+1)}{N}}\le\hat h(t,\xi)\le\hat h(T_{m},\xi)+\frac{6M\pi}{N^2}e^{\frac{M\pi(m+1)}{N}}.
\end{align*}
\end{proof}
Next, we show the main growth and study Case 3.
\begin{lemma}\label{lem-evo-h-case3}
   Let $\hat h$ be the solution of \eqref{eq-h-ft}. There exists $\varepsilon_1>0$ such that, for $t\in I_m=[T_{m},T_{m+1}]$, $\xi\in I_n$ with $m\le \frac{3}{2}\varepsilon_1\ln(\gamma^{-1}) N-1$, and $m+3\le n\le m+2+4N$, if it holds that
   \begin{align}\label{eq-h-assump}
     \inf_{\tau\in I_m}\ \inf_{\eta\in\cup_{ |i-m|\le1}I_{i}}\hat h(\tau,\eta)\ge \frac{9}{10},
   \end{align}
   we have 
\begin{align}\label{eq-grow}
  \hat h(t,\xi)\ge \hat h(T_{m},\xi)+(t-T_{m})\frac{4\gamma M(n-m-2)\pi  }{9N}e^{-\frac{(n-m-2)^2}{2N^2}}\inf_{\tau\in I_m}\inf_{\eta\in\cup_{ |i-m|\le1}I_{i}}\hat h(\tau,\eta),
\end{align}
and in particular at the endpoint,
   \begin{align*}
  \hat h(T_{m+1},\xi)\ge \hat h(T_{m},\xi)+\frac{4M(n-m-2)\pi}{9N^2}e^{-\frac{(n-m-2)^2}{2N^2}} \inf_{\tau\in I_m}\inf_{\eta\in\cup_{ |i-m|\le1}I_{i}}\hat h(\tau,\eta).
\end{align*}
\end{lemma}
\begin{proof}
  Similar to \eqref{eq-h-divide}, we write
  \begin{align*}
      \pa_t\hat h(t,\xi)=\mathcal A_m+\mathcal B_m+\mathcal C,
  \end{align*}
 where $\mathcal A_m$ describes the effect from the excitation region. 
  For $\xi\in I_n$, $n-m\ge3$, it holds that $(\xi-\eta)\ge \frac{n-m-2}{N}\gamma^{-1}$ with $\eta\in \cup_{ |i-m|\le1}I_{i}$. Then we can see from the assumption \eqref{eq-h-assump} that
\begin{align*}
  \mathcal A_m\ge& \frac{M\gamma(n-m-2)}{N}e^{-(\nu t+\frac{\gamma^2}{4})\frac{(n-m-2)^2\gamma^{-2}}{N^2}}\int_{|\eta-t|\le \frac{\gamma^{-1}}{N}} \frac{1}{(\eta-t)^2+1}d \eta\inf_{\tau\in I_m}\inf_{\eta\in\cup_{ |i-m|\le1}I_{i}}\hat h(\tau,\eta)\\
  \ge&\frac{M\gamma(n-m-2)}{N}e^{-\frac{(n-m-2)^2}{2N^2}} \frac{8\pi}{9}\inf_{\tau\in I_m}\inf_{\eta\in\cup_{ |i-m|\le1}I_{i}}\hat h(\tau,\eta)\\
  \ge&\frac{4M\gamma(n-m-2)\pi}{5N}e^{-\frac{(n-m-2)^2}{2N^2}}.
\end{align*}
As $t\in I_m$, $\xi\in I_n$ with $m+3\le n\le m+2+4N$ and $m\le \frac{3}{2}\varepsilon_1\ln(\gamma^{-1}) N-1$, it is clear that $t\le \frac{3}{2}\varepsilon_1\gamma^{-1}\ln(\gamma^{-1})$. One can easily check that
\begin{align*}
  \frac{1}{4}\mathcal A_m\ge&\frac{M\gamma(n-m-2)\pi}{5N}e^{-8}\ge\frac{M\gamma^{\frac{4}{3}}(n-m-2)\pi}{5}e^{-8},\\
  |\mathcal B_m|\le& 2MN\gamma^2e^{M\pi\gamma t}\le 2M\gamma^{\frac{5}{3}-\frac{3}{2}\varepsilon_1 M\pi},\\
  |\mathcal C|\le& \frac{\nu(n-m+2)^2\gamma^{-2}}{N^2}e^{M\pi\gamma t}\le \gamma^{\frac{5}{3}+6\delta_0-\frac{3}{2}\varepsilon_1 M\pi}(n-m+2)^2\le5\gamma^{\frac{4}{3}+6\delta_0-\frac{3}{2}\varepsilon_1 M\pi}(n-m+2).
\end{align*}
Then, by taking 
\begin{equation}\label{eq-epsilon1}
  \varepsilon_1=\frac{2\delta_0}{M\pi},
\end{equation}
we get
\begin{align*}
  |\mathcal B_m|\le2M\gamma^{\frac{5}{3}-3\delta_0}\le\frac{1}{4}\mathcal A_m,\quad |\mathcal C|\le5\gamma^{\frac{4}{3}+3\delta_0}(n-m+2)\le\frac{1}{4}\mathcal A_m
\end{align*}
for $\delta_0\le \frac{1}{10}$ and $\gamma$ small enough.

Therefore, we get that 
\begin{align}\label{eq-ht-deter}
  \pa_t\hat h(t,\xi)\ge \frac{1}{2}\mathcal A_m\ge \frac{4M\gamma(n-m-2)\pi}{9N}e^{-\frac{(n-m-2)^2}{2N^2}} \inf_{\tau\in I_m}\inf_{\eta\in\cup_{ |i-m|\le1}I_{i}}\hat h(\tau,\eta),
\end{align}
and 
\begin{align*}
  \hat h(t,\xi)\ge \hat h(T_{m},\xi)+(t-T_{m})\frac{4\gamma M(n-m-2)\pi  }{9N}e^{-\frac{(n-m-2)^2}{2N^2}}\inf_{\tau\in I_m}\inf_{\eta\in\cup_{ |i-m|\le1}I_{i}}\hat h(\tau,\eta),
\end{align*}
which gives the lemma. 
\end{proof}

For Case 4, we study the time evolution in the waiting region. The dissipation effect will be obvious, that is to say, $|\mathcal C|$ will be comparable to $\mathcal A_m$. So we do not expect $\hat h(t,\xi)$ to grow for $\xi\in \mathfrak{I}^{w}_m$. We give a lower-bound estimate.
\begin{lemma}\label{lem-evo-h-case4}
   Let $\hat h$ be the solution of \eqref{eq-h-ft}. For $t\in I_m$ and $\xi\in I_n$ with $m\le \frac{3}{2}\varepsilon_1\ln(\gamma^{-1}) N-1$, $m+3+4N\le n\le 2\varepsilon_1\ln{(\gamma^{-1})}N$, and $\varepsilon_1$ given in Lemma \ref{lem-evo-h-case3}, if it holds that 
   \begin{align}\label{eq-h-assump2}
     \inf_{0\le m'\le m}\ \inf_{\tau\in I_m}\ \inf_{\eta\in\cup_{ |i-m|\le1}I_{i}}\hat h(\tau,\eta)\ge 0,
   \end{align}
    we have
   \begin{align*}
  \hat h(t,\xi)\ge \frac{19}{20}.
\end{align*}
\end{lemma}
\begin{proof}
Under the assumption \eqref{eq-h-assump2}, it holds that
\begin{align*}
  \mathcal A_{m'}\ge0 \text{ for }0\le m'\le m.
\end{align*}
Then, by $m+3+4N\le n\le 2\varepsilon_1\ln{(\gamma^{-1})}N$, we have for $t\in \cup_{m'=0}^mI_{m'}$, $\xi\in I_n$ that
\begin{align*}
      &\partial_{t} \hat h(t,\xi)\ge-4\varepsilon_1^2\gamma^{1+6\delta_0}(\ln{\gamma})^2\hat h(t,\xi)-2MN\gamma^2e^{M\pi\gamma t}.
\end{align*}
Using Gronwall's inequality, we deduce that
\begin{align}\label{eq-h-dis}
  \hat h(t,\xi)\ge e^{-4\varepsilon_1^2\gamma^{1+6\delta_0}(\ln{\gamma})^2t}\Big(1-\int^t_02MN\gamma^2e^{(M\pi\gamma+4\varepsilon_1^2\gamma^{1+6\delta_0}(\ln{\gamma})^2) s}ds\Big).
\end{align}
With $\gamma$ small enough, it follows that
\begin{align*}
  \inf_{t\in I_m}\hat h(t,\xi)\ge e^{-6\varepsilon_1^3\gamma^{6\delta_0}(\ln{\gamma^{-1}})^3}\big(1-\frac{\gamma^{\frac{2}{3}-3\delta_0}}{\pi}\big)\ge \frac{19}{20}.
\end{align*}
\end{proof}
The assumptions \eqref{eq-h-assump} and \eqref{eq-h-assump2} hold naturally for the linear system with our well-chosen initial data, see the following lemma.
\begin{lemma}\label{lem-low-9/10}
With our well-chosen initial data, assumption \eqref{eq-h-assump} in Lemma \ref{lem-evo-h-case3} and assumption \eqref{eq-h-assump2} in Lemma \ref{lem-evo-h-case4} are automatically satisfied.
\end{lemma}
\begin{proof}
 From Lemma \ref{lem-evo-h-case2} and the fact $\gamma$ small enough, we have for $m=0,1$ that
\begin{align}\label{eq-assum-step}
  \inf_{\tau\in I_m}\inf_{\eta\in I_{m+3-j}}\hat h(\tau,\eta)\ge \frac{19}{20}-j\frac{6M\pi}{N^2}\gamma^{-3\delta_0}\ge \frac{9}{10},\quad j=1,2,3,4.
\end{align}
Next, we use mathematical induction to complete the proof. We assume that \eqref{eq-assum-step} holds for $0\le m\le K-1$ with $K\le\frac{3}{2}\varepsilon_1 \ln(\gamma^{-1}) N-1$, then assumption \eqref{eq-h-assump} and assumption \eqref{eq-h-assump2} holds for $0\le m\le K-1$. We will show that \eqref{eq-assum-step} holds for $m=K$. We only need to prove that 
\begin{align*}
  \inf_{\tau\in I_K}\inf_{\eta\in I_{K+2}}\hat h(\tau,\eta)\ge \frac{19}{20}-\frac{6M\pi}{N^2}\gamma^{-3\delta_0}.
\end{align*}
If $K\le 4N$, then the frequency $\xi\in I_{K+2}$ stay in the region of growth for $t\in \cup_{m=0}^{K-1}I_m=[0,T_K]$, then we deduce from Lemma \ref{lem-evo-h-case3} that
\begin{align*}
  \hat h(T_{K},\xi)\ge 1,\text{ for } \xi\in I_{K+2},\ K\le2\sqrt{\delta_0}\sqrt{-\ln{\gamma}}N.
\end{align*}
If $4N+1\le K\le \frac{3}{2}\varepsilon_1 \ln(\gamma^{-1}) N-1$, the frequency $\xi\in I_{K+2}$ will fist belong to the waiting region for $t\in \cup_{m=0}^{K-4N-1}I_m=[0,T_{K-4N}]$. We get from Lemma \ref{lem-evo-h-case4} that
\begin{align*}
  \hat h(T_{K-4N},\xi)\ge \frac{19}{20},\text{ for } \xi\in I_{K+2},\ 4N+1\le K\le \frac{3}{2}\varepsilon_1 \ln(\gamma^{-1}) N-1.
\end{align*}
Then such $\xi$ belong to the region of growth for $t\in \cup_{m=K-4N}^{K-1}I_m=[T_{K-4N},T_{K}]$, and it holds from Lemma \ref{lem-evo-h-case3} that
\begin{align*}
  \hat h(T_{K},\xi)\ge \frac{19}{20},\text{ for } \xi\in I_{K+2},\ 4N+1\le K\le \frac{3}{2}\varepsilon_1 \ln(\gamma^{-1}) N-1.
\end{align*}
Thus, for all $1\le K\le\frac{3}{2}\varepsilon_1 \ln(\gamma^{-1}) N-1$, we get from Lemma \ref{lem-evo-h-case2} that
\begin{align*}
  \inf_{\tau\in I_K}\inf_{\eta\in I_{K+2}}\hat h(\tau,\eta)\ge \frac{19}{20}-\frac{6M\pi}{N^2}\gamma^{-3\delta_0},
\end{align*}
and then \eqref{eq-assum-step} holds for $K$. 

Hence, by the principle of mathematical induction, \eqref{eq-assum-step} holds for all $m\le\frac{3}{2}\varepsilon_1 \ln(\gamma^{-1}) N-1$. The desired result follows immediately.
\end{proof}
Lemma \ref{lem-low-9/10} together with Lemma \ref{lem-evo-h-case2}-\ref{lem-evo-h-case4} give the evolution of $\hat h(t,\xi)$ in the excitation region, region of growth, and waiting region respectively. In the proof of Lemma \ref{lem-low-9/10}, we only use the fact that $\hat h(t,\xi)$ is growing in the region of growth, but do not care about the amount of the growth. Next, we give the proof of Proposition \ref{prop-lower}, and show how the growth accumulated in the region of growth eventually produces exponential growth.
\begin{proof}[Proof of Proposition \ref{prop-lower}]
We prove the following statement by mathematical induction:\\
{\bf Let $\{a_j\}$ and $\{b_j\}$ satisfy $a_0=b_0=1$, and $j\geq 1$,
\begin{align*}
  a_j=a_{j-1}+b_{j-1},\quad b_j=a_{j-1}+1.
\end{align*}
Then, for $1\le j\le \frac{5}{4}\varepsilon_1 \ln(\gamma^{-1})+2$,
\beq\label{eq:mathinduc}
\begin{aligned}
  &\frac{10}{9}\mathcal E_j\ge a_{j-1},\quad 
  \hat h(j\gamma^{-1},\xi)|_{\xi\in [(j+2)\gamma^{-1},(j+3)\gamma^{-1}]}\ge 1,\\
  &\hat h(j\gamma^{-1},\xi)|_{\xi\in[j\gamma^{-1},(j+1)\gamma^{-1}]}\ge a_j,\quad
  \hat h(j\gamma^{-1},\xi)|_{\xi\in[(j+1)\gamma^{-1},(j+2)\gamma^{-1}]}\ge b_j,
\end{aligned}
\eeq
where
  \begin{align*}
    \mathcal E_j\eqdef\inf_{(j-1)N\le m\le jN-1}\left(\inf_{\tau\in I_m}\left(\inf_{\eta\in\cup_{ |i-m|\le1}I_{i}}\hat h(\tau,\eta)\right)\right).
  \end{align*}
  }

\no\underline{Base case ($j=1,2$):}
  From Lemma \ref{lem-low-9/10} we can see that
  \begin{align*}
    \mathcal E_1\ge\frac{9}{10}.
  \end{align*}
  Then from Lemma \ref{lem-evo-h-case3}, we know for $\xi\in I_n$ with $N+2\le n\le 2+4N$ that
  \begin{equation}\label{eq-step-h2}
    \begin{aligned}    
    \hat h(\gamma^{-1},\xi)\ge&\hat h(0,\xi)+\sum_{m=0}^{N-1}\frac{4M(n-m-2)\pi}{9N^2}e^{-\frac{(n-m-2)^2}{2N^2}}\mathcal E_1\\
    \ge&\hat h(0,\xi)+\sum_{m=0}^{N-1}\frac{4M(n-m-2)\pi}{9N^2}e^{-\frac{(n-m-2)^2}{2N^2}}\mathcal E_1\\
    \ge&\hat h(0,\xi)+\int_{n-N}^{n}\frac{4M\pi}{9N^2}\sigma e^{-\frac{\sigma^2}{2N^2}}d\sigma\mathcal E_1\\
    \ge&\hat h(0,\xi)+\frac{4M\pi}{9}\big(e^{-\frac{(n-N)^2}{2N^2}}-e^{-\frac{n^2}{2N^2}}\big)\mathcal E_1.      
    \end{aligned}
  \end{equation}
Similarly, from Lemma \ref{lem-evo-h-case2} and Lemma \ref{lem-low-9/10}, for  $\xi\in I_n$ with $n=N,N+1$, it holds that
\begin{align*}
  \hat h(\gamma^{-1},\xi)\ge\hat h(0,\xi)+\frac{4M\pi}{9}\big(1-e^{-\frac{(N-2)^2}{2N^2}}\big)\mathcal E_1-\frac{12M\pi\gamma^{-3\delta_0}}{N^2}.
\end{align*}
Then, for $N\le n\le 3N-1$, it holds that
\begin{align*}
  \hat h(\gamma^{-1},\xi)|_{\xi\in I_n}\ge\hat h(0,\xi)+\frac{4M\pi}{9}\big(e^{-\frac{(3N-N)^2}{2N^2}}-e^{-\frac{(3N)^2}{2N^2}}\big)\mathcal E_1=1+\frac{4M\pi}{9}\big(e^{-2}-e^{-\frac{9}{2}}\big)\mathcal E_1.
\end{align*}
Let 
\begin{equation}\label{eq-M}
  M\ge\frac{5}{2\pi(e^{-2}-e^{-\frac{9}{2}})},
\end{equation}
then we have
\begin{align*}
  \hat h(\gamma^{-1},\xi)\ge 1+\frac{10}{9}\mathcal E_1\ge2,\text{ for }\xi\in[\gamma^{-1},3\gamma^{-1}].
\end{align*}
Therefore, using the same argument to Lemma \ref{lem-low-9/10}, we can see that
\begin{align}\label{eq-step-E2}
  \mathcal E_2\ge \hat h(\gamma^{-1},\eta)|_{\eta\in[\gamma^{-1},2\gamma^{-1}]}-\frac{1}{10}\ge \frac{19}{10},
\end{align}
which will excite $\hat h(2\gamma^{-1},\xi)|_{\xi\in [2\gamma^{-1},4\gamma^{-1}]}$ to
\begin{align*}
  \hat h(2\gamma^{-1},\xi)|_{\xi\in [2\gamma^{-1},3\gamma^{-1}]}\ge& \hat h(\gamma^{-1},\xi)|_{\xi\in [2\gamma^{-1},3\gamma^{-1}]}+\frac{10}{9}\mathcal E_2\ge2+\frac{10}{9}\frac{19}{10}\ge4,
\end{align*}
and
\begin{align*}
  \hat h(2\gamma^{-1},\xi)|_{\xi\in [3\gamma^{-1},4\gamma^{-1}]}\ge& \hat h(\gamma^{-1},\xi)|_{\xi\in [3\gamma^{-1},4\gamma^{-1}]}+\frac{10}{9}\mathcal E_2\ge1+\frac{10}{9}\frac{19}{10}\ge3.
\end{align*}
We also have 
\begin{align*}
  \hat h(2\gamma^{-1},\xi)|_{\xi\in [4\gamma^{-1},5\gamma^{-1}]}\ge& 1.
\end{align*}

\no\underline{Inductive step: } Suppose that the  statement \eqref{eq:mathinduc} holds for $1,2,\dots,j$, then let us prove that it holds for $j+1$: 

We have for the same reason as \eqref{eq-step-E2} that
\begin{align*}
  \frac{10}{9}\mathcal E_{j+1}\ge \frac{10}{9}\Big(\hat h\big(j\gamma^{-1},\eta\big)|_{\eta\in[j\gamma^{-1},(j+1)\gamma^{-1}]}-\frac{1}{10}\Big)\ge a_{j}+\frac{a_{j}-1}{9}\ge a_{j}.
\end{align*}
We also have that
\begin{align*}
  &\frac{4M\pi}{9}\big(e^{-\frac{(4N-N)^2}{2N^2}}-e^{-\frac{(4N)^2}{2N^2}}\big)\mathcal E_j\ge\frac{4M\pi}{9}\frac{9 a_j}{10}\big(e^{-\frac{9}{2}}-e^{-8}\big)\ge \frac{1}{20}.
\end{align*}
Therefore, by Lemma \ref{lem-evo-h-case4}, we deduce for $2\le j\le \frac{5}{4}\varepsilon_1 \ln(\gamma^{-1})+1$ that
\begin{align*}
  \hat h(j\gamma^{-1},\xi)|_{\xi\in [(j+2)\gamma^{-1},(j+3)\gamma^{-1}]}\ge& \frac{19}{20}+\frac{4M\pi}{9}\frac{9}{10}\big(e^{-\frac{(4N-N)^2}{2N^2}}-e^{-\frac{(4N)^2}{2N^2}}\big)\ge\frac{19}{20}+\frac{1}{20}=1.
\end{align*}
Similar to \eqref{eq-step-h2}, we have from \eqref{eq-M} that
\begin{align*}
  \hat h((j+1)\gamma^{-1},\xi)|_{\xi\in [(j+1)\gamma^{-1},(j+2)\gamma^{-1}]}\ge& \hat h(j\gamma^{-1},\xi)|_{\xi\in [(j+1)\gamma^{-1},(j+2)\gamma^{-1}]}+\frac{10}{9}\mathcal E_{j+1}\\
  \ge& b_j+a_j=a_{j+1},\\
  \hat h((j+1)\gamma^{-1},\xi)|_{\xi\in [(j+2)\gamma^{-1},(j+3)\gamma^{-1}]}\ge& \hat h(j\gamma^{-1},\xi)|_{\xi\in [(j+2)\gamma^{-1},(j+3)\gamma^{-1}]}+\frac{10}{9}\mathcal E_{j+1}\\
  \ge&1+a_j=b_{j+1}.
\end{align*}
That is, the statement also holds for $j+1$.

Recalling the definition of the series $\{a_j\}$ and $\{b_j\}$, one can see that
\begin{align*}
  b_j=a_{j-1}+1=a_{j-2}+b_{j-2}+1=b_{j-1}+b_{j-2}.
\end{align*}
So $\{b_j\}$ is Fibonacci sequence and $\{a_j\}$ is the summation of Fibonacci sequence. Therefore we know that there exists constants $\tilde c_0\ge \frac{1}{3}$ and $\tilde c_1\ge \ln \frac{3}{2}$ such that
\begin{align*}
  b_j\ge \tilde c_0 e^{\tilde c_1j}.
\end{align*}

As a conclusion, at each time $t\in[j\gamma^{-1},(j+1)\gamma^{-1}]$ with $1\le j\le \frac{5}{4}\varepsilon_1 \ln(\gamma^{-1})+1$ we have
\begin{equation}
  \hat h(t,\xi)\ge\left\{
    \begin{array}{ll}
      \tilde c_0e^{-\tilde c_1} e^{\tilde c_1\gamma t},&\xi\in [(j+1)\gamma^{-1},(j+2)\gamma^{-1}];\\
      \frac{19}{20},&\xi\in [(j+2)\gamma^{-1},2\varepsilon_1\gamma^{-1}\ln{(\gamma^{-1})}-1].
    \end{array}
  \right.
\end{equation}

From \eqref{eq-h-ini0}, we can see that
\begin{align*}
  \|\mathfrak h_{in}\|_{L^2_y}=4\pi\sqrt{\varepsilon_1}\gamma^{-\frac{1}{2}}\sqrt{-\ln{(\gamma)}}.
\end{align*}
Therefore, for $t\in[0,\frac{5}{4}\varepsilon_1\gamma^{-1}\ln (\gamma^{-1})]$, there exists $c_0,c_1>0$ such that
\begin{align*}
  \|\hat h(t,\xi)\|_{L^2_\xi}\ge& \tilde c_0e^{-\tilde c_1} e^{\tilde c_1\gamma t}\gamma^{-\frac{1}{2}}+ \frac{19}{20}|2\varepsilon_1\ln{(\gamma^{-1})}\gamma^{-1}-t-4\gamma^{-1}|^{\frac{1}{2}}\\
  \ge&c_0e^{c_1\gamma t}\sqrt{8\varepsilon_1}\gamma^{-\frac{1}{2}}\sqrt{-\ln{(\gamma)}}=\frac{1}{\sqrt2\pi}c_0e^{c_1\gamma t}\|\mathfrak h_{in}\|_{L^2_y}.
\end{align*}
As $\nu=\gamma^{\frac{3}{1-2\delta_0}}$, it is clear that
\begin{align*}
  T=\varepsilon_1\gamma^{-1}(\ln{\nu^{-1}})\le\frac{5}{4}\varepsilon_1\gamma^{-1}\ln (\gamma^{-1}).
\end{align*}
Recalling that $\hat h(t,\xi)$ we studied here is the $1$ mode of $\hat h(t,k,\xi)$, one can do the same estimate for the $-1$ mode, and get for $t\in[0,T]$ that
\begin{align*}
   &\|P_{\pm1}{S_\nu(t,0)}g_{in}\|_{L^2_{x,y}}=2\pi\|\hat h_{\pm1}(t)\|_{L^2_{\xi}}\ge \sqrt 2c_0e^{c_1\gamma t}\|\mathfrak h_{in}\|_{L^2_xL^2_y}=c_0e^{c_1\gamma t}\| g_{in}\|_{L^2_xL^2_y},
\end{align*}
which gives Proposition \ref{prop-lower}.
\end{proof}
As can be seen from the above proof, the result of Proposition \ref{prop-lower} is still valid for $0\le\nu<\gamma^{\frac{3}{1-2\delta_0}}$.
\subsection{Lower bound estimate for the inviscid problem}
In this subsection, we give the application of the dynamical approach to the inviscid problem. Similar to the viscous problem, we focus on the $k=1$ mode, and study the following system:
\begin{equation}\label{eq-h-ft-Euler}
  \left\{
    \begin{array}{l}
      \partial_{t} \hat h(t,\xi)=\int_{\mathbb{R}}M\gamma^2(\xi-\eta)e^{-\frac{\gamma^2}{4}|\xi-\eta|^2}\frac{\hat h(t,\eta)}{(\eta-t)^2+1} d\eta ,\\
      \hat h(0,\xi)=\hat {\mathfrak h}_{in}(\xi)\ge0.
    \end{array}
  \right.
\end{equation}
We also divide the time-frequency into small intervals and study the evolution of $\hat h(t,\xi)$ during each small time interval. For $t\in I_m$, we divide into 4 cases depending on the frequency $\xi\in \mathbb{R}=\mathfrak{I}^{r}_m\cup \mathfrak{I}^{e}_m\cup \mathfrak{I}^{g}_m\cup \mathfrak{I}^{w}_m$, where 
\begin{itemize}
\item[Case 1.] $\xi\in \mathfrak{I}^{r}_m=\left(\bigcup_{n\le m-3}I_n\right)\cup\left(\bigcup_{n>2\varepsilon_1\ln{(\gamma^{-1})}N}I_n\right)$;
\item[Case 2.] $\xi\in \mathfrak{I}^{e}_m=\bigcup_{|n-m|\le 2}I_n$;
\item[Case 3.] $\xi\in \mathfrak{I}^{g}_m=\bigcup_{m+3\le n\le m+2+4N}I_n$;
\item[Case 4.] $\xi\in \mathfrak{I}^{w}_m=\bigcup_{m+3+4N\le n\le 2\varepsilon_1\ln{(\gamma^{-1})}N}I_n$.
\end{itemize}
For Case 1, $\xi\in I_n$ is not in the excitation region nor the region of growth, we have from Lemma \ref{lem-h-up-inf} with $\nu=0$ that
   \begin{align}\label{eq-h-up-basic-Euler}
  \|\hat h(t)\|_{L^\infty_\xi}\le e^{M\pi\gamma t}\|\hat {\mathfrak h}_{in}\|_{L^\infty_\xi}=e^{M\pi\gamma t},\text{ for }t\ge0.
\end{align}
For Case 2, by using the same argument in Lemma \ref{lem-evo-h-case2}, one can see that the self-interaction in the excitation region is weak.
\begin{lemma}\label{lem-evo-h-case2-Euler}
   Let $\hat h$ be the solution of \eqref{eq-h-ft-Euler}. For $t\in I_m$, $\xi\in I_n$, with $m\le \frac{3}{2}\varepsilon_1\ln(\gamma^{-1}) N-1$ and $\ |n-m|\le2$, it holds that
   \begin{align*}
     \hat h(T_{m},\xi)-\frac{6M\pi}{N^2}e^{\frac{M\pi(m-1)}{N}}\le\hat h(t,\xi)\le\hat h(T_{m},\xi)+\frac{6M\pi}{N^2}e^{\frac{M\pi(m-1)}{N}}.
   \end{align*}
\end{lemma}
For Case 3, we treat the region of growth and give the growing rate.
\begin{lemma}\label{lem-evo-h-case3-Euler}
   Let $\hat h$ be the solution of \eqref{eq-h-ft-Euler}. There exists $\varepsilon_1>0$ such that, for $t\in I_m$, $\xi\in I_n$ with $m\le \frac{3}{2}\varepsilon_1\ln(\gamma^{-1}) N-1$, and $m+3\le n\le m+2+4N$, if it holds that
   \begin{align}\label{eq-h-assump3}
     \inf_{\tau\in I_m}\ \inf_{\eta\in\cup_{ |i-m|\le1}I_{i}}\hat h(\tau,\eta)\ge \frac{9}{10},
   \end{align}
   we have 
\begin{align}\label{eq-grow-Euler}
  \hat h(t,\xi)\ge \hat h(T_{m},\xi)+(t-T_{m})\frac{4\gamma M(n-m-2)\pi  }{9N}e^{-\frac{(n-m-2)^2}{2N^2}}\inf_{\tau\in I_m}\inf_{\eta\in\cup_{ |i-m|\le1}I_{i}}\hat h(\tau,\eta).
\end{align}
\end{lemma}
\begin{proof}
  Similar to the viscous problem, we write
  \begin{align*}
  \pa_t\hat h(t,\xi)=&\int_{\mathbb{R}}M\gamma^2(\xi-\eta)e^{-\frac{\gamma^2}{4}|\xi-\eta|^2}\frac{1}{(\eta-t)^2+1}\hat h(t,\eta)d\eta\\
  =&\int_{\cup_{i\in \mathbb Z,|i-m|\le1}I_{i}}M\gamma^2(\xi-\eta)e^{-\frac{\gamma^2}{4}|\xi-\eta|^2}\frac{1}{(\eta-t)^2+1} \hat h(t,\eta)d\eta\\
  &+\int_{\cup_{i\in \mathbb Z,|i-m|>1}I_{i}}M\gamma^2(\xi-\eta)e^{-\frac{\gamma^2}{4}|\xi-\eta|^2}\frac{1}{(\eta-t)^2+1} \hat h(t,\eta)d\eta\\
  \eqdef&\mathcal A_m+\mathcal B_m,
  \end{align*}
  and have
  \begin{align*}
  \mathcal A_m\ge&\frac{M\gamma(n-m-2)}{N}e^{-\frac{(n-m-2)^2}{2N^2}} \frac{8\pi}{9}\inf_{\tau\in I_m}\inf_{\eta\in\cup_{ |i-m|\le1}I_{i}}\hat h(\tau,\eta),
\end{align*}
\begin{align*}
  |\mathcal B_m|\le&2MN\gamma^2e^{M\pi\gamma t}\le2MN\gamma^{2-\frac{3M\pi}{2}\varepsilon_1}.
\end{align*}
By taking 
\begin{align}\label{eq-epsilon1-Euler}
  \varepsilon_1=\frac{1}{9M\pi},
\end{align}
one can easily check that $|\mathcal B_m|\le \frac{1}{2}\mathcal A_m$, which gives the result of this lemma.
\end{proof}
For case 4, we have a similar estimate to the one in the viscous case. 
\begin{lemma}\label{lem-evo-h-case4-Euler}
   Let $\hat h$ be the solution of \eqref{eq-h-ft-Euler}. For $t\in I_m$ and $\xi\in I_n$ with $m\le \frac{3}{2}\varepsilon_1\ln(\gamma^{-1}) N-1$, $m+3+4N\le n\le 2\varepsilon_1\ln{(\gamma^{-1})}N$, and $\varepsilon_1$ given in Lemma \ref{lem-evo-h-case3-Euler}, if it holds that 
   \begin{align}\label{eq-h-assump4}
     \inf_{0\le m'\le m}\ \inf_{\tau\in I_m}\ \inf_{\eta\in\cup_{ |i-m|\le1}I_{i}}\hat h(\tau,\eta)\ge 0,
   \end{align}
    we have
   \begin{align*}
  \hat h(t,\xi)\ge \frac{19}{20}.
\end{align*}
\end{lemma}
\begin{lemma}\label{lem-low-9/10-Euler}
With our well-chosen initial data, assumption \eqref{eq-h-assump3} in Lemma \ref{lem-evo-h-case3-Euler} and assumption \eqref{eq-h-assump4} in Lemma \ref{lem-evo-h-case4-Euler} are automatically satisfied.
\end{lemma}
\begin{proof}[Proof of Proposition \ref{prop-lower-Euler}]
  Proposition \ref{prop-lower-Euler} can be proved by combining Lemma \ref{lem-evo-h-case2-Euler}-\ref{lem-low-9/10-Euler} and following the proof of Proposition \ref{prop-lower}. We omit the details.
\end{proof}

\section{Nonlinear instability of the viscous flow}
In this section, we focus on the nonlinear system of the viscous flow and give the proof of Theorem \ref{thm-low-exp}. The nonlinear instability is produced from the linear instability which was derived by the dynamical approach in Section 3.

We rewrite \eqref{eq:b-perturbation} in the following form:
\begin{equation}\label{eq-om-nonl}
  \left\{
    \begin{array}{l}
      \pa_t\omega_{\neq}+y\pa_x\omega_{\neq}-\pa_y^2b\pa_x(\Delta)^{-1}\omega_{\neq}-\nu\Delta\omega_{\neq}=-\mathcal L-\mathcal N^{(1)}-\mathcal N^{(2)}-\mathcal N^{(3)},\\
      (u^{(1)}_{\neq},u^{(2)}_{\neq})=(\pa_y(-\Delta)^{-1}\omega_{\neq},-\pa_x(-\Delta)^{-1}\omega_{\neq}),\\
      \omega_{\neq}|_{t=0}=\omega_{in}=\frac{\varepsilon_0\nu^{\frac{2}{3}+\delta_1-\frac{1}{3}\delta_0}}{\sqrt{\varepsilon_1}\sqrt{-\ln{\gamma}}}g_{in},
    \end{array}
  \right.
\end{equation}
where $b_\nu(t,y)$ is given in \eqref{eq-shear-b}, $g_{in}$ is given in \eqref{eq-ini-g} which has only $\pm1$ modes, $\varepsilon_1=\frac{2\delta_0}{M\pi}$ given in \eqref{eq-epsilon1}, and $\gamma=\nu^{\frac{1}{3}-\frac{2}{3}\delta_0}$.

Here
\begin{align*}
  \mathfrak L=(b-y)\pa_x\omega_{\neq},&\quad  \mathcal N^{(1)}=\big(u^{(1)}_{\neq}\pa_x\omega_{\neq}\big)_{\neq}+\big(u^{(2)}_{\neq}\pa_y\omega_{\neq}\big)_{\neq},\\
  \mathcal N^{(2)}=u^{(1)}_{0}\pa_x\omega_{\neq},&\qquad \mathcal N^{(3)}=u^{(2)}_{\neq}\pa_y\omega_0,
\end{align*}
$\omega_0(t,y)$ is the $0$ mode of vorticity which satisfies
\begin{equation*}
  \left\{
    \begin{array}{l}
      \pa_t\omega_0-\nu\pa_y^2\omega_0=-\big(u^{(1)}_{\neq}\pa_x\omega_{\neq}\big)_0-\big(u^{(2)}_{\neq}\pa_y\omega_{\neq}\big)_0,\\
      \omega_0|_{t=0}=0,
    \end{array}
  \right.
\end{equation*}
and $u^{(1)}_0$ is the $0$ mode of horizontal velocity which satisfies
\begin{equation*}
  \left\{
    \begin{array}{l}
      \pa_t u^{(1)}_0-\nu\pa_y^2u^{(1)}_0=-\big(u^{(1)}_{\neq}\pa_xu^{(1)}_{\neq}\big)_0-\big(u^{(2)}_{\neq}\pa_yu^{(1)}_{\neq}\big)_0,\\
      u^{(1)}_0|_{t=0}=0.
    \end{array}
  \right.
\end{equation*}
Then we have
\begin{align}\label{eq:DH-1}
  \omega_{\neq}(t,x,y)={S_\nu(t,0)}\omega_{in}(x,y)-\int^t_0 {S_\nu(t,s)}(\mathfrak L+\mathcal N^{(1)}+\mathcal N^{(2)}+\mathcal N^{(3)})(s,x,y)ds,
\end{align}
\begin{align}\label{eq:DH-2}
  \omega_0(t,y)=-\int^t_0 e^{(t-s)\nu\pa_y^2}\Big(\big(u^{(1)}_{\neq}\pa_x\omega_{\neq}\big)_0+\big(u^{(2)}_{\neq}\pa_y\omega_{\neq}\big)_0\Big)(s,y) ds,
\end{align}
and
\begin{align}\label{eq:DH-3}
  u^{(1)}_0(t,y)=-\int^t_0 e^{(t-s)\nu\pa_y^2}\Big(\big(u^{(1)}_{\neq}\pa_xu^{(1)}_{\neq}\big)_0+\big(u^{(2)}_{\neq}\pa_yu^{(1)}_{\neq}\big)_0\Big)(s,y) ds.
\end{align}

From \eqref{eq-h-ini0} and \eqref{eq-ini-g}, we have 
\begin{align}\label{eq-size-om-ini}
  \|\omega_{in}\|_{\mathcal F L^1_kL^2_y}=\frac{\varepsilon_0\nu^{\frac{2}{3}+\delta_1-\frac{1}{3}\delta_0}}{\sqrt{\pi\varepsilon_1}\sqrt{-\ln{\gamma}}}\|g_{in}\|_{L^2_{x,y}}=\frac{\varepsilon_0\nu^{\frac{2}{3}+\delta_1-\frac{1}{3}\delta_0}}{\sqrt{2\pi\varepsilon_1}\sqrt{-\ln{\gamma}}}\|\mathfrak h_{in}\|_{L^2_{y}}=\sqrt{8\pi}\varepsilon_0\nu^{\frac{1}{2}+\delta_1}.
\end{align}

We first give upper bound estimates for the {{}solution operator $S_\nu(t,s)$}, then prove the upper bound estimates for the solution of \eqref{eq-om-nonl} by Duhamel's principle and bootstrap argument. Compared to \cite{MasmoudiZhao2020cpde}, here we modify the function spaces of the solution which simplifies the proof. Then we prove the lower bound estimates in Theorem \ref{thm-low-exp} by using the linear lower bound estimate obtained in Proposition \ref{prop-lower} together with the upper bound estimates.
\subsection{Upper bound estimates of ${S_\nu(t,s)}$}
In this subsection, we give upper bound estimates of the solution to \eqref{eq-g-linear} with $\gamma=\nu^{\frac{1}{3}-\frac{2}{3}\delta_0}$ from general initial data.
\begin{proposition}\label{prop-lin-upper}
Given $f(x,y)\in {\mathcal F L^1_kL^2_y}(\mathbb T\times\mathbb R)$ such that $\int_{\mathbb T} f(x,y) dx=0$. There exists constant $C_0,C_1>0$ such that for any $T\ge s\ge0$,
{{}\begin{align}
  \|e^{-\frac{1}{2}C_0M^2\gamma (t-s)}S_\nu(t,s)f\|_{\tilde L^\infty_t\left([s,T];\mathcal F L^1_kL^2_y\right)}\le C_1\|f\|_{\mathcal F L^1_kL^2_y};\label{eq-up-est-1}\\
  \|e^{-\frac{1}{2}C_0M^2\gamma (t-s)}\nabla S_\nu(t,s)f\|_{\tilde L^2_t\left([s,T];\mathcal F L^1_kL^2_y\right)}\le C_1 \nu^{-\frac{1}{2}} \|f\|_{\mathcal F L^1_kL^2_y};\label{eq-up-est-2}\\
  \|e^{-\frac{1}{2}C_0M^2\gamma (t-s)}\pa_xS_\nu(t,s)f\|_{\tilde L^1_t\left([s,T];\mathcal F L^1_kL^2_y\right)}\le C_1 \nu^{-\frac{1}{2}} \|f\|_{\mathcal F L^1_kL^2_y};\label{eq-up-est-3}\\
   \|e^{-\frac{1}{2}C_0M^2\gamma (t-s)} \nabla(-\Delta)^{-1} S_\nu(t,s)f\|_{\tilde L^\infty_t\left([s,T];\mathcal F L^1_kL^\infty_y\right)}\le C_1 \|f\|_{\mathcal F L^1_kL^2_y};\label{eq-up-est-4}\\
  \|e^{-\frac{1}{2}C_0M^2\gamma (t-s)}\pa_x(-\Delta)^{-1} S_\nu(t,s)f\|_{\tilde L^2_t\left([s,T];\mathcal F L^1_kL^\infty_y\right)}\le C_1 \|f\|_{\mathcal F L^1_kL^2_y}.\label{eq-up-est-5}
\end{align}
Here $S_\nu(t,s)$ denotes the solution operator of \eqref{eq-g-linear} that defined in Section 3.}
\end{proposition}
\begin{proof}
  By the change of coordinate, it suffices to study the system \eqref{eq-h-ft-ori} with $\hat h_k(s,\xi)=\hat f_k(\xi-ks)$. We introduce the ghost type weight $e^{\arctan{(\frac{\xi}{k}-t)}}$ which satisfies
\begin{align*}
  e^{-\pi}\le e^{2\arctan{(\frac{\xi}{k}-t)}}\le e^{\pi}.
\end{align*}
This kind of weight is first used by Alinhac \cite{Alinhac2001}, and it is useful in studying stability problems in fluid flow, see \cite{BGM2015, BVW2018}.

By energy method, we have
\begin{align*}
  &\pa_t\int_{\mathbb R} e^{2\arctan{(\frac{\xi}{k}-t)}} |\hat h_k(t,\xi)|^2d\xi\\
  %=&-2\int_{\mathbb R}\frac{|k|^2}{(\xi-kt)^2+k^2} e^{2\arctan{(\frac{\xi}{k}-t)}} |\hat h_k(t,\xi)|^2d\xi+2\int_{\mathbb R}e^{2\arctan{(\frac{\xi}{k}-t)}} \hat h_k(t,\xi)\pa_t\hat h_k(t,\xi)d\xi\\
  =&-2\int_{\mathbb R}\frac{|k|^2}{(\xi-kt)^2+k^2} e^{2\arctan{(\frac{\xi}{k}-t)}} |\hat h_k(t,\xi)|^2d\xi\\
  &+2\int_{\mathbb R}\int_{\mathbb R}M\gamma^2(\xi-\eta)e^{-(\nu t+\frac{\gamma^2}{4})|\xi-\eta|^2}\frac{k}{(\eta-kt)^2+k^2}\hat h_k(t,\eta)d\eta\  e^{2\arctan{(\frac{\xi}{k}-t)}} \hat h_k(t,\xi)d\xi\\
  &-2\int_{\mathbb R}\nu e^{2\arctan{(\frac{\xi}{k}-t)}}\big(k^2+(\xi-kt)^2\big)|\hat h_k(t,\xi)|^2d\xi\\
  \eqdef& I+II+III.
\end{align*}
By H{\"o}lder's and Young's convolution inequality, we have
\begin{align*}
  |II|\le&C\left\|\int_{\mathbb R}M\gamma^2(\xi-\eta)e^{-(\nu t+\frac{\gamma^2}{4})|\xi-\eta|^2}\frac{k}{(\eta-kt)^2+k^2}\hat h_k(t,\eta)d\eta\right\|_{L^2_\xi} \|e^{2\arctan{(\frac{\xi}{k}-t)}} \hat h_k(t,\xi)\|_{L^2_\xi}\\
  \le&M\left(\int_{\mathbb R}\gamma^4\xi^2e^{-2(\nu t+\frac{\gamma^2}{4})|\xi|^2}d\xi\right)^{\frac{1}{2}}
  \int_{\mathbb R}\frac{|k|}{(\xi-kt)^2+k^2}|\hat h_k(t,\xi)|d\xi\|e^{2\arctan{(\frac{\xi}{k}-t)}} \hat h_k(t,\xi)\|_{L^2_\xi}\\
  \le&CM\gamma^{\frac{1}{2}}
  \left(\int_{\mathbb R}\frac{|k|^2}{(\xi-kt)^2+k^2}|\hat h_k(t,\xi)|^2d\xi\right)^{\frac{1}{2}}\left(\int_{\mathbb R}\frac{1}{(\xi-kt)^2+k^2}d\xi\right)^{\frac{1}{2}}\|e^{\arctan{(\frac{\xi}{k}-t)}} \hat h_k(t,\xi)\|_{L^2_\xi}\\
  \le&CM\gamma^{\frac{1}{2}}\left(\int_{\mathbb R}\frac{|k|^2}{(\xi-kt)^2+k^2}|\hat h_k(t,\xi)|^2d\xi\right)^{\frac{1}{2}}\|e^{\arctan{(\frac{\xi}{k}-t)}} \hat h_k(t,\xi)\|_{L^2_\xi}\\
  \le&CM\gamma^{\frac{1}{2}} \left(\int_{\mathbb R}\frac{|k|^2}{(\xi-kt)^2+k^2} e^{2\arctan{(\frac{\xi}{k}-t)}} |\hat h_k(t,\xi)|^2d\xi\right)^{\frac{1}{2}}\|e^{\arctan{(\frac{\xi}{k}-t)}} \hat h_k(t,\xi)\|_{L^2_\xi},
\end{align*}
which implies
\begin{align*}
  &\pa_t\int_{\mathbb R} e^{2\arctan{(\frac{\xi}{k}-t)}} |\hat h_k(t,\xi)|^2d\xi\\
  \le&-2\int_{\mathbb R}\frac{|k|^2}{(\xi-kt)^2+k^2} e^{2\arctan{(\frac{\xi}{k}-t)}} |\hat h_k(t,\xi)|^2d\xi\\
  &+CM\gamma^{\frac{1}{2}} \left(\int_{\mathbb R}\frac{|k|^2}{(\xi-kt)^2+k^2} e^{2\arctan{(\frac{\xi}{k}-t)}} |\hat h_k(t,\xi)|^2d\xi\right)^{\frac{1}{2}}\left(\int_{\mathbb R} e^{2\arctan{(\frac{\xi}{k}-t)}} |\hat h_k(t,\xi)|^2d\xi\right)^{\frac{1}{2}}\\
  &-2\int_{\mathbb R}\nu e^{2\arctan{(\frac{\xi}{k}-t)}}\big(k^2+(\xi-kt)^2\big)|\hat h_k(t,\xi)|^2d\xi\\
  \le&C_0M^2\gamma \int_{\mathbb R} e^{2\arctan{(\frac{\xi}{k}-t)}} |\hat h_k(t,\xi)|^2d\xi-\int_{\mathbb R}\frac{|k|^2}{(\xi-kt)^2+k^2} e^{2\arctan{(\frac{\xi}{k}-t)}} |\hat h_k(t,\xi)|^2d\xi\\
  &-2\int_{\mathbb R}\nu e^{2\arctan{(\frac{\xi}{k}-t)}}\big(k^2+(\xi-kt)^2\big)|\hat h_k(t,\xi)|^2d\xi.
\end{align*}
From this, we deduce that
\begin{equation}\label{eq-est-h-energy}
  \begin{aligned}    
      &\pa_t\int_{\mathbb R}e^{-C_0M^2\gamma t} e^{2\arctan{(\frac{\xi}{k}-t)}} |\hat h_k(t,\xi)|^2d\xi\\
      \le&-\int_{\mathbb R} e^{-C_0M^2\gamma t}e^{2\arctan{(\frac{\xi}{k}-t)}}\frac{|k|^2}{(\xi-kt)^2+k^2} |\hat h_k(t,\xi)|^2d\xi\\
      &-2\int_{\mathbb R}\nu e^{-C_0M^2\gamma t}e^{2\arctan{(\frac{\xi}{k}-t)}}\big(k^2+(\xi-kt)^2\big)|\hat h_k(t,\xi)|^2d\xi.
  \end{aligned}
\end{equation}
It follows that{{}
\begin{equation}\label{eq-est-h-energy2}
  \begin{aligned}    
   &\int_{\mathbb R}e^{-C_0M^2\gamma t} e^{2\arctan{(\frac{\xi}{k}-t)}} |\hat h_k(t,\xi)|^2d\xi\\
  &+\int_s^t\int_{\mathbb R}e^{-C_0M^2\gamma s'}\frac{|k|^2}{(\xi-ks')^2+k^2} e^{2\arctan{(\frac{\xi}{k}-s')}} |\hat h_k(s',\xi)|^2d\xi ds'\\
  &+2\int_s^t\int_{\mathbb R}\nu e^{-C_0M^2\gamma s'} e^{2\arctan{(\frac{\xi}{k}-s')}}\big(k^2+(\xi-ks')^2\big)|\hat h_k(s',\xi)|^2d\xi ds'\\
  \le&\int_{\mathbb R}e^{-C_0M^2\gamma s} e^{2\arctan{(\frac{\xi}{k}-s)}} |\hat h_k(s,\xi)|^2d\xi=\int_{\mathbb R}e^{-C_0M^2\gamma s}e^{2\arctan{(\frac{\xi}{k}-s)}} |\hat f_k(\xi-ks)|^2d\xi.   
  \end{aligned}
\end{equation}
Summing up the above inequality in $k$, we get
\begin{align*}
  \|e^{-\frac{1}{2}C_0M^2\gamma (t-s)}\hat h\|_{\tilde L^\infty_t\left([s,T];\mathcal F L^1_kL^2_\xi\right)}\le C\|f\|_{\mathcal F L^1_kL^2_y},
\end{align*}
and
\begin{align*}
  \|e^{-\frac{1}{2}C_0M^2\gamma (t-s)}\big(k,\xi-kt\big)\hat h\|_{\tilde L^2_t\left([s,T];\mathcal F L^1_kL^2_y\right)}\le C \nu^{-\frac{1}{2}} \|f\|_{\mathcal F L^1_kL^2_y}.
\end{align*}}
As
\begin{align*}
 \|\tilde h_k(t)\|_{L^2_y}=\|\big(\widetilde {S_\nu(t,s)f}\big)_k\|_{L^2_y}, \quad \|(k,\pa_y-ikt)\tilde h_k(t)\|_{L^2_y}=\|(k,\pa_y)\big(\widetilde {S_\nu(t,s)f}\big)_k\|_{L^2_y}
\end{align*}
the estimates \eqref{eq-up-est-1} and \eqref{eq-up-est-2} follows immediately.

Next, we prove \eqref{eq-up-est-3}. By using Cauchy inequality, we have
\begin{align*}
  &\int_{\mathbb R}\nu^{\frac{1}{2}}|k|  e^{-C_0M^2\gamma t} e^{2\arctan{(\frac{\xi}{k}-t)}}|\hat h_k(t,\xi)|^2d\xi \\
  \le&\int_{\mathbb R}e^{-C_0M^2\gamma t}\frac{|k|^2}{(\xi-kt)^2+k^2} e^{2\arctan{(\frac{\xi}{k}-t)}} |\hat h_k(t,\xi)|^2d\xi \\
  &+\int_{\mathbb R}\nu e^{-C_0M^2\gamma t} e^{2\arctan{(\frac{\xi}{k}-t)}}\big(k^2+(\xi-kt)^2\big)|\hat h_k(t,\xi)|^2d\xi.
\end{align*}
Then, we deduce from \eqref{eq-est-h-energy} that
\begin{align*}
  \pa_t\int_{\mathbb R}e^{-C_0M^2\gamma t} e^{2\arctan{(\frac{\xi}{k}-t)}} |\hat h_k(t,\xi)|^2d\xi\le-\nu^{\frac{1}{2}}|k|\int_{\mathbb R}e^{-C_0M^2\gamma t} e^{2\arctan{(\frac{\xi}{k}-t)}} |\hat h_k(t,\xi)|^2d\xi.
\end{align*}
Therefore,
\begin{align*}
  \pa_t\Big(\int_{\mathbb R}e^{-C_0M^2\gamma t} e^{2\arctan{(\frac{\xi}{k}-t)}} |\hat h_k(t,\xi)|^2d\xi\Big)^{\frac{1}{2}}\le-\nu^{\frac{1}{2}}|k|\Big(\int_{\mathbb R}e^{-C_0M^2\gamma t}e^{2\arctan{(\frac{\xi}{k}-t)}} |\hat h_k(t,\xi)|^2d\xi\Big)^{\frac{1}{2}}.
\end{align*}
Then, it follows from Gronwall's inequality that{{}
\begin{align*}
  &\Big(\int_{\mathbb R}e^{-C_0M^2\gamma t} e^{2\arctan{(\frac{\xi}{k}-t)}} |\hat h_k(t,\xi)|^2d\xi\Big)^{\frac{1}{2}}\\
  \le& e^{-\nu^{\frac{1}{2}}|k|(t-s)}\Big(\int_{\mathbb R}e^{-C_0M^2\gamma s} e^{2\arctan{(\frac{\xi}{k}-s)}} |\hat h_k(s,\xi)|^2d\xi\Big)^{\frac{1}{2}}\\
  =&e^{-\nu^{\frac{1}{2}}|k|(t-s)}\Big(\int_{\mathbb R}e^{-C_0M^2\gamma s}e^{2\arctan{(\frac{\xi}{k}-s)}} |\hat f_k(\xi-ks)|^2d\xi\Big)^{\frac{1}{2}},
\end{align*}
and
\begin{align*}
  &\int^T_s\Big(\int_{\mathbb R}e^{-C_0M^2\gamma t} e^{2\arctan{(\frac{\xi}{k}-t)}} |\hat h_k(t,\xi)|^2d\xi\Big)^{\frac{1}{2}}dt\\
  \le& C \nu^{-\frac{1}{2}}|k|^{-1}\Big(\int_{\mathbb R}e^{-C_0M^2\gamma s}e^{2\arctan{(\frac{\xi}{k}-s)}} |\hat f_k(\xi-ks)|^2d\xi\Big)^{\frac{1}{2}}.
\end{align*}
As a result, we get
\begin{align*}
  \|e^{-\frac{1}{2}C_0M^2\gamma (t-s)}|k|\hat h_k(t,\xi)\|_{L^1_t\left([s,T];L^2_\xi\right)}\le C \nu^{-\frac{1}{2}} \|\hat f_k(\xi)\|_{L^2_\xi},
\end{align*}}
and \eqref{eq-up-est-3} follows immediately.

Recalling inequality \eqref{eq-est-h-energy2} and the fact that $\tilde f_k(y)=0$ for $k=0$, we have $\tilde h_k(t,y)\equiv0$ and{{}
\begin{align*}
  &\|e^{-\frac{1}{2}C_0M^2\gamma t}\nabla(-\Delta)^{-1}S_\nu(t,s)f\|_{\tilde L^\infty_t\left([s,T];\mathcal F L^1_kL^\infty_y\right)}\\
  \le& \sum_{k\in\mathbb Z/0}\left\|\int_{\mathbb R}e^{-\frac{1}{2}C_0M^2\gamma t}\frac{(|k|,|\xi-kt|)}{k^2+(\xi-kt)^2}|\hat h_k(t,\xi)| d \xi\right\|_{L^\infty_t([s,T])}\\
  \le& C\sum_{k\in\mathbb Z/0}\|e^{-\frac{1}{2}C_0M^2\gamma t} \hat h_k(t,\xi)\|_{L^\infty_t([s,T];L^2_{\xi})}\\
  \le& C e^{-\frac{1}{2}C_0M^2\gamma s}\sum_{k\in\mathbb Z/0}\|\hat f_k(\xi)\|_{L^2_{\xi}}=Ce^{-\frac{1}{2}C_0M^2\gamma s}\|f\|_{\mathcal F L^1_kL^2_y},
\end{align*}}
which is \eqref{eq-up-est-4}.

At last, we turn to \eqref{eq-up-est-5}. By using \eqref{eq-est-h-energy2}, we deduce that{{}
\begin{align*}
  &\|e^{-\frac{1}{2}C_0M^2\gamma t}\pa_x(-\Delta)^{-1}S_\nu(t,s)f\|_{\tilde L^2_t\left([s,T];\mathcal F L^1_kL^\infty_y\right)}\\
  \le&\sum_{k\in \mathbb Z/0}\left\|e^{-\frac{1}{2}C_0M^2\gamma t}\frac{|k|}{k^2+(\xi-kt)^2}|\hat h_k(t,\xi)|\right\|_{L^2_t([s,T];L^1_\xi)}\\
  \le&C\sum_{k\in \mathbb Z/0}\Big(\int_s^T\int_{\mathbb R}e^{-C_0M^2\gamma s'}\frac{|k|^2}{k^2+(\xi-ks')^2} e^{2\arctan{(\frac{\xi}{k}-s')}}|\hat h_k(s',\xi)|^2d\xi ds'\Big)^{\frac{1}{2}}\\
  \le& C e^{-\frac{1}{2}C_0M^2\gamma s}\sum_{k\in\mathbb Z/0}\|\hat f_k(\xi)\|_{L^2_{\xi}}=Ce^{-\frac{1}{2}C_0M^2\gamma s}\|f\|_{\mathcal F L^1_kL^2_y},
\end{align*}}
which gives \eqref{eq-up-est-5}.
\end{proof}
\subsection{Upper bound estimates for the nonlinear equation}
In this subsection, we use the bootstrap argument to prove the upper bound estimates for the solution of \eqref{eq-om-nonl}.

 Suppose for some $0\le \widetilde T\le T=\varepsilon_1\gamma^{-1} \ln(\nu^{-1})$, the following inequalities hold:
\begin{align}
  \|e^{-\frac{1}{2}C_0M^2\gamma t}\omega_{\neq}\|_{\tilde L^\infty_t\left([0,\widetilde T];\mathcal F L^1_kL^2_y\right)}\le& 2C_2 \|\omega_{in}\|_{\mathcal F L^1_kL^2_y};\label{eq-om-nonl-est1}\\
  \|e^{-\frac{1}{2}C_0M^2\gamma t}\pa_y\omega_{\neq}\|_{\tilde L^2_t\left([0,\widetilde T];\mathcal F L^1_kL^2_y\right)}\le& 2C_2 \nu^{-\frac{1}{2}}\|\omega_{in}\|_{\mathcal F L^1_kL^2_y};\label{eq-om-nonl-est2}\\
  \|e^{-\frac{1}{2}C_0M^2\gamma t}\pa_x\omega_{\neq}\|_{\tilde L^1_t\left([0,\widetilde T];\mathcal F L^1_kL^2_y\right)}\le& 2C_2 \nu^{-\frac{1}{2}}\|\omega_{in}\|_{\mathcal F L^1_kL^2_y};\label{eq-om-nonl-est3}\\
  \|e^{-\frac{1}{2}C_0M^2\gamma t}u^{(1)}_{\neq}\|_{\tilde L^\infty_t\left([0,\widetilde T];\mathcal F L^1_kL^\infty_y\right)}\le& 2C_2 \|\omega_{in}\|_{\mathcal F L^1_kL^2_y};\label{eq-om-nonl-est4}\\
  \|e^{-\frac{1}{2}C_0M^2\gamma t}u^{(2)}_{\neq}\|_{\tilde L^2_t\left([0,\widetilde T];\mathcal F L^1_kL^\infty_y\right)}\le& 2C_2 \|\omega_{in}\|_{\mathcal F L^1_kL^2_y};\label{eq-om-nonl-est5}\\
  \|e^{-\frac{1}{2}C_0M^2\gamma t}u^{(1)}_0\|_{L^\infty_t\left([0,\widetilde T]; L^\infty_y\right)}\le& 2C_2 \|\omega_{in}\|_{\mathcal F L^1_kL^2_y};\label{eq-om-nonl-est6}\\
  \|e^{-\frac{1}{2}C_0M^2\gamma t}\pa_y\omega_0\|_{ L^2_t\left([0,\widetilde T]; L^2_y\right)}\le& 2C_2\nu^{-\frac{1}{2}} \|\omega_{in}\|_{\mathcal F L^1_kL^2_y},\label{eq-om-nonl-est7}
\end{align}
where $C_2>0$ is a constant which will be determined later. 
\begin{proposition}[Bootstrap]\label{prop-upper}
  Let $\nu>0$ be small enough, and $\omega$ be the solution of \eqref{eq-om-nonl} for some $0<\widetilde T\le\varepsilon_1\gamma^{-1} \ln(\nu^{-1})$ with $\gamma=\nu^{\frac{1}{3}-\frac{2}{3}\delta_0}$, the estimates \eqref{eq-om-nonl-est1}-\eqref{eq-om-nonl-est7} hold for some constant $C_2>0$ on $[0,\widetilde T]$. Then there exist constants $\delta_1,\varepsilon_0>0$ so that these same estimates hold with all the occurrences of $2$ on the right-hand side replaced by $1$.
\end{proposition}
We first give some auxiliary lemmas.
\begin{lemma}\label{lem-chemin-space}
  Let $$f(t,x,y)\in \tilde L^{p_1}_t\left([0,\widetilde T];\mathcal F L^1_kL^{q_1}_y(\mathbb T\times\mathbb R)\right),\ g(t,x,y)\in \tilde L^{p_2}_t\left([0,\widetilde T];\mathcal F L^1_kL^{q_2}_y(\mathbb T\times\mathbb R)\right),$$ then we have
  \begin{align*}
  \|fg\|_{\tilde L^p_t\left([0,\widetilde T];\mathcal F L^1_kL^q_y\right)}\le \|f\|_{\tilde L^{p_1}_t\left([0,\widetilde T];\mathcal F L^1_kL^{q_1}_y\right)}\|g\|_{\tilde L^{p_2}_t\left([0,\widetilde T];\mathcal F L^1_kL^{q_2}_y\right)},
\end{align*}
where $p,p_1,p_2,q,q_1,q_2>0$ are constants satisfying
\begin{align}\label{eq-holder}
  \frac{1}{p}=\frac{1}{p_1}+\frac{1}{p_2},\quad\frac{1}{q}=\frac{1}{q_1}+\frac{1}{q_2}.
\end{align}
As 
\begin{align*}
  \|f_{0}\|_{ L^p_t\left([0,\widetilde T];L^q_y\right)}+\|f_{\neq}\|_{\tilde L^p_t\left([0,\widetilde T];\mathcal F L^1_kL^q_y\right)}=\|f\|_{\tilde L^p_t\left([0,\widetilde T];\mathcal F L^1_kL^q_y\right)},
\end{align*}
it follows from \eqref{eq-holder} that
\begin{align*}
  \|(fg)_{0}\|_{ L^p_t\left([0,\widetilde T];L^q_y\right)}+\|(fg)_{\neq}\|_{\tilde L^p_t\left([0,\widetilde T];\mathcal F L^1_kL^q_y\right)}\le& \|f\|_{\tilde L^{p_1}_t\left([0,\widetilde T];\mathcal F L^1_kL^{q_1}_y\right)}\|g\|_{\tilde L^{p_2}_t\left([0,\widetilde T];\mathcal F L^1_kL^{q_2}_y\right)}.
\end{align*}
\begin{proof}
  By using H{\"o}lder's inequality and Young's convolution inequality, we have
  \begin{align*}
    &\|fg\|_{\tilde L^p_t\left([0,\widetilde T];\mathcal F L^1_kL^q_y\right)}\\
  %  =&\sum_{k\in\mathbb Z}\left(\int_{0}^{\widetilde T}\left(\int_{\mathbb R} |\widetilde {(fg)}_k(t,y)|^qd y\right)^{\frac{p}{q}} d t\right)^{\frac{1}{p}}\\
    =&\sum_{k\in\mathbb Z}\left(\int_{0}^{\widetilde T}\left(\int_{\mathbb R} \left|\sum_{j\in\mathbb Z}\tilde f_j\tilde g_{k-j}(t,y)\right|^qd y\right)^{\frac{p}{q}} d t\right)^{\frac{1}{p}}\\
    \le&\sum_{k\in\mathbb Z}\sum_{j\in\mathbb Z}\left(\int_{0}^{\widetilde T}\left(\int_{\mathbb R} \left|\tilde f_j\tilde g_{k-j}(t,y)\right|^qd y\right)^{\frac{p}{q}} d t\right)^{\frac{1}{p}}\\
    \le&\sum_{k\in\mathbb Z}\sum_{j\in\mathbb Z}\left(\int_{0}^{\widetilde T}\left(\int_{\mathbb R} |\tilde {f}_j(t,y)|^{q_1}d y\right)^{\frac{p_1}{q_1}} d t\right)^{\frac{1}{p_1}}\left(\int_{0}^{\widetilde T}\left(\int_{\mathbb R} |\tilde {g}_{k-j}(t,y)|^{q_2}d y\right)^{\frac{p_2}{q_2}} d t\right)^{\frac{1}{p_2}}\\
    \le&\sum_{k\in\mathbb Z}\left(\int_{0}^{\widetilde T}\left(\int_{\mathbb R}|\tilde {f}_k(t,y)|^{q_1}d y\right)^{\frac{p_1}{q_1}} d t\right)^{\frac{1}{p_1}}\sum_{j\in\mathbb Z}\left(\int_{0}^{\widetilde T}\left(\int_{\mathbb R} |\tilde {g}_{j}(t,y)|^{q_2}d y\right)^{\frac{p_2}{q_2}} d t\right)^{\frac{1}{p_2}}\\
    =&\|f\|_{\tilde L^{p_1}_t\left([0,\widetilde T];\mathcal F L^1_kL^{q_1}_y\right)}\|g\|_{\tilde L^{p_2}_t\left([0,\widetilde T];\mathcal F L^1_kL^{q_2}_y\right)},
  \end{align*}
which gives the lemma.
\end{proof}
\end{lemma}
\begin{lemma}\label{lem-est-nonl}
  Under the bootstrap assumptions \eqref{eq-om-nonl-est1}-\eqref{eq-om-nonl-est7}, it holds that
  \begin{align*}
    \|e^{-\frac{1}{2}C_0M^2\gamma t}\mathfrak L\|_{\tilde L^1_t\left([0,\widetilde T];\mathcal F L^1_kL^2_y\right)}\le CC_2M\gamma^2\nu^{-\frac{1}{2}}\|\omega_{in}\|_{\mathcal F L^1_kL^2_y},
  \end{align*}
  and
  \begin{align*}
    \sum_{i=1,2,3}\|e^{-\frac{1}{2}C_0M^2\gamma t}\mathcal N^{(i)}\|_{\tilde L^1_t\left([0,\widetilde T];\mathcal F L^1_kL^2_y\right)}\le CC_2^2e^{\frac{1}{2}C_0M^2\gamma \widetilde T}\varepsilon_0\nu^{\delta_1}\|\omega_{in}\|_{\mathcal F L^1_kL^2_y}.
  \end{align*}
We also have
  \begin{align*}
    &\|e^{-\frac{1}{2}C_0M^2\gamma t}\big(u^{(1)}_{\neq}\pa_xu^{(1)}_{\neq}\big)_0\|_{L^1_t([0,\widetilde T];L^2_{y})}+\|e^{-\frac{1}{2}C_0M^2\gamma t}\big(u^{(2)}_{\neq}\pa_yu^{(1)}_{\neq}\big)_0\|_{L^1_t([0,\widetilde T];L^2_{y})}\\
    &+\|e^{-\frac{1}{2}C_0M^2\gamma t}\big(u^{(1)}_{\neq}\pa_x\omega_{\neq}\big)_0\|_{L^1_t([0,\widetilde T];L^2_{y})}+\|e^{-\frac{1}{2}C_0M^2\gamma t}\big(u^{(2)}_{\neq}\pa_y\omega_{\neq}\big)_0\|_{L^1_t([0,\widetilde T];L^2_{y})}\\
    \le& CC_2^2e^{\frac{1}{2}C_0M^2\gamma \widetilde T}\varepsilon_0\nu^{\delta_1}\|\omega_{in}\|_{\mathcal F L^1_kL^2_y}.
  \end{align*}
\end{lemma}
\begin{proof}
Recalling that $\|b_\nu(t,y)-y\|_{L^\infty_y}=\pi M\gamma^2$, we get from \eqref{eq-om-nonl-est3} that
  \begin{align*}
    &\|e^{-\frac{1}{2}C_0M^2\gamma t}\mathfrak L\|_{\tilde L^1_t\left([0,\widetilde T];\mathcal F L^1_kL^2_y\right)}\\
    \le&\|b_\nu(t,y)-y\|_{L^\infty_t([0,\widetilde T];L^\infty_y)}\|e^{-\frac{1}{2}C_0M^2\gamma t}\pa_x\omega_{\neq}\|_{\tilde L^1_t\left([0,\widetilde T];\mathcal F L^1_kL^2_y\right)}\\
    \le& CC_2M\gamma^2\nu^{-\frac{1}{2}} \|\omega_{in}\|_{\mathcal F L^1_kL^2_y}.
  \end{align*}
From \eqref{eq-size-om-ini}, \eqref{eq-om-nonl-est2}-\eqref{eq-om-nonl-est7}, and Lemma \ref{lem-chemin-space}, we get that
\begin{align*}
  &\sum_{i=1,2,3}\|e^{-\frac{1}{2}C_0M^2\gamma t}\mathcal N^{(i)}\|_{\tilde L^1_t\left([0,\widetilde T];\mathcal F L^1_kL^2_y\right)}\\
  \le&\|e^{-\frac{1}{2}C_0M^2\gamma t}\big(u^{(1)}_{\neq}\pa_x\omega_{\neq}\big)_{\neq}\|_{\tilde L^1_t\left([0,\widetilde T];\mathcal F L^1_kL^2_y\right)}+\|e^{-\frac{1}{2}C_0M^2\gamma t}\big(u^{(2)}_{\neq}\pa_y\omega_{\neq}\big)_{\neq}\|_{\tilde L^1_t\left([0,\widetilde T];\mathcal F L^1_kL^2_y\right)}\\
  &+\|e^{-\frac{1}{2}C_0M^2\gamma t}u^{(1)}_{0}\pa_x\omega_{\neq}\|_{\tilde L^1_t\left([0,\widetilde T];\mathcal F L^1_kL^2_y\right)}+\|e^{-\frac{1}{2}C_0M^2\gamma t}u^{(2)}_{\neq}\pa_y\omega_0\|_{\tilde L^1_t\left([0,\widetilde T];\mathcal F L^1_kL^2_y\right)}\\
  \le&e^{\frac{1}{2}C_0M^2\gamma \widetilde T}\|e^{-\frac{1}{2}C_0M^2\gamma t}u^{(1)}_{\neq}\|_{\tilde L^\infty_t\left([0,\widetilde T];\mathcal F L^1_kL^\infty_y\right)}\|e^{-\frac{1}{2}C_0M^2\gamma t}\pa_x\omega_{\neq}\|_{\tilde L^1_t\left([0,\widetilde T];\mathcal F L^1_kL^2_y\right)}\\
  &+e^{\frac{1}{2}C_0M^2\gamma \widetilde T}\|e^{-\frac{1}{2}C_0M^2\gamma t}u^{(2)}_{\neq}\|_{\tilde L^2_t\left([0,\widetilde T];\mathcal F L^1_kL^\infty_y\right)}\|e^{-\frac{1}{2}C_0M^2\gamma t}\pa_y\omega_{\neq}\|_{\tilde L^2_t\left([0,\widetilde T];\mathcal F L^1_kL^2_y\right)}\\
  &+e^{\frac{1}{2}C_0M^2\gamma \widetilde T}\|e^{-\frac{1}{2}C_0M^2\gamma t}u^{(1)}_0\|_{L^\infty_t\left([0,\widetilde T]; L^\infty_y\right)}\|e^{-\frac{1}{2}C_0M^2\gamma t}\pa_x\omega_{\neq}\|_{\tilde L^1_t\left([0,\widetilde T];\mathcal F L^1_kL^2_y\right)}\\
  &+e^{\frac{1}{2}C_0M^2\gamma \widetilde T}\|e^{-\frac{1}{2}C_0M^2\gamma t}u^{(2)}_{\neq}\|_{\tilde L^2_t\left([0,\widetilde T];\mathcal F L^1_kL^\infty_y\right)}\|e^{-\frac{1}{2}C_0M^2\gamma t}\pa_y\omega_0\|_{ L^2_t\left([0,\widetilde T]; L^2_y\right)}\\
  \le&16C_2^2e^{\frac{1}{2}C_0M^2\gamma \widetilde T}\nu^{-\frac{1}{2}}\|\omega_{in}\|_{\mathcal F L^1_kL^2_y}^2\le C\varepsilon_0C_2^2e^{\frac{1}{2}C_0M^2\gamma \widetilde T}\nu^{\delta_1}\|\omega_{in}\|_{\mathcal F L^1_kL^2_y}.
\end{align*}
A similar argument shows that
  \begin{align*}
    &\|e^{-\frac{1}{2}C_0M^2\gamma t}\big(u^{(1)}_{\neq}\pa_xu^{(1)}_{\neq}\big)_0\|_{L^1_t([0,\widetilde T];L^2_{y})}+\|e^{-\frac{1}{2}C_0M^2\gamma t}\big(u^{(2)}_{\neq}\pa_yu^{(1)}_{\neq}\big)_0\|_{L^1_t([0,\widetilde T];L^2_{y})}\\
    &+\|e^{-\frac{1}{2}C_0M^2\gamma t}\big(u^{(1)}_{\neq}\pa_x\omega_{\neq}\big)_0\|_{L^1_t([0,\widetilde T];L^2_{y})}+\|e^{-\frac{1}{2}C_0M^2\gamma t}\big(u^{(2)}_{\neq}\pa_y\omega_{\neq}\big)_0\|_{L^1_t([0,\widetilde T];L^2_{y})}\\
    \le&e^{\frac{1}{2}C_0M^2\gamma \widetilde T}\|e^{-\frac{1}{2}C_0M^2\gamma t}u^{(1)}_{\neq}\|_{\tilde L^\infty_t\left([0,\widetilde T];\mathcal F L^1_kL^\infty_y\right)}\|e^{-\frac{1}{2}C_0M^2\gamma t}\pa_xu^{(1)}_{\neq}\|_{\tilde L^1_t\left([0,\widetilde T];\mathcal F L^1_kL^2_y\right)}\\
  &+e^{\frac{1}{2}C_0M^2\gamma \widetilde T}\|e^{-\frac{1}{2}C_0M^2\gamma t}u^{(2)}_{\neq}\|_{\tilde L^2_t\left([0,\widetilde T];\mathcal F L^1_kL^\infty_y\right)}\|e^{-\frac{1}{2}C_0M^2\gamma t}\pa_yu^{(1)}_{\neq}\|_{\tilde L^2_t\left([0,\widetilde T];\mathcal F L^1_kL^2_y\right)}\\
  &+e^{\frac{1}{2}C_0M^2\gamma \widetilde T}\|e^{-\frac{1}{2}C_0M^2\gamma t}u^{(1)}_{\neq}\|_{\tilde L^\infty_t\left([0,\widetilde T];\mathcal F L^1_kL^\infty_y\right)}\|e^{-\frac{1}{2}C_0M^2\gamma t}\pa_x\omega_{\neq}\|_{\tilde L^1_t\left([0,\widetilde T];\mathcal F L^1_kL^2_y\right)}\\
  &+e^{\frac{1}{2}C_0M^2\gamma \widetilde T}\|e^{-\frac{1}{2}C_0M^2\gamma t}u^{(2)}_{\neq}\|_{\tilde L^2_t\left([0,\widetilde T];\mathcal F L^1_kL^\infty_y\right)}\|e^{-\frac{1}{2}C_0M^2\gamma t}\pa_y\omega_{\neq}\|_{\tilde L^2_t\left([0,\widetilde T];\mathcal F L^1_kL^2_y\right)}\\
  \le&C\varepsilon_0C_2^2e^{\frac{1}{2}C_0M^2\gamma \widetilde T}\nu^{\delta_1}\|\omega_{in}\|_{\mathcal F L^1_kL^2_y}.
  \end{align*}
  Here we use the fact that
  \begin{align*}
    &\|e^{-\frac{1}{2}C_0M^2\gamma t}\pa_xu^{(1)}_{\neq}\|_{\tilde L^1_t\left([0,\widetilde T];\mathcal F L^1_kL^2_y\right)}+\|e^{-\frac{1}{2}C_0M^2\gamma t}\pa_yu^{(1)}_{\neq}\|_{\tilde L^2_t\left([0,\widetilde T];\mathcal F L^1_kL^2_y\right)}\\
    \le& \|e^{-\frac{1}{2}C_0M^2\gamma t}\pa_x\omega_{\neq}\|_{\tilde L^1_t\left([0,\widetilde T];\mathcal F L^1_kL^2_y\right)}+\|e^{-\frac{1}{2}C_0M^2\gamma t}\pa_y\omega_{\neq}\|_{\tilde L^2_t\left([0,\widetilde T];\mathcal F L^1_kL^2_y\right)}.
  \end{align*}
\end{proof}

Now, we can prove Proposition \ref{prop-upper}.
\begin{proof}[Proof of Proposition \ref{prop-upper}]
We first give the estimates for $u^{(1)}_0$ and $\omega_0$. By using \eqref{eq:DH-2}, \eqref{eq:DH-3}, \eqref{eq-om-nonl-est2}-\eqref{eq-om-nonl-est5}, Lemma \ref{lem-est-nonl}, and the properties of heat kernel, we get for $0\le t\le \widetilde T$ that
\begin{align*}
  &\left\|e^{-\frac{1}{2}C_0M^2\gamma t}u^{(1)}_0(t)\right\|_{L^2_y}+\|e^{-\frac{1}{2}C_0M^2\gamma t}\omega_0(t)\|_{L^2_y}\\
  \le& \left\|e^{-\frac{1}{2}C_0M^2\gamma t}\int^t_0 e^{(t-s)\nu\pa_y^2}\Big(\big(u^{(1)}_{\neq}\pa_xu^{(1)}_{\neq}\big)_0+\big(u^{(2)}_{\neq}\pa_yu^{(1)}_{\neq}\big)_0\Big)(s,y) ds\right\|_{L^2_y} \\
  &+\left\|e^{-\frac{1}{2}C_0M^2\gamma t}\int^t_0 e^{(t-s)\nu\pa_y^2}\Big(\big(u^{(1)}_{\neq}\pa_x\omega_{\neq}\big)_0+\big(u^{(2)}_{\neq}\pa_y\omega_{\neq}\big)_0\Big)(s,y) ds\right\|_{L^2_y} \\
  \le&\|e^{-\frac{1}{2}C_0M^2\gamma t}\big(u^{(1)}_{\neq}\pa_xu^{(1)}_{\neq}\big)_0\|_{L^1_t([0,\widetilde T];L^2_{y})}+\|e^{-\frac{1}{2}C_0M^2\gamma t}\big(u^{(2)}_{\neq}\pa_yu^{(1)}_{\neq}\big)_0\|_{L^1_t([0,\widetilde T];L^2_{y})}\\
    &+\|e^{-\frac{1}{2}C_0M^2\gamma t}\big(u^{(1)}_{\neq}\pa_x\omega_{\neq}\big)_0\|_{L^1_t([0,\widetilde T];L^2_{y})}+\|e^{-\frac{1}{2}C_0M^2\gamma t}\big(u^{(2)}_{\neq}\pa_y\omega_{\neq}\big)_0\|_{L^1_t([0,\widetilde T];L^2_{y})}\\
  \le&CC_2^2e^{\frac{1}{2}C_0M^2\gamma \widetilde T}\varepsilon_0\nu^{\delta_1}\|\omega_{in}\|_{\mathcal F L^1_kL^2_y}.
\end{align*}
It follows from the Gagliardo-Nirenberg interpolation inequality that
\begin{align*}
  \|e^{-\frac{1}{2}C_0M^2\gamma t} u^{(1)}_0(t)\|_{L^\infty_t\left([0,\widetilde T]; L^\infty_y\right)}\le&C\|e^{-\frac{1}{2}C_0M^2\gamma t} u^{(1)}_0(t)\|^{\frac{1}{2}}_{L^\infty_t\left([0,\widetilde T]; L^2_y\right)}\|e^{-\frac{1}{2}C_0M^2\gamma t} \omega_0\|^{\frac{1}{2}}_{L^\infty_t\left([0,\widetilde T]; L^2_y\right)}\\
  \le&CC_2^2e^{\frac{1}{2}C_0M^2\gamma \widetilde T}\varepsilon_0\nu^{\delta_1}\|\omega_{in}\|_{\mathcal F L^1_kL^2_y}.
\end{align*}
Recalling the property of heat kernel that
\begin{align*}
  \|\pa_y e^{t\nu\pa_y^2}f\|_{L^2_t([0,\widetilde T];L^2_y)}\le C\nu^{-\frac{1}{2}}\|f\|_{L^2_y},
\end{align*}
we have
\begin{align*}
  &\|e^{-\frac{1}{2}C_0M^2\gamma t}\pa_y\omega_0\|_{ L^2_t\left([0,\widetilde T]; L^2_y\right)}\\
  \le& \left\|e^{-\frac{1}{2}C_0M^2\gamma t}\int^t_0 \pa_ye^{(t-s)\nu\pa_y^2}\Big(\big(u^{(1)}_{\neq}\pa_x\omega_{\neq}\big)_0+\big(u^{(2)}_{\neq}\pa_y\omega_{\neq}\big)_0\Big)(s,y) ds\right\|_{ L^2_t\left([0,\widetilde T]; L^2_y\right)}\\
  \le&\int^{\widetilde T}_0e^{-\frac{1}{2}C_0M^2\gamma s}\left(\int^{\widetilde T}_s \left\|\pa_ye^{(t-s)\nu\pa_y^2}\Big(\big(u^{(1)}_{\neq}\pa_x\omega_{\neq}\big)_0+\big(u^{(2)}_{\neq}\pa_y\omega_{\neq}\big)_0\Big)(s)\right\|_{L^2_y}^2dt\right)^{\frac{1}{2}}ds\\
  \le&C\nu^{-\frac{1}{2}}\int^{\widetilde T}_0e^{-\frac{1}{2}C_0M^2\gamma s}\left\|\Big(\big(u^{(1)}_{\neq}\pa_x\omega_{\neq}\big)_0+\big(u^{(2)}_{\neq}\pa_y\omega_{\neq}\big)_0\Big)(s)\right\|_{L^2_y}ds\\
  \le&CC_2^2e^{\frac{1}{2}C_0M^2\gamma \widetilde T}\varepsilon_0\nu^{\delta_1-\frac{1}{2}}\|\omega_{in}\|_{\mathcal F L^1_kL^2_y}.
\end{align*}
From \eqref{eq:DH-1} and Proposition \ref{prop-lin-upper}, we have
\begin{align*}
    &\|e^{-\frac{1}{2}C_0M^2\gamma t}\omega_{\neq}(t)\|_{\tilde L^\infty_t\left([0,\widetilde T];\mathcal F L^1_kL^2_y\right)}\\
  \le&\|e^{-\frac{1}{2}C_0M^2\gamma t}{S_\nu(t,0)}\omega_{in}\|_{\tilde L^\infty_t\left([0,\widetilde T];\mathcal F L^1_kL^2_y\right)}\\
  &+\left\|e^{-\frac{1}{2}C_0M^2\gamma t}\int^t_0 {S_\nu(t,s)}\big(\mathfrak L+\sum_{i=1,2,3}\mathcal N^{(i)}\big)(s)ds\right\|_{\tilde L^\infty_t\left([0,\widetilde T];\mathcal F L^1_kL^2_y\right)}\\
%  \le&C_1\|\omega_{in}\|_{\mathcal F L^1_kL^2_y}\\
 % &+\sum_{k\in\mathbb Z/0}\int^{\widetilde T}_0e^{-\frac{1}{2}C_0M^2\gamma s}\left\|e^{-\frac{1}{2}C_0M^2\gamma (t-s)} \left(\widetilde{S_\nu(t-s)\big(\mathcal L+\sum_{i=1,2,3}\mathcal N^{(i)}\big)}\right)_k(s)\right\|_{ L^\infty_t\left([s,\widetilde T];L^2_y\right)}ds\\
  \le&C_1\|\omega_{in}\|_{\mathcal F L^1_kL^2_y}+\|e^{-\frac{1}{2}C_0M^2\gamma t}\mathfrak L\|_{\tilde L^1_t\left([0,\widetilde T];\mathcal F L^1_kL^2_y\right)}+\sum_{i=1,2,3}\|e^{-\frac{1}{2}C_0M^2\gamma t}\mathcal N^{(i)}\|_{\tilde L^1_t\left([0,\widetilde T];\mathcal F L^1_kL^2_y\right)}\\
  \le&\Big(C_1+CC_2M\gamma^2\nu^{-\frac{1}{2}}+CC_2^2e^{\frac{1}{2}C_0M^2\gamma \widetilde T}\varepsilon_0\nu^{\delta_1}\Big) \|\omega_{in}\|_{\mathcal F L^1_kL^2_y}.
\end{align*}
By using similar arguments, we deduce that
\begin{align*}
    &\nu^{\frac{1}{2}}\|e^{-\frac{1}{2}C_0M^2\gamma t}\pa_y\omega_{\neq}\|_{\tilde L^2_t\left([0,\widetilde T];\mathcal F L^1_kL^2_y\right)}+\nu^{\frac{1}{2}}\|e^{-\frac{1}{2}C_0M^2\gamma t}\pa_x\omega_{\neq}\|_{\tilde L^1_t\left([0,\widetilde T];\mathcal F L^1_kL^2_y\right)}\\
  &+\|e^{-\frac{1}{2}C_0M^2\gamma t}u^{(1)}_{\neq}\|_{\tilde L^\infty_t\left([0,\widetilde T];\mathcal F L^1_kL^\infty_y\right)}+\|e^{-\frac{1}{2}C_0M^2\gamma t}u^{(2)}_{\neq}\|_{\tilde L^2_t\left([0,\widetilde T];\mathcal F L^1_kL^\infty_y\right)}\\
  \le&\Big(4C_1+CC_2M\gamma^2\nu^{-\frac{1}{2}}+CC_2^2e^{\frac{1}{2}C_0M^2\gamma \widetilde T}\varepsilon_0\nu^{\delta_1}\Big) \|\omega_{in}\|_{\mathcal F L^1_kL^2_y}.
\end{align*}

As $\widetilde T\le \varepsilon_1\gamma^{-1} \ln(\nu^{-1})$ with $\varepsilon_1=\frac{2\delta_0}{M\pi}$, then $e^{\frac{1}{2}C_0M^2\gamma \widetilde T}\le \nu^{-\frac{1}{\pi}C_0M\delta_0}$. By choosing 
\begin{equation}\label{eq-delta1}
  \delta_1=\frac{3}{\pi}C_0M\delta_0>0,
\end{equation}
we have
\begin{align*}
  e^{\frac{1}{2}C_0M^2\gamma \widetilde T}\nu^{\delta_1}\le \nu^{\frac{1}{3}\delta_1},
\end{align*}
and
\begin{align*}
  C_1+C_2M\gamma^2\nu^{-\frac{1}{2}}+C_2^2e^{\frac{1}{2}C_0M^2\gamma \widetilde T}\varepsilon_0\nu^{\delta_1}\le C_1+C_2M\nu^{\frac{1}{6}-\frac{2}{3}\delta_0}+C_2^2\varepsilon_0\nu^{\frac{2}{3}\delta_1}.
\end{align*}
By taking $C_2\ge CC_1$ big enough, and $\varepsilon_0$ small enough, we prove the Proposition \ref{prop-upper}.
\end{proof}
\subsection{Lower bound estimate to the nonlinear equation}
Now, we start to prove Theorem \ref{thm-low-exp}.
\begin{proof}[Proof of Theorem \ref{thm-low-exp}]
We study the solution $\omega(t,x,y)$ of nonlinear system \eqref{eq:b-perturbation} with initial data $\omega_{in}$ given in \eqref{eq-om-nonl} which satisfying $\int_{\mathbb T} \omega_{in}(x,y) d x=0$. The upper bound in \eqref{eq-est-nonlinear-viscous} follows from \eqref{eq-om-nonl-est1} and Proposition \ref{prop-upper}. By Proposition \ref{prop-lower}, Proposition \ref{prop-lin-upper}, Proposition \ref{prop-upper}, and Lemma \ref{lem-est-nonl}, we have the lower bound estimate
  \begin{align*}
  &\|\omega_{\pm1}(t)\|_{L^2_xL^2_y}\\
  =&\left\|P_{\pm1}\left({S_\nu(t,0)}\omega_{in}-\int^t_0 {S_\nu(t,s)}\big(\mathfrak L+\sum_{i=1,2,3}\mathcal N^{(i)}\big)(s)ds\right)\right\|_{L^2_xL^2_y}\\
  \ge&\|{S_\nu(t,0)}\omega_{in,\pm1}\|_{L^2_xL^2_y}-e^{\frac{1}{2}C_0M^2\gamma T}\left\|e^{-\frac{1}{2}C_0M^2\gamma t}\int^t_0 {S_\nu(t,s)}\big(\mathfrak L+\sum_{i=1,2,3}\mathcal N^{(i)}\big)(s)ds\right\|_{\tilde L^\infty_t\left([0,T];\mathcal F L^1_kL^2_y\right)}\\
  \ge&c_0e^{c_1\gamma t}\|\omega_{in}\|_{L^2_xL^2_y}\\
  &- e^{\frac{1}{2}C_0M^2\gamma T}\left(\|e^{-\frac{1}{2}C_0M^2\gamma t}\mathfrak L\|_{\tilde L^1_t\left([0,\widetilde T];\mathcal F L^1_kL^2_y\right)}+\sum_{i=1,2,3}\|e^{-\frac{1}{2}C_0M^2\gamma t}\mathcal N^{(i)}\|_{\tilde L^1_t\left([0,\widetilde T];\mathcal F L^1_kL^2_y\right)}\right)\\
  \ge&c_0e^{c_1\gamma t}\|\omega_{in}\|_{L^2_xL^2_y}-C\nu^{\frac{1}{3}\delta_1}\|\omega_{in}\|_{\mathcal F L^1_kL^2_y}
  \ge \frac{1}{2}c_0e^{c_1\gamma t}\|\omega_{in}\|_{H^1_xL^2_y}.
\end{align*}
The last inequality follows from the properties of our well-chosen initial data that
\begin{align*}
  \|\omega_{in}\|_{H^1_xL^2_y}=\|\omega_{in}\|_{L^2_xL^2_y}=\frac{1}{\sqrt\pi}\|\omega_{in}\|_{\mathcal F L^1_kL^2_y}.
\end{align*}
In particular, we have
\begin{align*}
  \|\omega(T)|\|_{L^2_xL^2_y}\ge \frac{c_0}{2\nu^{c_1\varepsilon_1}}\|\omega_{in}\|_{H^1_xL^2_y} \text{ for }{T=\varepsilon_1\gamma^{-1} \ln(\nu^{-1})}.
\end{align*}

To prove \eqref{eq-est-linear-viscous}, we study the solution of the linear system \eqref{eq:b-perturbation-linear} which satisfies
\begin{align*}
  \omega(t,x,y)=\omega_{\pm1}(t,x,y)={S_\nu(t,0)}\omega_{in}(x,y)-\int^t_0 {S_\nu(t,s)} \mathfrak L (s,x,y)ds.
\end{align*}
Note that for the linear system, $\mathfrak L $ has only $\pm1$ modes. Following the proof of Proposition \ref{prop-upper}, we have 
\begin{align*}
  \left\|\int^t_0 {S_\nu(t,s)} \mathfrak L (s,x,y)ds\right\|_{L^\infty_t([0,T];L^2_{x,y})}\le C\nu^{\frac{1}{3}\delta_1}\|\omega_{in}\|_{H^1_xL^2_y}.
\end{align*}
Finally, we get from Proposition \ref{prop-lower} and Proposition \ref{prop-lin-upper} that
  \begin{align*}
  Ce^{\frac{1}{2}C_0M^2\gamma t}\|\omega_{in}\|_{H^1_xL^2_y}\ge\|\omega(t)\|_{L^2_xL^2_y}\ge\frac{1}{2}c_0e^{c_1\gamma t}\|\omega_{in}\|_{H^1_xL^2_y}.
\end{align*}
\end{proof}

\section{Linear instability for the inviscid flow}
In this section, we study the inviscid flow and prove Theorem \ref{thm-invisid}. The proof is based on a combination of the dynamical approach and the classical ODE argument. 

After taking Fourier transform of \eqref{eq-Euler-shear} in $x$, we have
\begin{align*}
  \widetilde{\mathcal R_{M,\gamma}\omega}(k,y)=&\mathcal F_{x\to k}\left(b_0(y)\pa_x\omega-b_0''(y)\pa_x(\Delta)^{-1}\omega\right)(k,y)\\
  =&ik\left(b_0(y)-b_0''(y)(\pa_y^2-k^2)^{-1}\right)\tilde\omega(k,y).
\end{align*}
In this section, we only focus on the $k=1$ mode and study the operator
\begin{align*}
  \mathcal R=b_0(y)-b_0''(y)(\pa_y^2-1)^{-1}.
\end{align*}
If $\fc=\fc_r+i\fc_i$ is an eigenvalue of $\mathcal R$, then there exists $\omega_{\fc}(y)\in L^2_y$ such that
\begin{align*}
  \mathcal R\omega_{\fc}(y)=\fc\omega_{\fc}(y).
\end{align*}
Let $\fc=i\lambda$, and $\omega_\lambda(x,y)=e^{ix}\omega_{\fc}(y)$. Then, it holds that  
\begin{align*}
  \mathcal R_{M,\gamma}\omega_{\lambda}=\lambda\omega_{\lambda}.
\end{align*}

Therefore, to prove Theorem \ref{thm-invisid}, we prove the following proposition:
\begin{proposition}\label{prop-Rayleigh}
Let $M_0>0$ be big enough. For each $M\ge M_0$ there exists $0<\gamma_0=\gamma_0(M)$ such that for $0<\gamma\le\gamma_0$, the following properties hold for $\mathcal R$:
\begin{enumerate}
\item $\mathcal R$ has no embedded eigenvalue. 
\item $\mathcal R$ has eigenvalues, and the number of eigenvalues is finite. 
\item If $\fc=\fc_r+\fc_i$ is an eigenvalue of $\mathcal R$ then $\gamma\le|\fc_i|\le CM\gamma$ and $|\fc_r|< 4\gamma\sqrt{\ln\big(\ln(\gamma^{-1})\big)}$. 
\end{enumerate}
\end{proposition}
The first statement {in Proposition \ref{prop-Rayleigh} that $\mathcal R$} has no embedded eigenvalue is proved in Lemma \ref{lem-no-emb} by a classical ODE argument. 

Let us focus on the second statement {in Proposition \ref{prop-Rayleigh}}, namely, the existence of eigenvalues. 
We have the following lower bound estimates for the semi-group $e^{-it\mathcal R}$ by our new energy estimate presented in Section 3.
\begin{lemma}\label{cor-invisid}
There exists $\mathfrak w_{in}(y)\in L^2_y$ such that for $t\in[0,T]$ with $T=\varepsilon_1 \gamma^{-1}\ln(\gamma^{-1})$ and $\varepsilon_1>0$ small, it holds that
\begin{align}\label{eq-est-lower-invis}
  \|e^{-it\mathcal R}\mathfrak w_{in}\|_{L^2_{y}}\ge C^{-1}e^{C^{-1}\gamma t}\|\mathfrak w_{in}\|_{L^2_{y}},
\end{align}
where $C$ is independent of $\gamma$.
\end{lemma}
\begin{proof}
  We study the following system
 \begin{equation*}
    \partial_{t}  \omega(t,x,y)+ y\partial_{x} \omega(t,x,y)-\pa_y^2b_0(y)\pa_x(\Delta)^{-1}\omega(t,x,y)= -\big(b_0(y)-y\big)\pa_x\omega(t,x,y),
 \end{equation*}
with initial data $\omega_{in}(x,y)=2\cos(x)\mathfrak w_{in}(y)$. Note that 
\begin{align}\label{eq-qe-inviscid}
  \|\omega(t)\|_{L^2_{x,y}}=\|\pa_x\omega(t)\|_{L^2_{x,y}}=\sqrt2\|e^{-it\mathcal R}\mathfrak w_{in}\|_{L^2_{y}},\quad\|\omega_{in}\|_{L^2_{x,y}}=\sqrt2\|\mathfrak w_{in}\|_{L^2_{y}}.
\end{align}

By following the proof of Proposition \ref{prop-lin-upper}, we have for $s \le T$
\begin{align*}{
  \|e^{-\frac{1}{2} C_0 M^2\gamma (t-s)}S_0(t,s)\omega(s,x,y)\|_{L^\infty_t([s,T]; L^2_{x,y})}\le C \|\omega(s,x,y)\|_{L^2_{x,y}}.}
\end{align*}
We also have 
\begin{align*}
  \omega(t,x,y)={S_0(t,0)}\omega_{in}(x,y)-\int^t_0 {S_0(t,s)} \big(b_0(y)-y\big)\pa_x\omega(s,x,y)ds,
\end{align*}
and then
\begin{align*}
  &\|e^{-\frac{1}{2} C_0 M^2\gamma t}\omega\|_{L^\infty_t([0,T];L^2_{x,y})}\\
  \le & \|e^{-\frac{1}{2} C_0 M^2\gamma t}{S_0(t,0)}\omega_{in}\|_{L^\infty_t([0,T];L^2_{x,y})}+\left\|e^{-\frac{1}{2} C_0 M^2\gamma t}\int^t_0 {S_0(t,s)} \big(b_0(y)-y\big)\pa_x\omega(s,x,y)ds\right\|_{L^\infty_t([0,T];L^2_{x,y})}\\
  \le& C \|\omega_{in}\|_{L^2_{x,y}}+C \|b_0(y)-y\|_{L^\infty_y}T\|e^{-\frac{1}{2} C_0 M^2\gamma t}\pa_x\omega\|_{L^\infty_t([0,T];L^2_{x,y})}\\
  \le&C \|\omega_{in}\|_{L^2_{x,y}}+C M\gamma^2 T\|e^{-\frac{1}{2} C_0 M^2\gamma t}\omega\|_{L^\infty_t([0,T];L^2_{x,y})},
\end{align*}
which gives that 
\begin{align*}
  \|e^{-\frac{1}{2} C_0 M^2\gamma t}\omega\|_{L^\infty_t([0,T];L^2_{x,y})}\le C \|\omega_{in}\|_{L^2_{x,y}}. 
\end{align*}
From Proposition \ref{prop-lower-Euler}, we obtain that for well-chosen initial data, it holds that
\begin{align*}
  \|{S_0(t,0)}\omega_{in}\|_{L^2_{x,y}}\ge C^{-1} e^{c\gamma t}\|\omega_{in}\|_{L^2_{x,y}}, \text{ for } 0\le t\le T.
\end{align*}
Thus 
\begin{align*}
 \|\omega(t)\|_{L^2_{x,y}}
  \ge & \|{S_0(t,0)}\omega_{in}\|_{L^2_{x,y}}-e^{\frac{1}{2} C_0 M^2\gamma T}\left\|e^{-\frac{1}{2} C_0 M^2\gamma t}\int^t_0 {S_0(t,s)} \big(b_0(y)-y\big)\pa_x\omega(s)ds\right\|_{L^\infty_t([0,T];L^2_{x,y})}\\
  \ge & C^{-1} e^{c\gamma t}\|\omega_{in}\|_{L^2_{x,y}}-C M\gamma^2 Te^{\frac{1}{2} C_0 M^2\gamma T}\|\omega_{in}\|_{L^2_{x,y}}\ge C^{-1} e^{c\gamma t}\|\omega_{in}\|_{L^2_{x,y}},
\end{align*}
which together with \eqref{eq-qe-inviscid} gives the corollary.
\end{proof}
Normally, for self-adjoint operators, the exponential growth given in the lower bound estimate of the semi-group $e^{-it\mathcal R}$ \eqref{eq-est-lower-invis} implies the existence of unstable eigenvalues of $\mathcal R$ for $\gamma$ small enough. Here $\mathcal R$ is not self-adjoint. To prove the existence of an unstable eigenvalue, we will use a contradiction argument and an upper bound estimate of the semi-group $e^{-it\mathcal R}$. Indeed, under the assumption that $\mathcal R$ has no eigenvalues, we will prove that for $\g$ small
\begin{align}\label{eq-up-bound-R}
  \|e^{-it\mathcal R}\mathfrak w_{in}\|_{L^2_{y}}\le C\|\mathfrak w_{in}\|_{L^2_{y}},
\end{align}
where $C$ is a constant independent of $\gamma$. By taking $\gamma$ small, the upper bound \eqref{eq-up-bound-R} and the lower bound estimates \eqref{eq-est-lower-invis} of the semi-group lead to a contradiction.

The estimate \eqref{eq-up-bound-R} is obtained by the representation formula and the resolvent estimate. For the bounded domain case, namely, if $\Omega$ is replaced by $\mathbb T\times[0,1]$, the upper bound estimate is well studied in \cite{WeiZhangZhao2018}, for the unbounded domain case, the result \eqref{eq-up-bound-R} is new. Here we also need uniformity in $\gamma$.

Let us admit \eqref{eq-up-bound-R} which is given in Section 5.4 and finish the proof of the existence of eigenvalues. 
%We prove the existence of unstable eigenvalues by a contradiction argument. 
Suppose that $\mathcal R$ has no eigenvalues, then by Lemma \ref{lem-up-bound-noeigen} we get the upper bound estimate \eqref{eq-up-bound-R} which contradicts Lemma \ref{cor-invisid} by taking $\gamma$ small enough. 

The third statement {in Proposition \ref{prop-Rayleigh}} follows from Lemma \ref{lem-E1E2} and Lemma \ref{lem-eigen-E3}. Moreover, Corollary \ref{cor-finit-eigen} gives the fact that $\mathcal R$ has a finite number of eigenvalues. The rest parts of this section are mainly to prove \eqref{eq-up-bound-R}.
%The organization of the rest of this section is as follows. In Section \ref{sec-Ray}, we introduce some basic properties of a smooth solution to the homogeneous Rayleigh equation. In Section 5.2, we prove that $\mathcal R$ has no embedded eigenvalue. In Section 5.3, we give possible locations of the eigenvalues and show that there are at most finite numbers of eigenvalues. In Section 5.4, we prove the upper bound \eqref{eq-up-bound-R}. 

\subsection{Homogeneous Rayleigh equation}\label{sec-Ray}
Let $\fc=\fc_r+i\fc_i\in \mathbb C$ be an eigenvalue of $\mathcal R$, then there exists $0\not\equiv\psi\in H^2_y$  solution of the homogeneous Rayleigh equation
\begin{align}\label{eq-Rayleigh}
  \psi''(y)-\psi(y)-\frac{b_0''(y)}{b_0(y)-\fc}\psi(y)=0.
\end{align}

As there is a singularity in the coefficients when $\fc\in \mathbb R$, we first construct a regular solution $\phi(y,\fc)$ (not necessary in $H^2_y$) of \eqref{eq-Rayleigh}.
\begin{proposition}\label{prop-phi}
  There exist $0<\varepsilon_3\le 1$ and $0<C_4\le1$ which depend only on the upper and lower bounds of $b_0'(y)$, such that for $\fc\in\mathbb C$  with $0\leq |\fc_i|\le \varepsilon_3$, \eqref{eq-Rayleigh} has a regular solution
  \begin{align*}
    \phi(y,\fc)=\big(b_0(y)-\fc\big)\phi_1(y,\fc_r)\phi_2(y,\fc),
  \end{align*}
  which satisfies $\phi(y_\fc,\fc)={-i\fc_i}$ and $\phi'(y_\fc,\fc)=b_0'(y_\fc)$. Here $y_\fc=b_0^{-1}(\fc_r)$, $\phi_1(y,\fc_r)$ is a real function that solves
  \begin{equation}\label{eq-phi1}
    \left\{
      \begin{array}{ll}
        \pa_y\left((b_0(y)-\fc_r)^2\phi_1'(y,\fc_r)\right)-\phi_1(y,\fc_r)(b_0(y)-\fc_r)^2=0,\\
        \phi_1(y_\fc,\fc_r)=1,\quad\phi_1'(y_\fc,\fc_r)=0,
      \end{array}
    \right.
  \end{equation}
and $\phi_2(y,\fc)$ solves
\begin{equation}\label{eq-phi2}
  \left\{
    \begin{array}{ll}
      \pa_y\Big((b_0(y)-\fc)^2\phi_1^2(y,\fc_r)\phi_2'(y,\fc)\Big)+\frac{2i{\fc_i} b_0'(y)(b_0(y)-\fc)}{b_0(y)-\fc_r}\phi_1(y,\fc_r)\phi_1'(y,\fc_r)\phi_2(y,\fc)=0,\\
      \phi_2(y_\fc,\fc)=1,\quad\phi_2'(y_\fc,\fc)=0.
    \end{array}
  \right.
\end{equation}
It holds that
  \begin{equation}\label{eq-phi1-est}
    \begin{aligned}    
      \phi_1(y,\fc_r)\ge1,\quad Ce^{C_4(|y-y_\fc|)}\le\phi_1(y,\fc_r)\le e^{|y-y_\fc|},\  |\phi_1'(y,\fc_r)|\le\phi_1(y,\fc_r),\ \forall y\in\mathbb R,\\
     |\phi_1(y,\fc_r)-1|\le C|y-y_\fc|^2\text{ for } |y-y_\fc|\le1,\  |\phi_1'(y,\fc_r)|\ge C_4\phi_1(y,\fc_r)\text{ for } |y-y_\fc|\ge1,\\
     \frac{\phi_1(y,\fc_r)}{\phi_1(y',\fc_r)}\le Ce^{-C_4(y-y')}\text{ for }y'\le y\le y_\fc,\quad \frac{\phi_1(y,\fc_r)}{\phi_1(y',\fc_r)}\le Ce^{-C_4(y'-y)}\text{ for }y'\ge y\ge y_\fc,
    \end{aligned}
  \end{equation}
    \begin{equation}\label{eq-phi2-est}
    \begin{aligned}    
      \left|\phi_2(y,\fc)-1\right|\le C\min(|\fc_i|,|\fc_i||y-y_\fc|,|y-y_\fc|^2),\ \forall y\in\mathbb R,\\
       \left|\phi_2'(y,\fc)\right|\le C\min(|\fc_i|, |y-y_\fc|),\quad \|\phi_2''(y,\fc)\|_{L^\infty_y}\le C,\ \forall y\in\mathbb R,
    \end{aligned}
  \end{equation}
  and
    \begin{align}\label{eq-phi-est}
    C^{-1}(|y-y_\fc|+|\fc_i|)e^{C_4|y-y_\fc|}\le|\phi(y,\fc)|\le C(|y-y_\fc|+|\fc_i|)e^{|y-y_\fc|}\text{ for }y\in\mathbb R,
  \end{align}
  where $C>0$ is a constant independent of $\fc$ and $\gamma$.
\end{proposition}
It is easy to check that 
\beno
\phi(y,\fc)\quad &\text{and}&\quad \varphi^{+}(y,\fc),\\
\phi(y,\fc)\quad &\text{and}&\quad \varphi^{-}(y,\fc),
\eeno
where
\begin{align*}
  \varphi^\pm(y,\fc)=&\phi(y,\fc)\int^y_{\pm\infty}\frac{1}{\phi^2(y',\fc)}dy',
\end{align*}
are two fundamental sets of solutions of \eqref{eq-Rayleigh}. 
Therefore, if $\psi(y,\fc)\in H^2_y$ is a solution of \eqref{eq-Rayleigh}, then $\psi(y,\fc)$ has the following form
\begin{align}\label{eq-psi-eigenfunction}
  \psi(y,\fc)=a_1^-\phi(y,\fc)+a_2^-\varphi^-(y,\fc)=a_1^+\phi(y,\fc)+a_2^+\varphi^+(y,\fc),
\end{align}
where $a_1^\pm$, $a_2^\pm$ are constants. Note that by Proposition \ref{prop-phi},  $\varphi^\pm$ are well defined for $|\fc_i|\le \varepsilon_3$.

We give a criterion for whether $\fc$ is an eigenvalue of $\mathcal R$.
\begin{lemma}\label{lem-eigen}
  Let
  \begin{align*}
  \mathcal D(\fc)=&\int^{+\infty}_{-\infty}\frac{1}{\phi^2(y',\fc)}dy'.
  \end{align*}
Then, a number $\fc\in\mathbb C$  with $0<|\fc_i|\le \varepsilon_3$ is an eigenvalue of $\mathcal R$ if and only if $\mathcal D(\fc)=0$.
\end{lemma}
The criterion function $\mathcal D(\fc)$ can be extended to $\mathbb R$.
\begin{lemma}\label{lem-lim-eigen}
  It holds that
  \begin{align*}
    \lim_{\fc_i\to0\pm}\mathcal D(\fc)=\mathcal J_1(\fc_r)\mp i\mathcal J_2(\fc_r),
  \end{align*}
  where
  \begin{align*}
    \mathcal J_1(\fc_r)&=\frac{1}{b_0'(y_\fc)}\Pi_1(\fc_r)+\Pi_2(\fc_r),\quad\mathcal J_2(\fc_r)=\pi\frac{b_0''(y_\fc)}{{b_0'}^3(y_\fc)}.
  \end{align*}
  and
  \begin{align*}
  \Pi_1(\fc_r)= \text{P.V.}\int  \frac{b_0'(y_\fc)-b_0'(y)}{(b_0(y)-\fc_r)^2} dy,\quad \Pi_2(\fc)= \int_{-\infty}^{+\infty} \frac{1}{(b_0(y)-\fc_r)^2}\left(\frac{1}{\phi_1^2(y,\fc_r)}-1\right) dy.
\end{align*}
\end{lemma}

Recall that Ran $b_0(y)=\mathbb R$ is the continuous spectrum of $\mathcal R$. If $\fc\in\mathbb R$ is an eigenvalue of $\mathcal R$, we call $\fc$ an embedded eigenvalue.
\begin{lemma}\label{lem-iff-emb}
A number $\fc_r\in \mathbb R$ is an embedded eigenvalue of $\mathcal R$ if and only if 
\begin{align*}
  \mathcal J_1^2(\fc_r)+\mathcal J_2^2(\fc_r)=0.
\end{align*}
\end{lemma}

The proofs of the above proposition and lemmas can be found in Appendix \ref{appendix-B}.
\subsection{$\mathcal R$ has no embedded eigenvalues}
In this subsection, we show that there is no embedded eigenvalue in $\mathbb R$.
\begin{lemma}\label{lem-no-emb}
  Under the assumptions on $M$ and $\gamma$ in Preposition \ref{prop-Rayleigh}, the operator $\mathcal R$ does not have any embedded eigenvalue.
\end{lemma}
\begin{proof}
By the definition of $b_0(y)$, $b_0''(y)=0$ only at $y=0$, thus $\mathcal J_2(y)\neq 0$ if $y\neq 0$. By Lemma \ref{lem-iff-emb}, to prove Lemma \ref{lem-no-emb}, it suffices to show that $\mathcal J_1(0)\neq0$.

From Lemma \ref{lem-lim-eigen}, we have $\lim\limits_{\fc_i\to0}\mathcal D(i\fc_i)=\Pi_1(0)+\Pi_2(0)$. And it follows from \eqref{eq-lim-fci} that
\begin{align*}
  \Pi_1(0)=&-\mathcal H\big(\pa_v^2(b_0^{-1})\big)(0)=-\int_{\mathbb R} \frac{1}{v}\frac{b_0''(b_0^{-1}(v))}{\big(b_0'(b_0^{-1}(v))\big)^3} dv\\
  =&-\int_{\mathbb R} \frac{1}{b_0(y)}\frac{b_0''(y)}{\big(b_0'(y)\big)^2} dy=4\sqrt\pi M \int_{\mathbb R} \frac{y}{b_0(y)}\frac{\frac{1}{\gamma}e^{-\frac{y^2}{\gamma^2}}}{\big(b_0'(y)\big)^2} dy
  \ge CM,
\end{align*}
where we use the change of coordinate $y=b_0^{-1}(v)$. From \eqref{eq-phi1-est}, one can easily check that
\begin{align*}
  |\Pi_2(0)|=\left|\int^{+\infty}_{-\infty}\frac{1}{b_0^2(y)}\left(\frac{1}{\phi^2_1(y,0)}-1\right)dy\right|\le C.
\end{align*}
Finally, taking $M$ big enough, we have that
\begin{align*}
  \mathcal J_1(0)\ge \Pi_1(0)-|\Pi_2(0)|\ge CM>0,
\end{align*}
and there is no embedded eigenvalue of $\mathcal R$.
\end{proof}

\subsection{Possible locations of the eigenvalues}
In this subsection, we show that $\mathcal R$ has at most a finite number of eigenvalues, moreover, they are only in $E=\mathbb C\setminus \big(E_1\cup E_2\cup E_3\big)$, where
\begin{align*}
  E_1=\big\{\fc\in \mathbb C\big||\fc_r|\ge 4\gamma\sqrt{\ln\big(\ln(\gamma^{-1})\big)}, 0<|\fc_i|<8\sqrt\pi M\gamma\big\},\\
   E_2=\big\{\fc\in \mathbb C\big||\fc_i|\ge 8\sqrt\pi M\gamma\big\},\quad E_3=\big\{\fc\in \mathbb C\big||\fc_r|\le 4\gamma\sqrt{\ln\big(\ln(\gamma^{-1})\big)}, 0<|\fc_i|\le \gamma\big\}.
\end{align*}  
We first show that $\mathcal R$ has no eigenvalues in $E_1\cup E_2$  by an energetic argument of the Rayleigh type. 
\begin{lemma}\label{lem-E1E2}
The operator $\mathcal R$ has no eigenvalues in $E_1\cup E_2$.
\end{lemma}
\begin{proof}
  We first assume that $\fc\in E_1$ is an eigenvalue. Taking the inner product of \eqref{eq-Rayleigh} with $\psi$, we have
  \beno
 \|\psi\|_{H^1_y}^2= \|\psi'\|_{L^2}^2+\|\psi\|_{L^2}^2=-\int_{\mathbb R}\frac{b_0''}{b_0-\fc}|\psi|^2 d y.
  \eeno
We write 
\begin{align*}
  \int_{\mathbb R}\frac{b_0''}{b_0-\fc}|\psi|^2 d y=\int_{\mathbb R\setminus D}\frac{b_0''}{b_0-\fc}|\psi|^2 d y+\int_{D}\frac{b_0''}{b_0-\fc}|\psi|^2 d y\eqdef I+II,
\end{align*}
where $D=[\fc_r-\min(\frac{1}{2}|\fc_r|,\frac{1}{2}),\fc_r+\min(\frac{1}{2}|\fc_r|,\frac{1}{2})]$.

It is easy to check that $\|\psi\|_{L^\infty_y}^2\le \frac{1}{2}\|\psi\|_{H^1_y}^2$. For $y\in \mathbb R\setminus D$, it is clear that $|b_0-\fc|\ge \min(\frac{1}{4}|\fc_r|,\frac{1}{4})\ge \min(\gamma\sqrt{\ln\big(\ln(\gamma^{-1})\big)},\frac{1}{4})$.  Therefore,
\begin{align*}
  |I|\le \frac{\int_{\mathbb R}4\sqrt\pi M \frac{|y|}{\gamma}e^{-\frac{y^2}{\gamma^2}}d y}{\min(\frac{1}{4}|\fc_r|,\frac{1}{4})}\|\psi\|_{L^\infty_y}^2\le   \frac{CM}{\sqrt{\ln\big(\ln(\gamma^{-1})\big)}} \|\psi\|_{L^\infty_y}^2\le \frac{1}{4}\|\psi\|_{H^1_y}^2.
\end{align*}
For $II$, by integration by parts, we have
\begin{align*}
  II=&\int_{D}\frac{b_0''}{b_0-\fc}|\psi|^2 d y=\int_{D}\frac{b_0''}{b_0'}|\psi|^2\pa_y\ln(b_0-\fc)  d y\\
  =&-\int_{D}\frac{b_0'''b_0'-(b_0'')^2}{(b_0')^2}|\psi|^2\ln(b_0-\fc)+ 2\frac{b_0''}{b_0'}\Re(\psi\bar\psi')\ln(b_0-\fc) d y+\frac{b_0''}{b_0'}|\psi|^2\ln(b_0-\fc)\Big|^{\fc_r+\min(\frac{1}{2}|\fc_r|,\frac{1}{2})}_{\fc_r-\min(\frac{1}{2}|\fc_r|,\frac{1}{2})}.
\end{align*}
Then we have
\begin{align*}
  |II|\le& C \big(\|b_0'''\|_{L^\infty_y(D)}+\|b_0''\|^2_{L^\infty_y(D)}\big)\|\psi\|_{L^\infty_y}^2 \|\ln(b_0-\fc)\|_{L^1_y(D)}\\
  &+C\|b_0''\|_{L^\infty_y(D)}\|\psi\|_{L^\infty_y}\|\psi'\|_{L^2_y}\|\ln(b_0-\fc)\|_{L^2_y(D)}\\
  &+C\|b_0''\|_{L^\infty_y(D)}\|\psi\|_{L^\infty_y}^2\Big(|\ln\big(b_0(\fc_r+\min(\frac{1}{2}|\fc_r|,\frac{1}{2}))-\fc\big)|+|\ln\big(b_0(\fc_r-\min(\frac{1}{2}|\fc_r|,\frac{1}{2}))-\fc\big)|\Big).
\end{align*}
Recall the definition of $D$.  We have for $p=1,2$ that
\begin{align*}
  &\int_{D}|\ln(b_0-\fc)|^p d y\le C\int_{|z|\le \frac{3}{2}\min(\frac{1}{2}|\fc_r|,\frac{1}{2})} |\ln(z)|^pdz\\
  \le&C\int_{|z|\le \frac{3}{2}\min(\frac{1}{2}|\fc_r|,\frac{1}{2})} |y|^p e^ydy\le C\int_{-\infty}^{\ln\big(\frac{3}{2}\min(\frac{1}{2}|\fc_r|,\frac{1}{2})\big)}|y|^pe^{y}dy.
\end{align*}
Then we have
\begin{align*}
  \int_{D}|\ln(b_0-\fc)|^p d y\le C\fc_r|\ln(\frac{3}{4}\fc_r)|^p \text{ for }|\fc_r|\le1,\quad\int_{D}|\ln(b_0-\fc)|^p d y\le C\text{ for }|\fc_r|>1.
\end{align*}
For the case $2\gamma\sqrt{\ln(\gamma^{-1})}\ge |\fc_r|\ge4\gamma\sqrt{\ln\big(\ln(\gamma^{-1})\big)}$, we have
\begin{align*}
   |b_0''(y)|\le C M \frac{\sqrt{\ln\big(\ln(\gamma^{-1})\big)}}{\big(\ln(\gamma^{-1})\big)^4},\quad |b_0'''(y)|\le C M \frac{\ln\big(\ln(\gamma^{-1})\big)}{\gamma\big(\ln(\gamma^{-1})\big)^4},
\end{align*} 
and
\begin{align*}
  |\ln\big(b_0(\fc_r\pm\min(\frac{1}{2}|\fc_r|,\frac{1}{2}))-\fc\big)|\le C \ln(\gamma^{-1}),\quad  \|\ln(b_0-\fc)\|_{L^p_y(D)}^p\le C\gamma \big(\ln(\gamma^{-1}))^{p+1}\text{ for }p=1,2.
\end{align*}
For $|\fc_r|\ge 2\gamma\sqrt{\ln(\gamma^{-1})}$, we have
\begin{align*}
   |b_0''(y)|\le C M \gamma^3,\quad |b_0'''(y)|\le C M\gamma^2,
\end{align*} 
and
\begin{align*}
  |\ln\big(b_0(\fc_r\pm\min(\frac{1}{2}|\fc_r|,\frac{1}{2}))-\fc\big)|\le C \ln(\gamma^{-1}),\quad  \|\ln(b_0-\fc)\|_{L^p_y(D)}^p\le C\text{ for }p=1,2.
\end{align*}
Combing the above estimates and the Gagliardo-Nirenberg interpolation inequality, we have
\begin{align*}
  |II|\le \frac{1}{4}\|\psi\|_{H_y^1}^2,
\end{align*}
and
\begin{align*}
   \|\psi\|_{H_y^1}^2\le \frac{1}{2} \|\psi\|_{H_y^1}^2,
\end{align*}
which means that $\psi\equiv0$ and $\fc\in E_1$ could not be an eigenvalue.

Next, we assume that $\fc\in E_2$ is an eigenvalue. As $|\fc_i|\ge 8\sqrt\pi M\gamma$, we have
\begin{align*}
  \int_{\mathbb R}\frac{b_0''}{b_0-\fc}|\psi|^2 d y\le&\frac{1}{|\fc_i|}\int_{\mathbb R}|b_0''| dy \|\psi\|_{L^\infty_y}^2\le \frac{2\sqrt\pi M\gamma \|\psi\|_{H^1_y}^2}{8\sqrt\pi M\gamma}\le \frac{1}{2}\|\psi\|_{H^1_y}^2,
\end{align*}
which also implies $\psi\equiv0$ and $\fc\in E_2$ can not be an eigenvalue.  
\end{proof}
Next, we show that $\mathcal R$ has no eigenvalues in $E_3$. We use an ODE argument and study the Wronskian $\mathcal D(c)$.
\begin{lemma}\label{lem-eigen-E3}
  Under the assumptions on $M$ and $\gamma$ in Preposition \ref{prop-Rayleigh}, the operator $\mathcal R$ has no eigenvalues in $E_3$.
\end{lemma}
\begin{proof}
From Lemma \ref{lem-eigen}, it suffices to show that $\mathcal D(\fc)\neq0$ for $\fc\in E_3$. We prove that, the real part of $\mathcal D(\fc)\neq0$ for $|\fc_r|\le \frac{1}{4}\gamma$ and $|\fc_i|\le \gamma$, and the imaginary part of $\mathcal D(\fc)\neq0$ for $\frac{1}{4}\gamma\le|\fc_r|\le 4\gamma\sqrt{\ln\big(\ln(\gamma^{-1})\big)}$ and $|\fc_i|\le \gamma$.

We write
\begin{align*}
  \mathcal D(\fc)=&\int^{+\infty}_{-\infty}\frac{1}{(b_0(y)-\fc)^2}dy+\int^{+\infty}_{-\infty}\frac{1}{(b_0(y)-\fc)^2}\left(\frac{1}{\phi^2_1(y,\fc_r)\phi^2_2(y,\fc)}-1\right)dy=I+II.
\end{align*}
For $II$, we have
\begin{align*}
  &\int^{+\infty}_{-\infty}\frac{1}{(b_0(y)-\fc)^2}\left(\frac{1}{\phi^2_1(y,\fc)\phi^2_2(y,\fc)}-1\right)dy\\
  =&\int^{+\infty}_{-\infty}\frac{(b_0(y)-\fc_r)^2-{\fc_i}^2+2i{\fc_i}(b_0(y)-\fc_r)}{\big((b_0(y)-\fc_r)^2+{\fc_i}^2\big)^2}\\
  &\qquad\qquad\qquad\cdot\left(\frac{(\phi_{2,r}^2+\phi_{2,i}^2)\big(1-\phi^2_1(\phi_{2,r}^2+\phi_{2,i}^2)\big)-2\phi_{2,i}^2}{\phi^2_1(\phi_{2,r}^2+\phi_{2,i}^2)^2}-\frac{2i\phi_{2,r}\phi_{2,i}}{\phi^2_1(\phi_{2,r}^2+\phi_{2,i}^2)^2}\right)dy'\\
  =&\int^{+\infty}_{-\infty}\frac{(b_0(y)-\fc_r)^2-{\fc_i}^2}{\big((b_0(y)-\fc_r)^2+{\fc_i}^2\big)^2}\frac{(\phi_{2,r}^2+\phi_{2,i}^2)\big(1-\phi^2_1(\phi_{2,r}^2+\phi_{2,i}^2)\big)-2\phi_{2,i}^2}{\phi^2_1(\phi_{2,r}^2+\phi_{2,i}^2)^2}\\
  &\qquad\qquad\qquad\qquad\qquad\qquad\qquad+4\frac{{\fc_i}(b_0(y)-\fc_r)}{\big((b_0(y)-\fc_r)^2+{\fc_i}^2\big)^2}\frac{\phi_{2,r}\phi_{2,i}}{\phi^2_1(\phi_{2,r}^2+\phi_{2,i}^2)^2}dy'\\
  &+2i\int^{+\infty}_{-\infty}\frac{{\fc_i}(b_0(y)-\fc_r)\Big((\phi_{2,r}^2+\phi_{2,i}^2)\big(1-\phi^2_1(\phi_{2,r}^2+\phi_{2,i}^2)\big)-2\phi_{2,i}^2\Big)}{\big((b_0(y)-\fc_r)^2+{\fc_i}^2\big)^2\phi^2_1(\phi_{2,r}^2+\phi_{2,i}^2)^2}dy'\\
  &\qquad\qquad\qquad\qquad\qquad\qquad\qquad-\frac{\big((b_0(y)-\fc_r)^2-{\fc_i}^2\big)\phi_{2,r}\phi_{2,i}}{\big((b_0(y)-\fc_r)^2+{\fc_i}^2\big)^2\phi^2_1(\phi_{2,r}^2+\phi_{2,i}^2)^2}dy'\\
  \eqdef&II_r+iII_i.
\end{align*}
From \eqref{eq-phi1-est} and \eqref{eq-phi2-est}, we have 
\begin{align*}
  |\phi_{1}(y)-1|\le C|y-y_\fc|^2,\quad |\phi_{2,r}(y)-1|\le C|y-y_\fc|^2,\text{ for }|y-y_\fc|\le1,\\
|\phi_{2,r}(y)-1|\le C|\fc_i|\min(|y-y_\fc|,1),\quad |\phi_{2,i}(y)|\le C|\fc_i|\min(|y-y_\fc|,1),\text{ for }y\in\mathbb R.
\end{align*}
Then we have for $|y-y_\fc|\le1$ that
\begin{align*}
  \big|1-\phi^2_1(\phi_{2,r}^2+\phi_{2,i}^2)\big|\le C|y-y_\fc|^2.
\end{align*}
And then,
\begin{align*}
  |II_r|\le&C\int_{[y_\fc-1,y_\fc+1]}\frac{|y-y_\fc|^2+|\fc_i||y-y_\fc|}{(b_0(y)-\fc_r)^2+{\fc_i}^2}dy+C\int_{\mathbb R\setminus [y_\fc-1,y_\fc+1]}\frac{1}{|y-y_\fc|^2}dy\le C,
\end{align*}
and
\begin{align*}
  |II_i|\le&C\int_{[y_\fc-1,y_\fc+1]}\frac{|\fc_i||y-y_\fc|}{(b_0(y)-\fc_r)^2+{\fc_i}^2}dy+C\int_{\mathbb R\setminus [y_\fc-1,y_\fc+1]}\frac{|\fc_i|}{|y-y_\fc|^2}dy\\
  \le& C|\fc_i\ln(\fc_i)|\le C\gamma\ln(\gamma^{-1}),
\end{align*}
where the constant $C$ does not depend on  $M$.

Recall that $b_0(y)$ is a strictly monotonic function. Let $v=b_0(y)$,  we have
\begin{align*}
  I=&\int^{+\infty}_{-\infty}\frac{1}{(b_0(y)-\fc)^2}dy=\int^{+\infty}_{-\infty}\frac{\pa_v\big(b_0^{-1}\big)(v)}{(v-\fc)^2}dv=\int^{+\infty}_{-\infty}\frac{\pa_v^2\big(b_0^{-1}\big)(v)}{v-\fc}dv\\
=&\int^{+\infty}_{-\infty} \frac{(v-\fc_r)\pa_v^2\big(b_0^{-1}\big)(v)}{(v-\fc_r)^2+{\fc_i}^2}dv +\int^{+\infty}_{-\infty} \frac{i{\fc_i}\pa_v^2\big(b_0^{-1}\big)(v)}{(v-\fc_r)^2+{\fc_i}^2}dv\eqdef I_r+iI_i.
\end{align*}
Direct calculations show that
\begin{align*}
  \pa_v\big(b_0^{-1}\big)(v)=\frac{1}{b_0'\big(b_0^{-1}(v)\big)},\quad\pa_v^2\big(b_0^{-1}\big)(v)=-\frac{b_0''\big(b_0^{-1}(v)\big)}{\big(b_0'\big(b_0^{-1}(v)\big)\big)^3}.
\end{align*}
Then we have
\begin{align*}
  I_r=&-\int^{+\infty}_{-\infty} \frac{\big(b_0(y)-\fc_r\big)\frac{b_0''(y)}{\big(b_0'(y)\big)^2}}{\big(b_0(y)-\fc_r\big)^2+{\fc_i}^2}dy=4\sqrt\pi M \int^{+\infty}_{-\infty} \frac{1}{\big(b_0'(y)\big)^2}\frac{\big(b_0(y)-\fc_r\big)\frac{y}{\gamma}e^{-\frac{y^2}{\gamma^2}}}{\big(b_0(y)-\fc_r\big)^2+{\fc_i}^2}dy\\
  =&4\sqrt\pi M\left(\int_{\mathbb R\setminus[0,2y_\fc]}\dots dy+\int^{y_\fc}_{0}\dots dy+\int^{2y_\fc}_{y_\fc}\dots dy\right)\\
   \eqdef&4\sqrt\pi M\left(K_1+K_2+K_3\right).
\end{align*}
We remark that $\|b_0'(y)-1\|_{L^\infty_y}\le CM\gamma$, so we have $|y_\fc-\fc_r|\le CM\gamma^2$, and $|y_\fc|\le\frac{\sqrt2}{4}\gamma$. Then, it is clear that $K_1>0$, $K_2<0$, and $K_3>0$. As $\frac{y}{\gamma}e^{-\frac{y^2}{\gamma^2}}$ grows on $[0,\frac{\sqrt2}{2}\gamma]$, so 
\begin{align*}
  \min_{y\in[y_\fc,2y_\fc]}\frac{y}{\gamma}e^{-\frac{y^2}{\gamma^2}}\ge \max_{y\in[0,y_\fc]}\frac{y}{\gamma}e^{-\frac{y^2}{\gamma^2}}.
\end{align*}
Then we have $K_2+K_3\ge -CM\gamma$. Indeed, we have
\begin{align*}
  K_2+K_3\ge&\int^{y_\fc}_0\frac{1}{\big(b_0'(y)\big)^2}\frac{\big(b_0(y)-\fc_r\big)}{\big(b_0(y)-\fc_r\big)^2+{\fc_i}^2}+\frac{1}{\big(b_0'(2y_\fc-y)\big)^2}\frac{\big(b_0(2y_\fc-y)-\fc_r\big)}{\big(b_0(2y_\fc-y)-\fc_r\big)^2+{\fc_i}^2}dy\\
  =&\int^{\max(y_\fc-2{|\fc_i|},0)}_{0}\dots dy+\int^{y_\fc}_{\max(y_\fc-2{|\fc_i|},0)}\dots dy\eqdef J_1+J_2.
\end{align*}
For $0\le y_1\le y_2$, we have $1\le b_0'(y_2)\le b_0'(y_1)\le 1+CM\gamma$, and then
\begin{align*}
  &\frac{1}{\big(b_0'(y)\big)^2}\frac{\big(b_0(y)-\fc_r\big)}{\big(b_0(y)-\fc_r\big)^2+{\fc_i}^2}+\frac{1}{\big(b_0'(2y_\fc-y)\big)^2}\frac{\big(b_0(2y_\fc-y)-\fc\big)}{\big(b_0(2y_\fc-y)-\fc_r\big)^2+{\fc_i}^2}\\
  \ge&\frac{1}{\big(b_0'(y_\fc)\big)^2}\left(\frac{\big(b_0(y)-\fc_r\big)}{\big(b_0(y)-\fc_r\big)^2+{\fc_i}^2}+\frac{\big(b_0(2y_\fc-y)-\fc\big)}{\big(b_0(2y_\fc-y)-\fc_r\big)^2+{\fc_i}^2}\right).
\end{align*}
It holds that $b_0(y)-\fc_r=b_0'(\tilde y_1)(y-y_\fc)$ and $b_0(2y_\fc-y)-\fc=b_0'(\tilde y_2)(y_\fc-y)$ for some $\tilde y_1\in[y,y_\fc]$ and $\tilde y_2\in[2y_\fc-y,y_\fc]$, we have $\tilde y_2\ge\tilde y_1$ and $1\le b_0'(\tilde y_2)\le b_0'(\tilde y_1)$. Therefore
\begin{align*}
  &\frac{\big(b_0(y)-\fc_r\big)}{\big(b_0(y)-\fc_r\big)^2+{\fc_i}^2}+\frac{\big(b_0(2y_\fc-y)-\fc_r\big)}{\big(b_0(2y_\fc-y)-\fc_r\big)^2+{\fc_i}^2}\\
  =&\frac{\big(b_0'(\tilde y_1)-b_0'(\tilde y_2)\big)(y_\fc-y)\big(b_0'(\tilde y_1)b_0'(\tilde y_2)(y_\fc-y)^2-{\fc_i}^2\big)}{\big((b_0'(\tilde y_1))^2(y-y_\fc)^2+{\fc_i}^2\big)\big((b_0'(\tilde y_1))^2(y-y_\fc)^2+{\fc_i}^2\big)}\eqdef Q.
\end{align*}
For $y_\fc-y\ge2|\fc_i|$, we have $b_0'(\tilde y_1)b_0'(\tilde y_2)(y_\fc-y)^2-{\fc_i}^2\ge0$, then $J_1\ge0$. It is easy to check that
\begin{align*}
  |Q|\le CM\gamma \frac{|y-y_\fc|}{(y-y_\fc)^2+{\fc_i}^2},
\end{align*}
and
\begin{align*}
  |J_2|\le CM\gamma\int^{|\fc_i|}_0 \frac{x}{x^2+{\fc_i}^2}dx\le CM\gamma. 
\end{align*}

For $K_1$, we have
\begin{align*}
  K_1=&\int_{\mathbb R\setminus [0,2y_\fc]} \frac{1}{\big(b_0'(y)\big)^2}\frac{\big(b_0(y)-\fc_r\big)\frac{y}{\gamma}e^{-\frac{y^2}{\gamma^2}}}{\big(b_0(y)-\fc_r\big)^2+{\fc_i}^2}dy\ge\int_{\mathbb R\setminus [-\gamma,\gamma]} \frac{1}{\big(b_0'(y)\big)^2}\frac{\big(b_0(y)-\fc_r\big)\frac{y}{\gamma}e^{-\frac{y^2}{\gamma^2}}}{\big(b_0(y)-\fc_r\big)^2+{\fc_i}^2}dy\\
  \ge&C\int_{\mathbb R\setminus [-\gamma,\gamma]} \frac{\frac{y^2}{\gamma}e^{-\frac{y^2}{\gamma^2}}}{y^2}dy=C\int_{\mathbb R\setminus [-1,1]}e^{-z^2}dz=C>0.
\end{align*}
Combing the above estimates, we have $I_r\ge CM$.

Next, we show that $|I_i|\ge CM \frac{1}{\big(\ln(\gamma^{-1})\big)^3}$ for $|\fc_r|\ge \frac{1}{4}\gamma$ and $|{\fc_i}|\le \gamma$. Recall that
\begin{align*}
  I_i=&\int^{+\infty}_{-\infty} \frac{{\fc_i}\pa_v^2\big(b_0^{-1}\big)(v)}{(v-\fc_r)^2+{\fc_i}^2}dv=-\int^{+\infty}_{-\infty} \frac{{\fc_i}\frac{b_0''(y)}{\big(b_0'(y)\big)^2}}{\big(b_0(y)-\fc_r\big)^2+{\fc_i}^2}dy\\
  =&4\sqrt\pi M \int^{+\infty}_{-\infty} \frac{1}{\big(b_0'(y)\big)^2}\frac{{\fc_i}\frac{y}{\gamma}e^{-\frac{y^2}{\gamma^2}}}{\big(b_0(y)-\fc_r\big)^2+{\fc_i}^2}dy.
\end{align*}
We have $\fc_iI_i>0$ for $\fc_r>0$ and $\fc_iI_i<0$ for $\fc_r<0$. Without loss of generally, we assume $\fc_r,\fc_i>0$. Then we write
\begin{align*}
  I_i=&4\sqrt\pi M  \int^{+\infty}_{0}{\fc_i}\frac{y}{\gamma}e^{-\frac{y^2}{\gamma^2}}\left(\frac{1}{\big(b_0'(y)\big)^2}\frac{1}{\big(b_0(y)-\fc_r\big)^2+{\fc_i}^2}-\frac{1}{\big(b_0'(-y)\big)^2}\frac{1}{\big(b_0(-y)-\fc_r\big)^2+{\fc_i}^2}\right) dy.
\end{align*}
It is clear for each $y\in(0,+\infty)$ we have
\begin{align*}
  \left(\frac{1}{\big(b_0'(y)\big)^2}\frac{1}{\big(b_0(y)-\fc_r\big)^2+{\fc_i}^2}-\frac{1}{\big(b_0'(-y)\big)^2}\frac{1}{\big(b_0(-y)-\fc_r\big)^2+{\fc_i}^2}\right)>0.
\end{align*}
Therefore, we have
\begin{align*}
  |I_i|\ge& CM \int^{y_\fc+{\fc_i}}_{y_\fc}{\fc_i}\frac{y}{\gamma}e^{-\frac{y^2}{\gamma^2}}\left(\frac{1}{\big(b_0'(y)\big)^2}\frac{1}{\big(b_0(y)-\fc_r\big)^2+{\fc_i}^2}-\frac{1}{\big(b_0'(-y)\big)^2}\frac{1}{\big(b_0(-y)-\fc_r\big)^2+{\fc_i}^2}\right) dy\\
  \ge&CM \frac{\sqrt{\ln\big(\ln(\gamma^{-1})\big)}}{\big(\ln(\gamma^{-1})\big)^4}\int^{y_\fc+{\fc_i}}_{y_\fc}\frac{{\fc_i}}{(y-y_\fc)^2+{\fc_i}^2} dy=CM \frac{\sqrt{\ln\big(\ln(\gamma^{-1})\big)}}{\big(\ln(\gamma^{-1})\big)^4}\int^{1}_{0}\frac{1}{y^2+1} dy\\
  \ge&CM \frac{\sqrt{\ln\big(\ln(\gamma^{-1})\big)}}{\big(\ln(\gamma^{-1})\big)^4}.
\end{align*}

As a conclusion, we have
\begin{align*}
  |II_r|\le C,\ |II_i|\le C\gamma\ln(\gamma^{-1}) \text{ for }\fc_r\in\mathbb R,\ |\fc_i|\le \gamma;\quad I_r\ge CM \text{ for }|\fc_r|\le \frac{1}{4}\gamma,\ |\fc_i|\le \gamma,\\
  |I_i|\ge CM \frac{\sqrt{\ln\big(\ln(\gamma^{-1})\big)}}{\big(\ln(\gamma^{-1})\big)^4}\text{ for }\frac{1}{4}\gamma<|\fc_r|< 4\gamma\sqrt{\ln\big(\ln(\gamma^{-1})\big)},\ |\fc_i|\le \gamma.
\end{align*}
Then by taking $M$ big enough, we have $\mathcal D(\fc)\neq0$ for $\fc\in E_3$. Thus we proved Lemma \ref{lem-eigen-E3}. 
\end{proof}
\begin{corol}\label{cor-finit-eigen}
  Under the assumptions on $M$ and $\gamma$ in Preposition \ref{prop-Rayleigh}, $\mathcal R$ could have at most a finite number of eigenvalues.
\end{corol}
\begin{proof}
By Remark \ref{Rmk:D(c)wronskian}, we get that $\fc\in E$ is an eigenvalue of $\mathcal R$ if and only if $\mathcal{D}(\fc)=0$. And $\mathcal{D}(\fc)$ is analytic. As $E$ is a bounded set, there could only be at most a finite number of zero points of $\mathcal{D}(\fc)$.
\end{proof}
\begin{remark}
The Cauchy's argument principle gives that the number of eigenvalues can be estimated by studying the contour integral $\f{1}{2\pi i}\oint_{\pa E}\f{\mathcal{D}'(\fc)}{\mathcal{D}(\fc)}d\fc$. 
\end{remark}

\subsection{Upper bound estimate}
Now, we prove \eqref{eq-up-bound-R} by assuming that \underline{$\mathcal R$ has no eigenvalues}. Note that $\mathcal R$ has no embedded eigenvalues. We consider the following equation
\begin{align}\label{eq: LinearEuler-Psi}
  \left\{
\begin{aligned}
&\pa_t \mathfrak w(t,y)+i\mathcal {R}\mathfrak w(t,y)=0,\\
&(\pa_y^2-1)\Psi(t,y)=\mathfrak w(t,y),\\
&\mathfrak w|_{t=0}(y)=\mathfrak w_{in}(y)=(\pa_y^2-1)\Psi_{in}(y),
\end{aligned}
\right.
\end{align}
and study the upper bound of the semi-group $e^{-it\mathcal R}$ by the following representation formula:
\begin{equation}\label{eq-rep-Psi}
  \begin{aligned}    
    \Psi(t,y)=&\lim_{\fc_i\to 0+}\lim_{T\to \infty}\frac{1}{2\pi i}\int_{-T}^T\big[e^{-i(\fc_r-i\fc_i) t}(\fc_r-i\fc_i-\mathcal{L})^{-1}\Psi_{in}\\
    &\qquad\qquad\qquad\qquad\qquad\qquad-e^{-i(\fc_r+i\fc_i) t}(\fc_r+i\fc_i-\mathcal{L})^{-1}\Psi_{in}\big]d\fc_r
  \end{aligned}
\end{equation}
where $\mathcal L=(\pa_y^2-1)^{-1}\mathcal{R}(\pa_y^2-1)$ and $\s(\mathcal L)=\mathbb{R}$. {{}The representation formula \eqref{eq-rep-Psi} can be formally regarded as the extension of the Dunford integral for unbounded spectral operators, which is however not obvious. We give the rigorous proof of the identity \eqref{eq-rep-Psi} in the proof of Proposition \ref{prop-repr} in Appendix \ref{appendix-B} by using the inverse Fourier-Laplace transform.} Let $(\fc-\mathcal{L})^{-1}\Psi_{in}=i\Phi(y,\fc)$, then $\Phi(y,\fc)$ solves the inhomogeneous Rayleigh equation:
\begin{align}\label{eq-Phi-ori}
  \pa_y^2\Phi(y,\fc)-\Phi(y,\fc)-\frac{b_0''}{b_0-\fc}\Phi(y,\fc)=i\frac{\mathfrak w_{in}(y)}{b_0(y)-\fc}. 
\end{align}

Recall that $\phi(y,\fc)$ given in Proposition \ref{prop-phi} solves the homogeneous Rayleigh equation \eqref{eq-Rayleigh}, it is easy to check that
\begin{align}\label{eq-Phi}
  \pa_y\Big(\phi^2(y,\fc)\pa_y\Big(\frac{\Phi(y,\fc)}{\phi(y,\fc)}\Big)\Big)=i\mathfrak w_{in}(y)\phi_1(y,\fc).
\end{align}
Here we briefly write $\phi_1(y,\fc)=\phi_1(y,\fc_r)\phi_2(y,\fc)$. One can see that $\phi_1(y,\fc)=\phi_1(y,\fc_r)$ if $\fc_i=0$, and $\phi_1(y,\fc)$ is well defined for $\fc\in \mathbb C$.

For $\fc_i\neq 0$, there is a unique solution $\Phi(y,\fc)$ of \eqref{eq-Phi} which decays at infinity. Then $\Phi(y,\fc)$ can be written as follows:{{}
\begin{align}\label{eq-Phi-exp}
\Phi(y,\fc)=&\Phi_{i,l}(y,\fc)-\mu(\mathfrak w_{in},\fc)\Phi_{h,l}(y,\fc)=\Phi_{i,r}(y,\fc)-\mu(\mathfrak w_{in},\fc)\Phi_{h,r}(y,\fc),
\end{align}
where 
\begin{align*}
  \Phi_{i,l}(y,\fc)=&i\phi(y,\fc)\int_{-\infty}^y\frac{\int_{y_{\fc}}^{y'}\mathfrak w_{in}(y'')\phi_1(y'',\fc)dy''}{\phi^2(y',\fc)}dy',\ \Phi_{h,l}(y,\fc)=i\phi(y,\fc)\int_{-\infty}^y\frac{1}{\phi^2(y',\fc)}dy',\\
  \Phi_{i,r}(y,\fc)=&i\phi(y,\fc)\int_{+\infty}^y\frac{\int_{y_{\fc}}^{y'}\mathfrak w_{in}(y'')\phi_1(y'',\fc)dy''}{\phi^2(y',\fc)}dy',\  \Phi_{h,r}(y,\fc)=i\phi(y,\fc)\int_{+\infty}^y\frac{1}{\phi^2(y',\fc)}dy',
\end{align*}
and
\begin{align*}
  \mu(\mathfrak w_{in},\fc)=\frac{\int_{-\infty}^{+\infty}\frac{\int_{y_{\fc}}^{y'}\mathfrak w_{in}(y'')\phi_1(y'',\fc)dy''}{\phi^2(y',\fc)}dy'}{\int_{-\infty}^{+\infty}\frac{1}{\phi^2(y',\fc)}dy'}.
\end{align*}}

Indeed, we have the following lemma. 
\begin{lemma}\label{lem:limit}
Let $\mathfrak w_{in}\in C_c^{\infty}(\mathbb{R})$, it holds for $\forall y\in\mathbb R$ that
\begin{align*}
  \lim_{\fc_i\to 0\pm}\Phi(y,\fc_r+i\fc_i)=\left\{
\begin{aligned}
&i\phi(y,\fc_r)\int_{-\infty}^y\frac{\int_{y_{\fc}}^{y'}\mathfrak w_{in}(y'')\phi_1(y'',\fc_r)dy''}{\phi^2(y',\fc_r)}dy'\\
&\quad -i\mu_{\pm}(\mathfrak w_{in},\fc_r)\phi(y,\fc_r)\int_{-\infty}^y\frac{1}{\phi^2(y',\fc_r)}dy',\quad y<y_\fc,\\
&i\phi(y,\fc_r)\int_{+\infty}^y\frac{\int_{y_{\fc}}^{y'}\mathfrak w_{in}(y'')\phi_1(y'',\fc_r)dy''}{\phi^2(y',\fc_r)}dy'\\
&\quad-i\mu_{\pm}(\mathfrak w_{in},\fc_r)\phi(y,\fc_r)\int_{+\infty}^y\frac{1}{\phi^2(y',\fc_r)}dy',\quad y>y_\fc,
\end{aligned}
\right.
\end{align*}
where 
\begin{align}
  \mu_{\pm}(\mathfrak w_{in},\fc_r)=\frac{\mathcal J_3(\mathfrak w_{in},\fc_r)\pm i\mathcal J_4(\mathfrak w_{in},\fc_r)}{\mathcal J_1(\fc_r)\mp i\mathcal J_2(\fc_r)},
\end{align}
with $\mathcal J_1(\fc_r)$, $\mathcal J_2(\fc_r)$ given in Lemma \ref{lem-lim-eigen}, and
\begin{align}\label{eq:J_3,J_4}
  \mathcal J_3(\mathfrak w_{in},\fc_r)&=\text{P.V.}\int \frac{\int_{y_{\fc}}^{y'}\mathfrak w_{in}(y'')\phi_1(y'',\fc_r)dy''}{\phi^2(y',\fc_r)}dy',\quad\mathcal J_4(\mathfrak w_{in},\fc_r)=\pi\frac{\mathfrak w_{in}(y_\fc)}{\big(b_0'(y_\fc)\big)^2}.
\end{align}
\end{lemma}
\begin{proof}
First, we show $\lim\limits_{\fc_i\to 0\pm}\mu(\mathfrak w_{in},\fc)= \mu_{\pm}(\mathfrak w_{in},\fc_r)$. 
It follows from Lemma \ref{lem-lim-eigen} and the definition of $\mathcal D(\fc)$ that
\begin{align*}
  \lim_{\fc_i\to 0\pm}\int_{-\infty}^{+\infty}\frac{1}{(b_0(y')-\fc)^2\phi_1^2(y',\fc)}dy'=\mathcal J_1(\fc_r)\mp i\mathcal J_2(\fc_r).
\end{align*}
We write the numerator of $\mu(\mathfrak w_{in},\fc)$ as
\begin{equation}\label{eq-nume-mu}
   \begin{aligned}    
&  \int_{-\infty}^{+\infty}\frac{\int_{y_{\fc}}^{y'}\mathfrak w_{in}(y'')\phi_1(y'',\fc)dy''}{(b_0(y')-\fc)^2\phi_1^2(y',\fc)}dy'\\
  =&\int_{-\infty}^{+\infty}\frac{\int_{y_{\fc}}^{y'}\mathfrak w_{in}(y'')\big(\phi_1(y'',\fc)-1\big)dy''}{(b_0(y')-\fc)^2\phi_1^2(y',\fc)}dy'
  +\int_{-\infty}^{+\infty}\frac{\int_{y_{\fc}}^{y'}\mathfrak w_{in}(y'')dy''}{(b_0(y')-\fc)^2}\big(\frac{1}{\phi_1^2(y',\fc)}-1\big)dy'\\
  &+\int_{-\infty}^{+\infty}\frac{\int_{y_{\fc}}^{y'}\mathfrak w_{in}(y'')dy''}{(b_0(y')-\fc)^2}dy'
  \eqdef \mathcal I_1+\mathcal I_2+\mathcal I_3.     
   \end{aligned}
 \end{equation} 
From \eqref{eq-phi1-est}, \eqref{eq-phi2-est} and the Lebesgue dominated convergence theorem, we have
\begin{align*}
  &\lim_{\fc_i\to0}\mathcal I_1+\mathcal I_2\\
  =&\int_{-\infty}^{+\infty}\frac{\int_{y_{\fc}}^{y'}\mathfrak w_{in}(y'')\big(\phi_1(y'',\fc_r)-1\big)dy''}{(b_0(y')-\fc_r)^2\phi_1^2(y',\fc_r)}dy'+\int_{-\infty}^{+\infty}\frac{\int_{y_{\fc}}^{y'}\mathfrak w_{in}(y'')dy''}{(b_0(y')-\fc_r)^2}\big(\frac{1}{\phi_1^2(y',\fc_r)}-1\big)dy'.
\end{align*}
Let $v=b_0(y')$, we have {{}
\begin{equation}\label{eq-I3-c}
  \begin{aligned}    
  \mathcal I_3=&\int_{-\infty}^{\infty}\frac{\int_{b_0^{-1}(\fc_r)}^{b_0^{-1}(v)}\mathfrak w_{in}(y'')dy''}{\big(v-\fc_r- i\fc_i\big)^2} \frac{1}{b'_0(b_0^{-1}(v))}d v\\
  =&\int_{-\infty}^{\infty} \frac{1}{v-\fc_r- i\fc_i} \pa_v \frac{\int_{b_0^{-1}(\fc_r)}^{b_0^{-1}(v)}\mathfrak w_{in}(y'')dy''}{b'_0(b_0^{-1}(v))} d v\\
  =&\int_{-\infty}^{\infty} \frac{v-\fc_r}{(v-\fc_r)^2+\fc_i^2} \pa_v \frac{\int_{b_0^{-1}(\fc_r)}^{b_0^{-1}(v)}\mathfrak w_{in}(y'')dy''}{b'_0(b_0^{-1}(v))} d v\\
  &+i\int_{-\infty}^{\infty} \frac{\fc_i}{(v-\fc_r)^2+\fc_i^2} \pa_v \frac{\int_{b_0^{-1}(\fc_r)}^{b_0^{-1}(v)}\mathfrak w_{in}(y'')dy''}{b'_0(b_0^{-1}(v))} d v
  \eqdef\mathcal I_{3,r}+i\mathcal I_{3,i}.  
  \end{aligned}
\end{equation}
It is clear that
\begin{align*}
  \lim_{\fc_i\to0\pm}\mathcal I_{3,r}+i\mathcal I_{3,i}=\text{P.V.}\int_{-\infty}^{\infty} \frac{1}{v-\fc_r} \pa_v \frac{\int_{b_0^{-1}(\fc_r)}^{b_0^{-1}(v)}\mathfrak w_{in}(y'')dy''}{b'_0(b_0^{-1}(v))} d v\pm \pi i\frac{\mathfrak w_{in}(y_\fc)}{\big(b_0'(y_\fc)\big)^2}.
\end{align*}

It follows that
\begin{align*}
  &\lim_{\fc_i\to0\pm}\int_{-\infty}^{+\infty}\frac{\int_{y_{\fc}}^{y'}\mathfrak w_{in}(y'')\phi_1(y'',\fc)dy''}{(b_0(y')-\fc)^2\phi_1^2(y',\fc)}dy'\\
  =&\text{P.V.}\int \frac{\int_{y_{\fc}}^{y'}\mathfrak w_{in}(y'')\phi_1(y'',\fc_r)dy''}{(b_0(y')-\fc_r)^2\phi_1^2(y',\fc_r)}dy'\pm \pi i\frac{\mathfrak w_{in}(y_\fc)}{\big(b_0'(y_\fc)\big)^2}.
\end{align*}

Next, we show the point-wise limit of $\Phi_{i,l}(y,\fc)$ and $\Phi_{h,l}(y,\fc)$ for $y<y_\fc$, as well as the limits of $\Phi_{i,r}(y,\fc)$ and $\Phi_{h,r}(y,\fc)$ for $y>y_\fc$. We only give the proof for the case $y<y_\fc$.

For $y'\le y<y_{\fc}$ and $y'\le y''\le y_{\fc}$, using Proposition \ref{prop-phi}, we have
\begin{equation}\label{eq-decay-part1}
  \begin{aligned}    
    &\left|\phi(y,\fc) \frac{\int_{y_{\fc}}^{y'}\mathfrak w_{in}(y'')\phi_1(y'',\fc)dy''}{\phi^2(y',\fc)} \right|\le C\frac{\left\|\mathfrak w_{in}\right\|_{L^\infty}\left(e^{-C_4|y'-y_{\fc}|}-1\right)e^{-C_4|y-y'|}}{|y'-y_{\fc}|},   
  \end{aligned}
\end{equation}
\begin{align}\label{eq-decay-part2}
  \left| \frac{\phi(y,\fc)}{\phi^2(y',\fc)}dy_1\right|\le C \frac{ e^{-C_4|y'-y_{\fc}|} e^{-C_4|y-y'|}}{|y'-y_{\fc}|}.
\end{align}
Then by the Lebesgue-dominated convergence theorem, for $y<y_\fc$, we get
\begin{align*}
  \lim_{\fc_i\to0\pm}\Phi_{i,l}(y,\fc)=\phi(y,\fc_r)\int_{-\infty}^y\frac{\int_{y_{\fc}}^{y'}\mathfrak w_{in}(y'')\phi_1(y'',\fc_r)dy''}{\phi^2(y',\fc_r)}dy',\\
  \lim_{\fc_i\to0\pm}\Phi_{h,l}(y,\fc)=\phi(y,\fc_r)\int_{-\infty}^y\frac{1}{\phi^2(y',\fc_r)}dy'.
\end{align*}
This finishes the proof of this lemma.
}

\end{proof}
The following estimates hold for $\mathcal J_1(\fc_r)$, $\mathcal J_2(\fc_r)$, and $\mathcal J_3(\cdot,\fc_r)$.
\begin{lemma}\label{lem-est-J1J2J3}
  Under the assumptions on $M$ and $\gamma$ in Preposition \ref{prop-Rayleigh}, it holds that
  \begin{align}
    \mathcal J_1^2(\fc_r)+\mathcal J_2^2(\fc_r)\ge C_M,\ \text{ for }\ \fc_r\in\mathbb R,\label{eq-est-J1J2}\\
    \|\mathcal J_3(f,\fc_r)\|_{L^2_{\fc_r}}\le C_M \|f\|_{L^2_y}\ \text{ for } \  f(y)\in C_c^{\infty}(\mathbb{R}),\label{eq-est-J3}
  \end{align}
  where $C_M$ is independent of $\gamma$.
\end{lemma}
\begin{proof}
Recall that 
\begin{align*}
  \mathcal J_1(\fc_r)&=\frac{1}{b_0'(y_\fc)}\Pi_1(\fc_r)+\Pi_2(\fc_r).
\end{align*}
As $M\gamma$ is small enough, we have 
\begin{align*}
   1\le b_0'(y)\le \frac{3}{2}\text{ for }\forall y\in\mathbb R.
 \end{align*} 
Then, from \eqref{eq-phi1-est}, we have that 
\begin{align*}
  -C_6\le\Pi_2(\fc_r)\le -C_7<0,
\end{align*}
where $C_6,C_7>0$ are two constants independent of $M$ and $\gamma$. 

Similar to Lemma \ref{lem-eigen-E3}, we have $\frac{1}{b_0'(y_\fc)}\Pi_1(\fc_r)\ge CM$ for $|\fc_r|\le \frac{1}{4}\gamma$. Thus by taking $M$ big enough, we have $\mathcal J_1\ge C$ for $|\fc_r|\le \frac{1}{4}\gamma$.

Next, we write
\begin{align*}
  \Pi_1(\fc_r)=& \text{P.V.}\int\frac{b_0'(y_\fc)-b_0'(y)}{(b_0(y)-\fc_r)^2} dy=\text{P.V.}\int \frac{\big(b_0'(y_\fc)-b_0'(b_0^{-1}(v))\big)\pa_vb_0^{-1}(v)}{(v-\fc_r)^2} dv\\
  =&-\text{P.V.}\int\frac{b_0'(y_\fc)-b_0'(b_0^{-1}(v))}{b_0'(b_0^{-1}(v))}\pa_v \frac{1}{v-\fc_r} dv=-b_0'(y_\fc)\text{P.V.}\int\frac{b_0''(b_0^{-1}(v))}{\big(b_0'(b_0^{-1}(v))\big)^3}\frac{1}{v-\fc_r} dv\\
  =&-b_0'(y_\fc)\text{P.V.}\int\frac{b_0''(v)}{v-\fc_r} dv-b_0'(y_\fc)\text{P.V.}\int\left(\frac{b_0''(b_0^{-1}(v))}{\big(b_0'(b_0^{-1}(v))\big)^3}-b_0''(v)\right)\frac{1}{v-\fc_r} dv\\
  \eqdef&\Pi_{1,1}(\fc_r)+\Pi_{1,2}(\fc_r).
\end{align*}
Taking $V=\frac{v}{\gamma}$, we have
\begin{align*}
  \Pi_{1,1}(\fc_r)=&-\text{P.V.}\int\frac{b_0''(v)}{v-\fc_r} dv =4\sqrt\pi M\text{P.V.}\int\frac{\frac{v}{\gamma}e^{-\frac{v^2}{\gamma^2}}}{v-\fc_r} dv\\
  =&4\sqrt\pi M\text{P.V.}\int\frac{Ve^{-V^2}}{V-\frac{\fc_r}{\gamma}} dV
  =-4\sqrt\pi M\mathcal H(Ve^{-V^2})(\frac{\fc_r}{\gamma}).
\end{align*}
It is clear that $Ve^{-V^2}\in H^1_V$, then for $\varepsilon>0$ there exists constants $C_\varepsilon$ independent of $\gamma$ that $|\Pi_{1,1}(\fc_r)|\le \varepsilon$ for $|\fc_\gamma|\ge C_\varepsilon\gamma$.

Recalling that $\|b_0(y)-y\|_{L^\infty_y}\le CM\gamma^2$, one can easily check that
\begin{align*}
  \left\|\frac{b_0''(b_0^{-1}(v))}{\big(b_0'(b_0^{-1}(v))\big)^3}-b_0''(v)\right\|_{L^2_v}\le CM^2\gamma^{\frac{3}{2}},\quad\left\|\frac{b_0''(b_0^{-1}(v))}{\big(b_0'(b_0^{-1}(v))\big)^3}-b_0''(v)\right\|_{\dot H^1_v}\le CM^2\gamma^{\frac{1}{2}},
\end{align*}
and then
\begin{align*}
  |\Pi_{1,2}(\fc_r)|\le\left\|b_0'(y_\fc)\text{P.V.}\int\left(\frac{b_0''(b_0^{-1}(v))}{\big(b_0'(b_0^{-1}(v))\big)^3}-b_0''(v)\right)\frac{1}{v-\fc_r} dv\right\|_{L^\infty_{\fc_r}}\le CM^2\gamma.
\end{align*}
Combing the estimates for $\Pi_{1,1}(\fc_r)$ and $\Pi_{1,2}(\fc_r)$, we can see for $\gamma$ small enough, there exists constant $C_8$ independent of $\gamma$ such that for $|\fc_r|\ge C_8\gamma$, $\Pi_1(\fc_r)\le \frac{C_7}{2}$, and then $|\mathcal J_1(\fc_r)|\ge\frac{C_7}{2}$. 

For $\frac{1}{4}\gamma\le|\fc_r|\le C_8\gamma$, we have $\mathcal J_2(\fc_r)\ge C_M$, which gives \eqref{eq-est-J1J2}.

Next, we turn to $\mathcal J_3$. We write
\begin{align*}
  \mathcal J_3(\fc_r)=&\text{P.V.}\int\frac{\int_{\fc_r}^{v} f\big(b_0^{-1}(v')\big)\phi_1\big(b_0^{-1}(v'),\fc_r\big)\pa_{v'}b_0^{-1}(v')dv'\pa_vb_0^{-1}(v)}{(v-\fc_r)^2\phi_1^2\big(b_0^{-1}(v),\fc_r\big)}dv\\
  =&\frac{1}{\big(b_0'(y_\fc)\big)^2}\text{P.V.}\int \frac{\int_{\fc_r}^{v} f\big(b_0^{-1}(v')\big)dv'}{(v-\fc_r)^2}\frac{1+3(v-\fc_r)^2}{\big(1+(v-\fc_r)^2\big)^2}dv\\
  &+\int_{-\infty}^{+\infty}\frac{\int_{\fc_r}^{v} f\big(b_0^{-1}(v')\big)\left(\frac{\phi_1\big(b_0^{-1}(v'),\fc_r\big)\pa_{v'}b_0^{-1}(v')\pa_vb_0^{-1}(v)}{\phi_1^2\big(b_0^{-1}(v),\fc_r\big)} -\frac{1+3(v-\fc_r)^2}{\big(b_0'(y_\fc)\big)^2\big(1+(v-\fc_r)^2\big)^2}\right)dv'}{(v-\fc_r)^2}dv\\
  \eqdef&\mathcal J_{3,1}(\fc_r)+\mathcal J_{3,2}(\fc_r).
\end{align*}
As $|v'-\fc_r|\le|v-\fc_r|$, it follows from \eqref{eq-phi1-est} that $\left|\phi_1\big(b_0^{-1}(v'),\fc_r\big)-1\right|\le C|v'-\fc_r|^2$, and $\phi_1\big(b_0^{-1}(v'),\fc_r\big)\ge Ce^{C_4|b_0^{-1}(v')-y_\fc|}$. Then one can easily check that
\begin{align*}
  \left|\frac{\phi_1\big(b_0^{-1}(v'),\fc_r\big)\pa_{v'}b_0^{-1}(v')\pa_vb_0^{-1}(v)}{\phi_1^2\big(b_0^{-1}(v),\fc_r\big)} -\frac{1+3(v-\fc_r)^2}{\big(b_0'(y_\fc)\big)^2\big(1+(v-\fc_r)^2\big)^2}\right|\le C\min(|v-\fc_r|,\frac{1}{|v-\fc_r|^2}).
\end{align*}
Therefore, we have
\begin{align*}
  \|\mathcal J_{3,2}\|_{L^2_{\fc_r}}\le& C\left\|\int_{-\infty}^{+\infty}\frac{\min(|v-\fc_r|,\frac{1}{|v-\fc_r|^2})}{|v-\fc_r|}\frac{1}{|v-\fc_r|}\int_{\fc_r}^{v}| f\big(b_0^{-1}(v')\big)|dv'dv\right\|_{L^2_{\fc_r}}\\
  \le&C \| f\big(b_0^{-1}(\cdot)\big)\|_{L^2_{v}}\le C \| f\|_{L^2_y}.
\end{align*}
For $\mathcal J_{3,1}(\fc_r)$, we have
\begin{align*}
  \mathcal J_{3,1}(\fc_r)=&-\frac{1}{\big(b_0'(y_\fc)\big)^2}\text{P.V.}\int \int_{\fc_r}^{v} f\big(b_0^{-1}(v')\big)dv'\pa_v\left(\frac{1}{(v-\fc_r)\big(1+(v-\fc_r)^2\big)}\right)dv\\
  =&\frac{1}{\big(b_0'(y_\fc)\big)^2}\text{P.V.}\int \frac{ f\big(b_0^{-1}(v)\big)}{(v-\fc_r)\big(1+(v-\fc_r)^2\big)}dv\\
  =&\frac{1}{\big(b_0'(y_\fc)\big)^2}\text{P.V.}\int_{\fc_r-1}^{\fc_r+1}\frac{ f\big(b_0^{-1}(v)\big)}{(v-\fc_r)}dv-\frac{1}{\big(b_0'(y_\fc)\big)^2}\int_{\fc_r-1}^{\fc_r+1}\frac{ f\big(b_0^{-1}(v)\big)(v-\fc_r)}{1+(v-\fc_r)^2}dv\\
  &+\frac{1}{\big(b_0'(y_\fc)\big)^2}\int_{\mathbb R\setminus[\fc_r-1,\fc_r+1]}\frac{ f\big(b_0^{-1}(v)\big)}{(v-\fc_r)\big(1+(v-\fc_r)^2\big)}dv\\
  =&\frac{\mathcal H_1\big( f(b_0^{-1}(\cdot))\big)(\fc_r)-\mathcal H\big( f(b_0^{-1}(\cdot))\big)(\fc_r)}{\big(b_0'(y_\fc)\big)^2}\\
  &-\frac{\int_{\mathbb R} f\big(b_0^{-1}(v+\fc_r)\big)\frac{\chi_{|v|\le1}v}{v\big(1+v^2\big)}dv-\int_{\mathbb R} f\big(b_0^{-1}(v+\fc_r)\big)\frac{\chi_{|v|\ge1}}{v\big(1+v^2\big)}dv}{\big(b_0'(y_\fc)\big)^2},
\end{align*}
where $\mathcal H_1\big( f(b_0^{-1}(\cdot))\big)(\fc_r)=\int_{\mathbb R\setminus[-1,1]}\frac{ f(b_0^{-1}(\fc_r-v))}{v}dv$. By Young's convolution inequality and the properties of the (maximal) Hilbert operator, we have
\begin{align*}
  \|\mathcal J_{3,1}\|_{L^2_{\fc_r}}\le C \| f\big(b_0^{-1}(\cdot)\big)\|_{L^2_{v}}\le C \| f\|_{L^2_y}.
\end{align*}
The estimate \eqref{eq-est-J3} follows immediately.
\end{proof}{{}
\begin{proposition}[Representation formula]\label{prop-repr}
  Let $\mathfrak w_{in}\in C_c^{\infty}(\mathbb{R})$, it holds that 
  \begin{equation}\label{eq-rep-Psi-lim}
    \begin{aligned}    
      \Psi(t,y)=&-\frac{1}{\pi}\int_{-\infty}^{b_0(y)} e^{-i\fc_r t} \phi(y,\fc_r)\frac{\mathcal J_1\mathcal J_4+\mathcal J_2\mathcal J_3}{\mathcal J_1^2+\mathcal J_2^2}\int_{+\infty}^y\frac{1}{\phi^2(y',\fc_r)}dy'd\fc_r\\
      &\qquad\qquad-\frac{1}{\pi}\int_{b_0(y)}^{+\infty} e^{-i\fc_r t}\phi(y,\fc_r) \frac{\mathcal J_1\mathcal J_4+\mathcal J_2\mathcal J_3}{\mathcal J_1^2+\mathcal J_2^2}\int_{-\infty}^y\frac{1}{\phi^2(y',\fc_r)}dy'd\fc_r.    
    \end{aligned}
  \end{equation}
\end{proposition}
By Lemma \ref{lem:limit}, exchanging the limit and integral in \eqref{eq-rep-Psi} yields \eqref{eq-rep-Psi-lim}, for a detailed proof, we refer the readers to Appendix \ref{appendix-B}.}

Now we give the proof of the upper bound estimate \eqref{eq-up-bound-R}.
\begin{lemma}\label{lem-up-bound-noeigen}
  Under the assumptions on $M$ and $\gamma$ in Preposition \ref{prop-Rayleigh}, and the assumption that $\mathcal R$ has no eigenvalues, it holds that
  \begin{align}\label{eq:semigroup-up-bound-e-itR}
  \|e^{-it\mathcal R}\mathfrak w_{in}\|_{L^2_{y}}\le C\|\mathfrak w_{in}\|_{L^2_{y}},
\end{align}
where $C$ is a constant independent of $\gamma$.
\end{lemma}
\begin{proof}
Let $\Psi(t,y)$ be the solution of \eqref{eq: LinearEuler-Psi} { with $\om_{in}\in C_c^{\infty}$}. { It has the representation formula \eqref{eq-rep-Psi-lim}.} Instead of acting $(\pa_{yy}-1)$ on the stream function $\Psi(t,y)$, to estimate $\mathfrak w(t,y)$,  we use the duality argument to give the upper bound of $\|\mathfrak w(t,y)\|_{L^2_y}$. We make inner product of $\mathfrak w(t,y)$ with a test function $g(y)\in C_c^\infty(\mathbb R)$ such that $\|g\|_{L^2_y}=1$, and get
\begin{align*}
  \int_{\mathbb{R}}\mathfrak w(t,y)g(y)dy=\int_{\mathbb{R}}\Psi(t,y)(\pa_y^2-1)g(y)dy=\int_{\mathbb{R}}e^{-i\fc_r t}K(\fc_r)d\fc_r,
\end{align*}
where 
\begin{align*}
  K(\fc_r)=&-\frac{1}{\pi}\frac{\mathcal J_1(\fc_r)\mathcal J_4(\mathfrak w_{in},\fc_r)+\mathcal J_2(\fc_r)\mathcal J_3(\mathfrak w_{in},\fc_r)}{\mathcal J_1^2(\fc_r)+\mathcal J_2^2(\fc_r)}\int^{+\infty}_{-\infty}\phi(y,\fc_r)(\pa_y^2-1)g(y)\\
  &\qquad\quad\cdot\Big(\chi_{y>y_\fc}\int_{+\infty}^y\frac{1}{\phi^2(y',\fc_r)}dy'+\chi_{y<y_\fc}\int_{-\infty}^y\frac{1}{\phi^2(y',\fc_r)}dy'\Big)dy\\
  =&\frac{1}{\pi}\frac{\mathcal J_1(\fc_r)\mathcal J_4(\mathfrak w_{in},\fc_r)+\mathcal J_2(\fc_r)\mathcal J_3(\mathfrak w_{in},\fc_r)}{\mathcal J_1^2(\fc_r)+\mathcal J_2^2(\fc_r)}\int^{+\infty}_{-\infty}\frac{\int_{y_\fc}^{y'}\phi(y,\fc_r)(\pa_y^2-1)g(y)dy}{\phi^2(y',\fc_r)}dy'\\
  =&\frac{1}{\pi}\frac{\mathcal J_1(\fc_r)\mathcal J_4(\mathfrak w_{in},\fc_r)+\mathcal J_2(\fc_r)\mathcal J_3(\mathfrak w_{in},\fc_r)}{\mathcal J_1^2(\fc_r)+\mathcal J_2^2(\fc_r)}\\
  &\quad\cdot\text{P.V.}\int \frac{\int_{y_\fc}^{y'}b_0''(y)\phi_1(y,\fc_r)  g(y)dy+\phi(y',\fc_r) g'(y')-\phi'(y',\fc_r) g(y')+b_0'(y_\fc) g(y_\fc)}{\phi^2(y',\fc_r)}dy',
\end{align*}
and $\mathfrak c_r=b(y_{\fc})$. Here we have used the Fubini theorem in the second identity and integration by part twice in the third identity. 

Hence by the definition of $\mathcal J_3$ in \eqref{eq:J_3,J_4}, we have 
\begin{align*}
  \text{P.V.}\int \frac{\int_{y_\fc}^{y'}b_0''(y)\phi_1(y,\fc_r)  g(y)dy}{\phi^2(y',\fc_r)}dy'=\mathcal J_3(b_0''g,\fc_r).
\end{align*}
We rewrite
\begin{align*}
  &\text{P.V.}\int \frac{\phi(y',\fc_r) g'(y')-\phi'(y',\fc_r) g(y')+b_0'(y_\fc) g(y_\fc)}{\phi^2(y',\fc_r)}dy'\\
  =&\text{P.V.}\int \frac{ g'(y')}{\phi(y',\fc_r)}dy'-\text{P.V.}\int \frac{ g'(y')b_0'(y')\phi_1(y',\fc_r)-b_0'(y_\fc) g(y_\fc)}{\phi^2(y',\fc_r)}dy'\\
  &-\text{P.V.}\int \frac{ g(y')\big(b_0(y')-\fc_r\big)\phi_1'(y',\fc_r)}{\phi^2(y',\fc_r)}dy'
  \eqdef\mathcal K_1+\mathcal K_2+\mathcal K_3.
\end{align*}
For each term, we deduce that
\begin{align*}
  \mathcal K_1=&\text{P.V.}\int \frac{\pa_{y'}\big( g(y')- g(y_\fc)\big)}{\phi(y',\fc_r)}dy'=\text{P.V.}\int \frac{\big( g(y')- g(y_\fc)\big)\phi'(y',\fc_r)}{(b_0(y')-\fc_r)^2\phi_1^2(y',\fc_r)}dy'\\
  =&\text{P.V.}\int \frac{\big( g(y')- g(y_\fc)\big)b_0'(y')\phi_1(y',\fc_r)}{(b_0(y')-\fc_r)^2\phi_1^2(y',\fc_r)}dy'\\
  &+\text{P.V.}\int \frac{\big( g(y')- g(y_\fc)\big)\big(b_0(y')-\fc_r\big)\phi_1'(y',\fc_r)}{(b_0(y')-\fc_r)^2\phi_1^2(y',\fc_r)}dy'
  \eqdef \mathcal K_{1,1}+\mathcal K_{1,2},
\end{align*}
and
\begin{align*}
  &\mathcal K_{1,1}+\mathcal K_2=- g(y_\fc)\text{P.V.}\int \frac{b_0'(y')\phi_1(y',\fc_r)-b_0'(y_\fc)}{(b_0(y')-\fc_r)^2\phi_1^2(y',\fc_r)}dy'\\
  =&- g(y_\fc)\text{P.V.}\int \frac{b_0'(y')-b_0'(y_\fc)}{\big(b_0(y')-\fc_r\big)^2}dy'- g(y_\fc)\int^{+\infty}_{-\infty}\frac{b_0'(y')-b_0'(y_\fc)}{\big(b_0(y')-\fc_r\big)^2}\left(\frac{1}{\phi_1(y',\fc_r)}-1\right)dy'\\
  &- b_0'(y_\fc)g(y_\fc)\int^{+\infty}_{-\infty}\frac{1}{\big(b_0(y')-\fc_r\big)^2}\frac{\phi_1(y',\fc_r)-1}{\phi_1^2(y',\fc_r)}dy',
\end{align*}
and
\begin{align*}
  \mathcal K_{1,2}+\mathcal K_3=&- g(y_\fc)\int^{+\infty}_{-\infty}\frac{\big(b_0(y')-\fc_r\big)\phi_1'(y',\fc_r)}{(b_0(y')-\fc_r)^2\phi_1^2(y',\fc_r)}dy'\\
  =& g(y_\fc)\int^{+\infty}_{-\infty}\frac{1}{b_0(y')-\fc_r}\pa_y\left(\frac{1}{\phi_1(y',\fc_r)}-1\right)dy'\\
  =& g(y_\fc)\int^{+\infty}_{-\infty}\frac{b_0'(y')}{\big(b_0(y')-\fc_r\big)^2}\left(\frac{1}{\phi_1(y',\fc_r)}-1\right)dy'.
\end{align*}
Thus, we conclude that
\begin{align*}
  %&\text{P.V.}\int \frac{\phi(y',\fc_r) g'(y')-\phi'(y',\fc_r) g(y')+b_0'(y_\fc) g(y_\fc)}{(b_0(y')-\fc_r)^2\phi_1^2(y',\fc_r)}dy'\\
  &\mathcal K_1+\mathcal K_2+\mathcal K_3
  =\mathcal K_{1,1}+\mathcal K_2+\mathcal K_{1,2}+\mathcal K_3\\
  =& g(y_\fc)\text{P.V.}\int \frac{b_0'(y_\fc)-b_0'(y')}{\big(b_0(y')-\fc_r\big)^2}dy'\\
  &+ b_0'(y_\fc)g(y_\fc)\int^{+\infty}_{-\infty}\frac{1}{\big(b_0(y')-\fc_r\big)^2}\left(\frac{1}{\phi_1^2(y',\fc_r)}-1\right)dy'
  = b_0'(y_\fc)g(y_\fc)\mathcal J_1(\fc_r).
\end{align*}
It follows from the above results that
\begin{align*}
  K(\fc_r)=&\frac{1}{\pi}\frac{\mathcal J_1(\fc_r)\mathcal J_4(\mathfrak w_{in},\fc_r)+\mathcal J_2(\fc_r)\mathcal J_3(\mathfrak w_{in},\fc_r)}{\mathcal J_1^2(\fc_r)+\mathcal J_2^2(\fc_r)}\big(\mathcal J_3(b_0''g,\fc_r)+b_0'(y_\fc)g(y_\fc)\mathcal J_1(\fc_r)\big).
\end{align*}
Then by Lemma \ref{lem-est-J1J2J3} and the fact that $\|g\|_{L^2}=1$, we have
\begin{align*}
\|\mathfrak w(t)\|_{L^2_y}&\leq  \|K(\fc_r)\|_{L^1_{c_r}}\\
&\le C\left(\|\mathfrak w_{in}\|_{L^2_y}+\left\|\mathcal J_3(\mathfrak w_{in},\fc_r)\right\|_{L^2_{\fc_r}}\right)\left(\|b_0'g\|_{L^2_y}+\left\|\mathcal J_3(b_0''g,\fc_r)\right\|_{L^2_{\fc_r}}\right)\le C \|\mathfrak w_{in}\|_{L^2_y},
\end{align*}
where $C>0$ is independent of $\gamma$.
{{}At last, by using the fact that $C_{c}^{\infty}$ is dense in $L^2$, and using the $L^2$ boundedness of  operators $J_3(\mathfrak w_{in},\fc_r)$ and $J_4(\mathfrak w_{in},\fc_r)$, the above estimates holds for general $\om_{in}\in L^2$. }
\end{proof}
\begin{remark}
The estimate \eqref{eq:semigroup-up-bound-e-itR} holds if $\mathcal{P}^d_{\mathcal{R}}\mathfrak w_{in}=0$ where $\mathcal{P}^d_{\mathcal{R}}$ is the spectral projection to the eigenspaces which correspond to the discontinuous spectrum $\s_d(\mathcal{R})$. 
\end{remark}
\begin{appendix}
  \section{Key points for the proof of Corollary \ref{col:easy-corol}}
Consistent with \cite{MasmoudiZhao2020cpde}, in the appendix we use $S_\nu(t,s)f$ to denote the solution of 
\begin{equation}
  \left\{
    \begin{array}{l}
      \pa_t\omega+y\pa_x\omega-\nu\Delta\omega=0,\\
      \omega|_{t=s}=f(x,y),
    \end{array}
  \right.
\end{equation}
with $\int_{\mathbb T} f(x,y) d x=0$.

We now consider the nonlinear equation,
\begin{equation}\label{eq-om-nonl-app}
  \left\{
    \begin{array}{l} 
      \pa_t\omega_{\neq}+y\pa_x\omega_{\neq}-\nu\Delta\omega_{\neq}=-\mathcal L-\mathcal N^{(1)}-\mathcal N^{(2)}-\mathcal N^{(3)},\\
      (u^{(1)}_{\neq},u^{(2)}_{\neq})=(\pa_y(-\Delta)^{-1}\omega_{\neq},-\pa_x(-\Delta)^{-1}\omega_{\neq}),\\
      \omega_{\neq}|_{t=0}=P_{\neq}\omega_{in},
    \end{array}
  \right.
\end{equation}
where $b_\nu(t,y)$ is the solution of \eqref{eq:heat},
\begin{align*}
  \mathcal L=(b_\nu-y)\pa_x\omega_{\neq}+\pa_y^2b_\nu\pa_x(\Delta)^{-1}\omega_{\neq},&\quad  \mathcal N^{(1)}=\big(u^{(1)}_{\neq}\pa_x\omega_{\neq}\big)_{\neq}+\big(u^{(2)}_{\neq}\pa_y\omega_{\neq}\big)_{\neq},\\
  \mathcal N^{(2)}=u^{(1)}_{0}\pa_x\omega_{\neq},&\qquad \mathcal N^{(3)}=u^{(2)}_{\neq}\pa_y\omega_0,
\end{align*}
$\omega_0(t,y)$ is the zero mode of vorticity which satisfies
\begin{equation}
  \left\{
    \begin{array}{l}
      \pa_t\omega_0-\nu\pa_y^2\omega_0=-\big(u^{(1)}_{\neq}\pa_x\omega_{\neq}\big)_0-\big(u^{(2)}_{\neq}\pa_y\omega_{\neq}\big)_0,\\
      \omega_0|_{t=0}=P_0\omega_{in},
    \end{array}
  \right.
\end{equation}
and $u^{(1)}_0$ is the zero mode of horizontal velocity which satisfies
\begin{equation}
  \left\{
    \begin{array}{l}
      \pa_t u^{(1)}_0-\nu\pa_y^2u^{(1)}_0=-\big(u^{(1)}_{\neq}\pa_xu^{(1)}_{\neq}\big)_0-\big(u^{(2)}_{\neq}\pa_yu^{(1)}_{\neq}\big)_0,\\
      u^{(1)}_0|_{t=0}=P_0u^{1}_{in}.
    \end{array}
  \right.
\end{equation}
Here
\begin{align*}
  \|b_{in}(y)-y\|_{L^\infty_y\cap \dot{H}^1_y}\leq \varepsilon_0\nu^{\frac{1}{2}},\quad \|\omega_{in}\|_{H^1_xL_y^2}\le\varepsilon_0\nu^{\frac{1}{2}}.
\end{align*}
In \cite{MasmoudiZhao2020cpde}, the authors take $b_\nu(t,y)\equiv y$, and then $\mathcal L\equiv0$. Therefore, to prove Corollary \ref{col:easy-corol}, it suffices to give an estimate for $\mathcal L$ similar to the estimates in Lemma 3.3 of \cite{MasmoudiZhao2020cpde}:
\begin{align*}
  \|\ln(e+|D_x|)\mathcal L(t)\|_{L^1_t([0,T];L^2_{x,y})}\le C \varepsilon_0\|\ln(e+|D_x|)\omega_{in}\|_{L^2_{x,y}}.
\end{align*}

It follows from the properties of the heat kernel that
\begin{align*}
  \|b_\nu(t,y)-y\|_{L^\infty_t([0,T];L^\infty_y\cap \dot{H}^1_y)}+\nu^{\frac{1}{2}}\|\pa_y^2b_\nu\|_{L^2_t([0,T];L^2_y)}\leq C\varepsilon_0\nu^{\frac{1}{2}}.
\end{align*}
Then we have
\begin{align*}
  &\|\ln(e+|D_x|)\mathcal L(t)\|_{L^1_t([0,T];L^2_{x,y})}\\
  \le&\|b_\nu(t,y)-y\|_{L^\infty_t([0,T];L^\infty_y)}\|\ln(e+|D_x|)\pa_x\omega_{\neq}\|_{L^1_t([0,T];L^2_{x,y})}\\
  &+\|\pa_y^2\omega_{\neq}\|_{L^2_t([0,T];L^2_{y})}\|\ln(e+|D_x|)\pa_x(\Delta)^{-1}\omega_{\neq}\|_{L^2_t([0,T];L^2_{x}L^\infty_y)}\\
  \le&C(C_3+C_5)\varepsilon_0\|\ln(e+|D_x|)\omega_{in}\|_{L^2_{x,y}},
\end{align*}
where $C_3$ and $C_5$ are given in Section 3 of \cite{MasmoudiZhao2020cpde}. 

With this estimate, one can easily deduce the result of Corollary \ref{col:easy-corol} by following the proof in \cite{MasmoudiZhao2020cpde}.

\section{Rayleigh equation}\label{appendix-B}
Here we give the proofs of the proposition and lemmas given in Section \ref{sec-Ray}. We modify the argument in \cite{WeiZhangZhao2018} to adapt it to the unbounded domain case. We remark that the estimates in this section hold for general $b_0(y)$ which satisfies:
\begin{align}\label{eq-b-property}
  0<C^{-1}\le|b_0'(y)|\le C \  \text{ for } \ y\in\mathbb R,\quad \|b_0''\|_{L^\infty_y\cap L^1_y}\le C.
\end{align}

To prove Proposition \ref{prop-phi}, we introduce function spaces $X$ and $Y$. For a function $g(y,\fc_r)$ defined on $\mathbb R\times \mathbb R$, we define
\begin{align*}
  \|g\|_{X}\eqdef \sup_{(y,\fc)\in\mathbb R\times \mathbb R}\left|\frac{g(y,\fc_r)}{\cosh(\tilde C(y-y_\fc))}\right|,\quad\|g\|_{Y}\eqdef \sum_{k=0}^2\tilde C^{-k} \|\pa_y^kg\|_{X}.
\end{align*}
Here $\tilde C>0$ is a constant which will be determined later.
\begin{proof}[Proof of Proposition \ref{prop-phi}]
  From \eqref{eq-phi1}, we write
  \begin{align}\label{eq-phi1-exp}
  \phi_1(y,\fc_r)=1+\int^y_{y_\fc} \frac{1}{\big(b_0(y')-b_0(y_\fc)\big)^2}\int^{y'}_{y_\fc}\phi_1(z,\fc_r)\big(b_0(z)-b_0(y_\fc)\big)^2dzdy'.
\end{align}
Submitting $\phi(y,\fc)=\big(b_0(y)-\fc\big)\phi_1(y,\fc_r)\phi_2(y,\fc)$ into \eqref{eq-Rayleigh}, we have that $\phi_2(y,\fc)$ satisfies
\begin{align*}
  (b_0-\fc)\phi_1\phi_2''+2b_0'\phi_1\phi_2'+2(b_0-\fc)\phi_1'\phi_2'+\frac{2i{\fc_i} b_0'}{b_0-\fc_r}\phi_1'\phi_2=0,
\end{align*}
%Multiplying the above equation by $(b-\fc)\phi_1$, we have
%\begin{align*}
%  (b-\fc)^2\phi_1^2\phi_2''+2b_0'(b-\fc)\phi_1^2\phi_2'+2(b-\fc)^2\phi_1\phi_1'\phi_2'+\frac{2i{\fc_i} b_0'(b-\fc)}{b-\fc}\phi_1\phi_1'\phi_2=0.
%\end{align*}
which together with the boundary conditions $\phi_2(y_\fc,\fc)=1$ and $\phi_2'(y_\fc,\fc)=0$ gives \eqref{eq-phi2}. Then we write
\begin{align}\label{eq-phi2-exp}
  \phi_2(y,\fc)=1-2i{\fc_i}\int^y_{y_\fc} \frac{1}{\big(b_0(y')-\fc\big)^2\phi_1^2(y')}\int^{y'}_{y_\fc}\frac{b_0'(z)(b_0(z)-\fc)}{b_0(z)-\fc_r}\phi_1(z)\phi_1'(z)\phi_2(z)dzdy'.
\end{align}

Let $T_1$ be the integral operator that
\begin{align*}
  T_1g(y,\fc_r)=T_0\circ T_{2,2}g(y,\fc_r)=\int^y_{y_\fc} \frac{1}{\big(b_0(y')-b_0(y_\fc)\big)^2}\int^{y'}_{y_\fc}g(z,\fc_r)\big(b_0(z)-b_0(y_\fc)\big)^2dzdy',
\end{align*}
where
\begin{align*}
  T_0 g(y,\fc_r)=&\int^y_{y_\fc} g(y',\fc_r) dy'\\
  T_{k,j}g(y,\fc_r)=&\frac{1}{\big(b_0(y)-b_0(y_\fc)\big)^j}\int^{y}_{y_\fc}g(z,\fc_r)\big(b_0(z)-b_0(y_\fc)\big)^kdz.
\end{align*}
By the definition of the operator $T$, we have $\phi_1=1+T\phi_1$.

It is clear that
\begin{align*}
  \|T_0 g\|_{X}=&\sup_{(y,\fc_r)\in\mathbb R\times \mathbb R}\left|\frac{1}{\cosh(\tilde C(y-y_\fc))}\int^y_{y_\fc}\frac{g(y',\fc_r)}{\cosh(\tilde C(y'-y_\fc))}\cosh(\tilde C(y'-y_\fc))  dy'\right|\\
  \le&\sup_{(y,\fc_r)\in\mathbb R\times \mathbb R}\left|\frac{1}{\cosh(\tilde C(y-y_\fc))}\int^y_{y_\fc}\cosh(\tilde C(y'-y_\fc))  dy'\right|\|g\|_{X}\le\frac{1}{\tilde C}\|g\|_{X},
\end{align*}
and
\begin{align*}
  \|T_{2,2} g\|_{X}=&\sup_{(y,\fc_r)\in\mathbb R\times \mathbb R}\left|\frac{\int^y_{y_\fc}\frac{g(y',\fc_r)\big(b_0(y')-b_0(y_\fc)\big)^2}{\cosh(\tilde C(y'-y_\fc))}\cosh(\tilde C(y'-y_\fc))  dy'}{\cosh(\tilde C(y-y_\fc))\big(b_0(y)-b_0(y_\fc)\big)^2}\right|\\
  \le&\sup_{(y,\fc_r)\in\mathbb R\times \mathbb R}\left|\frac{1}{\cosh(\tilde C(y-y_\fc))}\int^y_{y_\fc}\cosh(\tilde C(y'-y_\fc))  dy'\right|\|g\|_{X}\le\frac{1}{\tilde C}\|g\|_{X}.
\end{align*}
Here we use the fact that $|b_0(y')-b_0(y_\fc)|\le|b_0(y)-b_0(y_\fc)|$ for $|y'-y_\fc|\le |y-y_\fc|$. 

It follows directly that
\begin{align*}
  \|T_1 g\|_{X}\le \frac{1}{\tilde C^2}\|g\|_{X}.
\end{align*}

Direct calculations show that
\begin{align*}
  \pa_yT_1 g(y,\fc_r)=T_{2,2} g(y,\fc_r),\quad \pa_y^2T_1 g(y,\fc_r)=-2b_0'(y)T_{2,3} g(y,\fc_r)+g(y,\fc_r).
\end{align*}

We write
\begin{align*}
  T_{2,3} g(y,\fc_r)=&\frac{1}{\big(b_0(y)-b_0(y_\fc)\big)^3}\int^{y}_{y_\fc}g(z,\fc_r)\big(b_0(z)-b_0(y_\fc)\big)^2dz\\
  =&\int^y_{y_\fc} \frac{(z-y_\fc)^2(\int^1_0b_0'(y_\fc+s(z-y_\fc))ds)^2}{(y-y_\fc)^3(\int^1_0b_0'(y_\fc+s(y-y_\fc))ds)^3}g(z,\fc_r)dz\\
  =&\int^1_{0} \frac{(\int^1_0b_0'(y_\fc+st(y-y_\fc))ds)^2}{(\int^1_0b_0'(y_\fc+s(y-y_\fc))ds)^3}g(y_\fc+t(y-y_\fc),\fc)t^2dt,
\end{align*}
then we have
\begin{align*}
  \|T_{2,3} g\|_{X}\le&C\sup_{(y,\fc)\in\mathbb R\times \mathbb R}\left|\frac{1}{\cosh(\tilde C(y-y_\fc))}\int^1_{0}\frac{g(y_\fc+t(y-y_\fc),\fc)t^2}{\cosh(\tilde Ct(y-y_\fc))}\cosh(\tilde Ct(y-y_\fc))  dy'\right|\\
  \le&C\sup_{(y,\fc)\in\mathbb R\times \mathbb R}\left|\frac{1}{\cosh(\tilde C(y-y_\fc))}\int^1_{0}\cosh(\tilde Ct(y-y_\fc))  dy'\right|\|g\|_{X}\le C\|g\|_{X},
\end{align*}
which gives that
\begin{align*}
  \frac{1}{\tilde C}\|\pa_yT_1 g\|_{X}\le \frac{1}{\tilde C^2}\|g\|_{X},\quad \frac{1}{\tilde C^2}\|\pa_y^2T_1 g\|_{X}\le \frac{C}{\tilde C^2}\|g\|_{X}.
\end{align*}
It follows that
\begin{align*}
  \|T_1 g\|_{Y}\le \frac{C}{\tilde C^2}\|g\|_{Y}.
\end{align*}
By taking $\tilde C$ big enough, we have that $I-T_1$ is invertible in the space $Y$. Thus,
\begin{align*}
  \phi_1(y,\fc_r)=(I-T_1)^{-1}1,
\end{align*}
with the bound $\| \phi_1\|_{Y}\le C$. Then, from the expression \eqref{eq-phi1-exp}, one can easily verify that $\phi_1(y,\fc_r)\ge1$.

Let $F(y,\fc_r)=\frac{\phi_1'(y,\fc_r)}{\phi_1(y,\fc_r)}$. It is easy to check that
\begin{align*}
  F'(y,\fc_r)+F^2(y,\fc_r)+\frac{2b_0'(y)F(y,\fc_r)}{b_0(y)-\fc_r}-1=0.
\end{align*}
It follows from $\phi_1(y_\fc,\fc_r)=1$ and $\phi_1'(y_\fc,\fc_r)=0$ that $F(y_\fc,\fc_r)=0$. Then we can see that
\begin{align*}
  \lim_{y\to y_\fc}F'(y,\fc_r)=&1-\lim_{y\to y_\fc}F^2(y,\fc_r)-\lim_{y\to y_\fc}\frac{2b_0'(y)F(y,\fc_r)}{b_0(y)-\fc_r}\\
%  =&1-\lim_{y\to y_\fc}2b_0'(y)\frac{y-y_\fc}{b_0(y)-\fc_r}\frac{F(y,\fc_r)-F(y_\fc,\fc_r)}{y-y_\fc}\\
  =&1-2\lim_{y\to y_\fc}F'(y,\fc_r).
\end{align*}
Therefore, we have $F'(y_\fc,\fc_r)=\frac{1}{3}>0$. From \eqref{eq-phi1-exp} we have
\begin{align*}
  \phi_1'(y,\fc_r)=\frac{1}{\big(b_0(y)-b_0(y_\fc)\big)^2}\int^{y}_{y_\fc}\phi_1(z,\fc_r)\big(b_0(z)-b_0(y_\fc)\big)^2dz,
\end{align*}
and $\phi_1'(y,\fc_r)>0$ for $y>y_\fc$ and $\phi_1'(y,\fc_r)<0$ for $y<y_\fc$. Therefore $\phi_1(y,\fc_r)\ge\phi_1(z,\fc_r)$ for $z\in[y,y_\fc]$ or $z\in[y_\fc,y]$. Then we have $\frac{F(y,\fc_r)}{b_0(y)-\fc_r}\ge0$ for $\forall y\in\mathbb R$, and
\begin{align*}
  |F(y,\fc_r)|=\left|\frac{\phi_1'(y,\fc_r)}{\phi_1(y,\fc_r)}\right|=\left|\frac{1}{\big(b_0(y)-b_0(y_\fc)\big)^2}\int^{y}_{y_\fc}\frac{\phi_1(z,\fc_r)}{\phi_1(y,\fc_r)}\big(b_0(z)-b_0(y_\fc)\big)^2dz\right|\le |y-y_\fc|.
\end{align*}

Next, we show that $|F(y,\fc_r)|\le1$. If $F$ attains its maximum (minimum) at $y_0$, we have $F'(y_0,\fc_r)=0$ and 
\begin{align*}
  F^2(y_0,\fc_r)=1-\frac{2b_0'(y_0)F(y_0,\fc_r)}{b_0(y_0)-\fc_r}\le 1.
\end{align*}
Therefore $|F(y,\fc_r)|\le1$. If $F(y,\fc_r)>1$ at $y_1$, we know that $y_1>y_\fc$ as $F(y_1,\fc_r)>0$. Then we have $F'(y_1,\fc_r)< 0$. Recall that $F'(y_\fc,\fc_r)=\frac{1}{3}>0$, so there exits $y_2\in(y_\fc,y_1)$ such that $F'(y_2,\fc_r)=0$ and $F(y_2,\fc_r)\ge F(y_1,\fc_r)>1$, which is impossible. If $F(y,\fc_r)<-1$ at $y_3$, we know that $y_3<y_\fc$ and $F'(y_3,\fc_r)<0$, then $F(y,\fc_r)$ decreases strictly on $[y_3,y_\fc]$, which contradicts $F(y_\fc,\fc_r)=0$. As a conclusion, we have $|F(y,\fc_r)|\le \min(1,|y-y_\fc|)$. It follows that
\begin{align*}
  e^{-|y-y'|}\le \frac{\phi_1(y',\fc_r)}{\phi_1(y,\fc_r)}\le e^{|y-y'|}.
\end{align*}

Next, we show that for $y$ such that $|y-y_\fc|\ge 1$, $|F(y,\fc_r)|\ge C_4>0$. Without loss of generality, we assume $y\ge y_\fc+1$, and have
\begin{align*}
  F(y,\fc_r)=&\frac{1}{\big(b_0(y)-b_0(y_\fc)\big)^2}\int^{y}_{y_\fc}\frac{\phi_1(z,\fc_r)}{\phi_1(y,\fc_r)}\big(b_0(z)-b_0(y_\fc)\big)^2dz\\
  \ge&\frac{1}{\big(b_0(y)-b_0(y_\fc)\big)^2}\int^{y}_{y_\fc}e^{-|y-z|}\big(b_0(z)-b_0(y_\fc)\big)^2dz\\
  \ge&\frac{\big(b_0(y-\frac{1}{2})-b_0(y_\fc)\big)^2}{\big(b_0(y)-b_0(y_\fc)\big)^2}\int^{y}_{y-\frac{1}{2}}e^{-|y-z|}dz
  \ge C.
\end{align*}
Here we use \eqref{eq-b-property} that $b_0(y)$ is a strictly monotonic function. Then we get the existence of $\phi_1(y,\fc_r)$ and the estimates \eqref{eq-phi1-est}.

Similarly, we introduce operator $T_2$ and write \eqref{eq-phi2-exp} as
\begin{align*}
  \phi_2(y,\fc)=&1-2i{\fc_i}\int^y_{y_\fc} \frac{1}{\big(b_0(y')-\fc\big)^2\phi_1^2(y')}\int^{y'}_{y_\fc}\frac{b'_0(z)(b_0(z)-\fc)}{b_0(z)-\fc_r}\phi_1(z,\fc_r)\phi_1'(z,\fc_r)\phi_2(z,\fc)dzdy'\\
  =&1+T_2\phi_2(y,\fc).
\end{align*}
Recalling that $F(y,\fc_r)=\frac{\phi_1'(y,\fc_r)}{\phi_1(y,\fc_r)}$, we have
\begin{align*}
  &\left|\int^{y'}_{y_\fc}\frac{b_0'(z)(b_0(z)-\fc)}{b_0(z)-\fc_r}\phi_1(z,\fc_r)\phi_1'(z,\fc_r)\phi_2(z,\fc)dz\right|\\
  \le&\left|\int^{y'}_{y_\fc}b_0'(z)\phi_1(z,\fc_r)\phi_1'(z,\fc_r)\phi_2(z,\fc)dz\right|+{\fc_i}\left|\int^{y'}_{y_\fc}\frac{b_0'(z)F}{b_0(z)-\fc_r}\phi_1^2(z,\fc)\phi_2(z,\fc)dz\right|.
\end{align*}
As $\phi_1'(y,\fc_r)>0$ for $y>y_\fc$, and $\phi_1'(y,\fc_r)<0$ for $y<y_\fc$, we have
\begin{align*}
  &\left|\int^{y}_{y_\fc}b_0'(z)\phi_1(z,\fc_r)\phi_1'(z,\fc_r)\phi_2(z,\fc)dz\right|\\
  \le& C \|\phi_2\|_{L^\infty}\left|\int^{y}_{y_\fc}\phi_1(z,\fc_r)\phi_1'(z,\fc_r)dz\right|=C\|\phi_2\|_{L^\infty} \big(\phi_1^2(y,\fc_r)-1\big),
\end{align*}
and
\begin{align*}
  &\left|\int^{y}_{y_\fc}\frac{b_0'(z)F}{b_0(z)-\fc}\phi_1^2(z,\fc_r)\phi_2(z)dz\right|\le C \|\phi_2\|_{L^\infty}\int^{y}_{y_\fc}\phi_1^2(z,\fc_r)dz\\
  \le& \|\phi_2\|_{L^\infty}\int^{y_\fc+1}_{y_\fc}\phi_1^2(z,\fc_r)dz+C\|\phi_2\|_{L^\infty}\left|\int^{y}_{y_\fc+1}\phi_1(z,\fc_r)\phi_1'(z,\fc_r)dz\right|\\
  \le& C\|\phi_2\|_{L^\infty} \big(\phi_1^2(y,\fc_r)+C\big).
\end{align*}
Here we use the fact that $|F(y,\fc_r)|\le\min(|y-y_\fc|,1)$ and $|F(y,\fc_r)|\ge C_4>0$ for $|y-y_\fc|\ge 1$. 

We have $\big(b_0(y')-\fc\big)^2=\big(b_0(y')-\fc_r\big)^2-\fc_i^2+2i(b_0(y')-\fc_r)\fc_i$. It follows that
\begin{align*}
  \int^y_{y_\fc} \frac{C\|\phi_2\|_{L^\infty} \big(\phi_1^2(y',\fc)-1\big)}{\big(b_0(y')-\fc\big)^2\phi_1^2(y')}dy'\le C\|\phi_2\|_{L^\infty},
\end{align*}
and
\begin{align*}
  \int^y_{y_\fc} \frac{C\|\phi_2\|_{L^\infty} \big(\phi_1^2(y',\fc)+C\big)}{\big(b_0(y')-\fc\big)^2\phi_1^2(y',\fc_r)}dy'\le \frac{C}{{\fc_i}}\|\phi_2\|_{L^\infty}.
\end{align*}
Combing \eqref{eq-phi2-exp} and the above two inequalities, we have 
\begin{align*}
  \|T_2\phi_2\|_{L^\infty}\le C{\fc_i}\|\phi_2\|_{L^\infty}.
\end{align*}
For $\fc_i$ small enough, we have that $I-T_2$ is invertible in $L^\infty_y$. Thus there exists
\begin{align*}
  \phi_2(y,\fc)=(I-T_2)^{-1}1,   
\end{align*}
such that
\begin{align*}
  \|\phi_2(y,\fc)-1\|_{L^\infty}\le C{\fc_i}.
\end{align*}

Taking the derivative of \eqref{eq-phi2-exp}, it holds that
\begin{align*}
  \phi_2'(y,\fc)=-2i{\fc_i} \frac{1}{\big(b_0(y)-\fc\big)^2\phi_1^2(y)}\int^{y}_{y_\fc}\frac{b_0'(z)(b_0(z)-\fc)}{b_0(z)-\fc_r}\phi_1(z,\fc_r)\phi_1'(z,\fc_r)\phi_2(z)dz.
\end{align*}
As $|F(y,\fc_r)|\le \min(1,|y-y_\fc|)$, we have
\begin{align*}
  \left|\int^{y}_{y_\fc}\frac{b_0'(z)F}{b_0(z)-\fc_r}\phi_1^2(z,\fc)\phi_2(z)dz\right|\le C \|\phi_2\|_{L^\infty}\int^{y}_{y_\fc}\phi_1^2(z,\fc)dz\le C|y-y_\fc|, \text{ for }|y-y_\fc|\le 1.
\end{align*}
Then by a similar argument to $\|\phi_2(y)-1\|_{L^\infty}$, one can easily deduce that
\begin{align*}
  |\phi_2'(y)|\le C{\fc_i},
\end{align*}
and $|\phi_2'(y)|$ decay to $0$ as $|y|\to+\infty$.

We write from \eqref{eq-phi2} that
\begin{align*}
  \phi_2''=-\frac{2b_0'\phi_2'}{b_0-\fc}-2 \frac{\phi_1'}{\phi_1}\phi_2'-\frac{2i{\fc_i} b_0'\phi_1'}{(b_0-\fc)(b_0-\fc_r)\phi_1}\phi_2,
\end{align*}
from which we can see that $\|\phi_2''\|_{L^{\infty}}\le C$. Then we get the existence of $\phi_2(y,\fc)$ and the estimates \eqref{eq-phi2-est}.

The existence of $\phi(y,\fc)$ and the estimate \eqref{eq-phi-est} follow immediately.
\end{proof}
Then we give the proof of Lemma \ref{lem-eigen}.
\begin{proof}[Proof of Lemma \ref{lem-eigen}]
  From Proposition \ref{prop-phi}, since $\phi(y,\fc)=\big(b_0(y)-\fc\big)\phi_1(y,\fc_r)\phi_2(y,\fc)$ is a solution to \eqref{eq-Rayleigh}, then 
\begin{align*}
  \varphi^-(y,\fc)=&\phi(y,\fc)\int^y_{-\infty}\frac{1}{\phi^2(y',\fc)}dy'
\end{align*}
is a solution to \eqref{eq-Rayleigh} which is independent of $\phi(y,\fc)$. If $\fc$ is an eigenvalue, and $\psi(y,\fc)\in H^2_y$ is the corresponding eigenfunction, then by \eqref{eq-psi-eigenfunction} and the fact that as $y\to-\infty$, $\phi(y,\fc)\to +\infty$ and $\varphi^-(y,\fc)\to0$, then necessarily $a_1^-$ in \eqref{eq-psi-eigenfunction} is $0$ and $\lim_{y\to+\infty} \varphi^-(y,\fc)=0$. Then it follows from the fact $|\phi(y,\fc)|\ge  C|y-y_\fc|e^{C_4|y-y_\fc|}$ that $\mathcal D(\fc)=0$. 

Next we show that if $\mathcal D(\fc)=0$, then $\varphi^-(y,\fc)\in H^1_y$. We claim that as $y\to-\infty$, $\varphi^-(y,\fc)$ and ${\varphi^{-}}'(y,\fc)$ decay to $0$ exponentially. Indeed, by the fact that $\phi_1(y',\fc_r)\geq \phi_1(y,\fc_r)\geq e^{C_4(y_\fc-y)}$ for $y'\leq y\leq y_{\fc}-1$, thus we have,
\begin{equation}\label{eq-exp-decay-phi}
  \begin{aligned}    
\left|e^{\frac{C_4}{2}(y_\fc-y)}\varphi^{-}(y,\fc)\right|=&\left|e^{\frac{C_4}{2}(y_\fc-y)}\phi(y,\fc)\int^y_{-\infty}\frac{1}{\phi^2(y',\fc)}dy'\right|\\
\leq &\left|\int^y_{-\infty}\frac{1}{|(b(y')-\fc)^2||\phi_2^2(y',\fc)|^2e^{\frac{C_4}{3}(y_\fc-y')}}dy'\right|\leq C.
  \end{aligned}
\end{equation}
Similarly, we have for $y'\leq y\leq y_{\fc}-1$
 \begin{align*}
 \left|e^{\frac{C_4}{2}(y_\fc-y)}{\varphi^{-}}'(y,\fc)\right|
 \leq\left|\frac{e^{\frac{C_4}{2}(y_\fc-y)}}{\phi(y,\fc)}\right|+\left|e^{\frac{C_4}{2}(y_\fc-y)}F(y,\fc)\varphi^{-}(y,\fc)\right|\leq C.
\end{align*}
Here $C$ is a constant independent of $\fc_i$. 
As $\varphi^+(y,\fc)=\phi(y,\fc)\int_y^{+\infty}\frac{1}{\phi^2(y',\fc)}dy'$ and $\phi(y,\fc)$ are two independent solutions of \eqref{eq-Rayleigh}, there exists two constants $a_1$ and $a_2$ such that
\begin{align*}
  \varphi^-(y,\fc)=\phi(y,\fc)\int^y_{-\infty}\frac{1}{\phi^2(y',\fc)}dy'=a_1\phi(y,\fc)+a_2\phi(y,\fc)\int_y^{+\infty}\frac{1}{\phi^2(y',\fc)}dy',
\end{align*}
then,
\begin{align*}
  0=\mathcal D(\fc)=\lim_{y\to+\infty}\int^y_{-\infty}\frac{1}{\phi^2(y',\fc)}dy'=\lim_{y\to+\infty}\frac{\varphi(y,\fc)}{\phi(y,\fc)}=a_1+\lim_{y\to+\infty}a_2\int_y^{+\infty}\frac{1}{\phi^2(y',\fc)}dy'=a_1.
\end{align*}
Thus we have $\varphi^-(y,\fc)=a_2\phi(y,\fc)\int_y^{+\infty}\frac{1}{\phi^2(y',\fc)}dy'$. Similar to the case $y\to-\infty$, one can prove that as $y\to+\infty$, $\varphi^-(y,\fc)$ and ${\varphi^-}'(y,\fc)$ decay to $0$ exponentially. Therefore $\varphi^-\in H^1_y$. Then by the equation \eqref{eq-Rayleigh}, $\varphi^-\in H^2_y$, and $\fc$ is an eigenvalue of the Rayleigh operator $\mathcal R$.
\end{proof}
\begin{remark}\label{Rmk:D(c)wronskian}
Let $\varphi^{\pm}(y,\fc)=\phi(y,\fc)\int_{\pm\infty}^y\f{1}{\phi^2(y',\fc)}dy'$, then 
\beno
\mathcal{D}(\fc)=\int_{\mathbb{R}}\f{1}{\phi^2(y,\fc)}dy=\det\left(\begin{matrix}\varphi^{-}(y,\fc)&\varphi^{+}(y,\fc)\\
 \pa_y\varphi^{-}(y,\fc)&\pa_y\varphi^{+}(y,\fc)\end{matrix}\right)
 \eeno 
 is the Wronskian. {{} There exists a non-zero function $f(\fc)$ such that $f(\fc)\mathcal{D}(\fc)$ is analytic for $0<|\fc_i|\le \varepsilon_3$, where $\varepsilon_3$ is given in Proposition \ref{prop-phi}.}
\end{remark}{{}
\begin{proof}
  We claim that there exists constant $Y>0$ such that for each $y$ and $0<|\fc_i|\le \varepsilon_3$, $\varphi^{+}(Y,\fc)\neq 0$, $\varphi^{-}(-Y,\fc)\neq 0$, and both $\frac{\varphi^{+}(y,\fc)}{\varphi^{+}(Y,\fc)}$ and $\frac{\varphi^{-}(y,\fc)}{\varphi^{-}(-Y,\fc)}$  are analytic.

  Therefore, by taking $f(\fc)=\frac{1}{\varphi^{+}(Y,\fc)\varphi^{-}(-Y,\fc)}$, we have that
\begin{equation*}
  \frac{1}{\varphi^{+}(Y,\fc)\varphi^{-}(-Y,\fc)}\mathcal{D}(\fc)=\text{det}\left(
    \begin{array}{ll}
      \frac{\varphi^{-}(y,\fc)}{\varphi^{-}(-Y,\fc)}&\frac{\varphi^{+}(y,\fc)}{\varphi^{+}(Y,\fc)}\\
      \frac{\pa_y\varphi^{-}(y,\fc)}{\varphi^{-}(-Y,\fc)}&\frac{\pa_y\varphi^{+}(y,\fc)}{\varphi^{+}(Y,\fc)}
    \end{array}
  \right)
\end{equation*}
is analytic.

Next, we provide the proof of the claim, focusing only on the part related to function $\varphi^{-}$. By using the same techniques to Lemma \ref{lem-E1E2}, one can see that their exists $Y$ such that the equation \eqref{eq-Rayleigh} with boundary condition
\begin{align}\label{eq-boun-con-psi}
  \lim_{y\to-\infty}\psi(y)=0,\quad \psi(-Y)=0,
\end{align}
has only zero solution. It follows that $\varphi^{-}(-Y,\fc)\neq 0$. We can see that $\pa_{\bar{\fc}} \frac{\varphi^{-}(y,\fc)}{\varphi^{-}(-Y,\fc)}$ also satisfies \eqref{eq-Rayleigh} and the boundary condition $\pa_{\bar{\fc}} \frac{\varphi^{-}(y,\fc)}{\varphi^{-}(-Y,\fc)}\Big|_{y=-Y}=0$. Provided $\lim_{y\to-\infty}\pa_{\bar{\fc}} \frac{\varphi^{-}(y,\fc)}{\varphi^{-}(-Y,\fc)} =0$, it holds that $\pa_{\bar{\fc}} \frac{\varphi^{-}(y,\fc)}{\varphi^{-}(-Y,\fc)} \equiv0$. Then we have $\frac{\varphi^{-}(y,\fc)}{\varphi^{-}(-Y,\fc)} $ is analytic. Therefore, to prove the claim, it suffices to show that $\lim_{y\to-\infty}\pa_{\bar{\fc}} \frac{\varphi^{-}(y,\fc)}{\varphi^{-}(-Y,\fc)} =0$.

We introduce a good derivative $\pa_G=\pa_y+\pa_{y_{\fc}}$. A direct calculation gives
  \begin{equation*} 
    \left\{
      \begin{array}{ll}
        &\pa_y\left((b_0(y)-b_0(y_{\fc}))^2\phi_{1}^2(y,\fc_r)\pa_y\Big(\frac{\pa_G\phi_{1}(y,\fc_r)}{\phi_{1}(y,\fc_r)}\Big)\right)\\
        &=-\pa_y\left(\frac{2(b'_0(y)-b_0'(y_{\fc}))}{b_0(y)-b_0(y_{\fc})}\right)(b_0(y)-b_0(y_{\fc}))^2\phi_{1}(y,\fc_r)\pa_y\phi_{1}(y,\fc_r),\\
      & \pa_G\phi_{1}(y_{\fc},\fc_r)=0,\quad\pa_y \pa_G\phi_{1}(y_{\fc},\fc_r)=0.
      \end{array}
    \right.
  \end{equation*}
Then we have
\begin{align*}
  \frac{\pa_G\phi_{1}(y,\fc_r)}{\phi_{1}(y,\fc_r)}=&-2\int_{y_{\fc}}^y\frac{1}{(b_0(y')-b_0(y_{\fc}))^2\phi_{1}^2(y',\fc_r)}\\
&\quad \times\int_{y_{\fc}}^{y'}\pa_{z}\left(\frac{(b'_0(z)-b_0'(y_{\fc}))}{b_0(z)-b_0(y_{\fc})}\right)(b_0(z)-b_0(y_{\fc}))^2\phi_{1}(z,\fc_r)\pa_y\phi_{1}(z,\fc_r)dzdy'.
\end{align*}
By using Proposition \ref{prop-phi}, one can easily check that
\begin{align}\label{eq-est-G-phi1}
  \left|\frac{\pa_G\phi_{1}(y,\fc_r)}{\phi_{1}(y,\fc_r)}\right|\le C|y-y_{\fc}|^3,\quad\left|\frac{\pa_y\pa_G\phi_{1}(y,\fc_r)}{\phi_{1}(y,\fc_r)}\right|\le C|y-y_{\fc}|^2.
\end{align}

Taking derivative of \eqref{eq-phi2-exp} with respect to $y+y_{\fc}$ and $\fc_i$, we get the integral equations of $\pa_G\phi_{2}(y,\fc)$ and $\pa_{\fc_i}\phi_{2}(y,\fc)$ respectively. By employing similar techniques used for proving Proposition \ref{prop-phi}, one could get that
\begin{align}\label{eq-est-G-phi2}
  \left|\pa_G\phi_{2}(y,\fc)\right|+\left|\pa_{\fc_i}\phi_{2}(y,\fc)\right|\le C|y-y_{\fc}|.
\end{align}

By utilizing the estimates \eqref{eq-est-G-phi1} and \eqref{eq-est-G-phi2}, we can employ the approach in the proof of Lemma \ref{lem-eigen} to get the estimate for $y\le y_{\fc}-1$ that
\begin{align*}
  \left|\pa_{\bar{\fc}}\frac{\varphi^{-}(y,\fc)}{\varphi^{-}(-Y,\fc)}\right|\le Ce^{-c|y-y_{\fc}|}.
\end{align*}

Note that $\varepsilon_3$ depends only on the upper and lower bound of $b_0'$, so for small enough $\gamma$, we will always have $\varepsilon_3\ge 8\sqrt\pi M\gamma$.
This finishes the proof of this remark.
\end{proof}}
\begin{proof}[Proof of Lemma \ref{lem-lim-eigen}]
We write
\begin{align*}
  \mathcal D(\fc)=&\int^{+\infty}_{-\infty}\frac{1}{(b_0(y)-\fc)^2}dy+\int^{+\infty}_{-\infty}\frac{1}{(b_0(y)-\fc)^2}\left(\frac{1}{\phi^2_1(y,\fc_r)\phi^2_2(y,\fc)}-1\right)dy\eqdef I+II.
\end{align*}
From \eqref{eq-phi1-est} and the Lebesgue-dominated convergence theorem, we get
\begin{align*}
  \lim_{{\fc_i}\to0}II=&\lim_{{\fc_i}\to0}\int^{+\infty}_{-\infty}\frac{1}{(b_0(y)-\fc_r-i\fc_i)^2}\left(\frac{1}{\phi^2_1(y,\fc_r)\phi^2_2(y,\fc_r+i\fc_i)}-1\right)dy\\
  =&\int^{+\infty}_{-\infty}\frac{1}{(b_0(y)-\fc_r)^2}\left(\frac{1}{\phi^2_1(y,\fc_r)}-1\right)dy=\Pi_2(\fc_r).
\end{align*}
Let $v=b_0(y)$,  we have
\begin{align*}
  I=&\int^{+\infty}_{-\infty}\frac{1}{(b_0(y)-\fc)^2}dy=\int^{+\infty}_{-\infty}\frac{\pa_v\big(b_0^{-1}\big)(v)}{(v-\fc)^2}dv=\int^{+\infty}_{-\infty}\frac{\pa_v^2\big(b_0^{-1}\big)(v)}{v-\fc}dv\\
=&\int^{+\infty}_{-\infty} \frac{(v-\fc_r)\pa_v^2\big(b_0^{-1}\big)(v)}{(v-\fc_r)^2+{\fc_i}^2}dv +\int^{+\infty}_{-\infty} \frac{i{\fc_i}\pa_v^2\big(b_0^{-1}\big)(v)}{(v-\fc_r)^2+{\fc_i}^2}dv\eqdef I_r+iI_i.
\end{align*}
From the properties of the Hilbert transform and the Poisson kernel, we know that
  \begin{align}    \label{eq-lim-fci}
      &\lim_{{\fc_i}\to0\pm}I=-\mathcal H\big(\pa_v^2(b_0^{-1})\big)(\fc_r)\pm i\pi\pa_v^2\big(b_0^{-1}\big)(\fc_r)\\
 %\nonumber =&-\text{P.V.}\int_{\mathbb R} \frac{1}{v-\fc_r}\frac{b_0''(b_0^{-1}(v))}{\big(b_0'(b_0^{-1}(v))\big)^3} dv\mp i\pi\frac{b_0''(y_\fc)}{\big(b_0'(y_\fc)\big)^3}\\
 \nonumber  =&\frac{1}{b_0'(y_\fc)}\text{P.V.}\int_{\mathbb R} \frac{1}{v-\fc_r}\pa_v\left(\frac{b_0'(y_\fc)-b_0'(b_0^{-1}(v))}{b_0'(b_0^{-1}(v))} \right)dv\mp i\pi\frac{b_0''(y_\fc)}{\big(b_0'(y_\fc)\big)^3}\\
 \nonumber  =&\frac{1}{b_0'(y_\fc)}\text{P.V.}\int  \frac{b_0'(y_\fc)-b_0'(y)}{(b_0(y)-\fc)^2} dy\mp i\pi\frac{b_0''(y_\fc)}{\big(b_0'(y_\fc)\big)^3} 
 =\frac{1}{b_0'(y_\fc)}\Pi_1(\fc_r)\mp i\mathcal J_2(\fc_r),
  \end{align}
which gives the lemma.  
\end{proof}
\begin{proof}[Proof of Lemma \ref{lem-iff-emb}]
  We first show that if $\fc_r\in\mathbb R$ is an eigenvalue, then $b_0''(y_\fc)=0$. Let $\psi(y,\fc_r)\in H^1_y$ be the corresponding eigenfunction, namely, the solution of \eqref{eq-Rayleigh}. If $b_0''(y_\fc)\neq0$, we can see from \eqref{eq-Rayleigh} that $\psi(y_\fc)=0$. Then by taking inner product with $\frac{\psi}{b_0-\fc_r}$ on both sides of \eqref{eq-Rayleigh} and integration by parts, we have
\begin{align*}
  \int_{\mathbb R}\left|\psi'-\frac{b_0'\psi}{b_0-\fc_r}\right|^2 d y+\int_{\mathbb R}|\psi|^2 d y=0,
\end{align*}
which implies that $\psi\equiv0$. Thus, if $\fc_r$ is an eigenvalue, $b_0''(y_\fc)=0$ and $\mathcal J_2(\fc_r)=0$. Therefore, we only need to study $\mathcal J_1(\fc_r)$ under the assumption that $b_0''(y_\fc)=0$.

For $\fc_r\in\mathbb R$, $\phi_2(y,\fc_r)\equiv1$ and $\phi(y,\fc_r)=\phi_1(y,\fc_r)\big(b_0(y)-\fc_r\big)$ is a solution to \eqref{eq-Rayleigh}, then 
\begin{align*}
  \varphi^-(y,\fc_r)=&\phi(y,\fc_r)\int^y_{-\infty}\frac{1}{\phi^2(y',\fc_r)}dy'
\end{align*}
is another solution of \eqref{eq-Rayleigh}.

We write
\begin{align*}
  &\varphi^-(y,\fc_r)=\phi(y,\fc_r)\int^y_{-\infty}\frac{1}{\phi^2(y',\fc_r)}dy'\\
  =&\phi(y,\fc_r)\int^y_{-\infty}\frac{1}{(b_0(y')-\fc_r)^2}dy'+\phi(y,\fc_r)\int^y_{-\infty}\frac{1}{(b_0(y')-\fc_r)^2}\left(\frac{1}{\phi_1^2(y',\fc_r)}-1\right)dy'\\
  %=&\phi(y,\fc_r)\int^y_{-\infty}\frac{1}{(b_0(y')-\fc_r)^2}dy'+\phi(y,\fc_r)\int^y_{-\infty}\frac{1}{(b_0(y')-\fc_r)^2}\left(\frac{(1-\phi_1(y',\fc_r))(1+\phi_1(y',\fc_r))}{\phi_1^2(y',\fc_r)}\right)dy'\\
  \eqdef&I+II.
\end{align*}
As $\phi_1(y_\fc,\fc_r)=1$ and $\phi_1'(y_\fc,\fc_r)=0$, the integral function in $II$ does not have singularity at $y_\fc$, then $II$ is well defined on $\mathbb R$. We also deduce that 
  \begin{align}  \label{eq-phi-I}  
      I=&\frac{\phi(y,\fc_r)}{b_0'(y_\fc)}\int^{y}_{-\infty}\frac{b_0'(y_\fc)-b_0'(y')}{(b_0(y')-\fc_r)^2}dy'+\frac{\phi(y,\fc_r)}{b_0'(y_\fc)}\int^{y}_{-\infty}\frac{b_0'(y')}{(b_0(y')-\fc_r)^2}dy'\\
 %\nonumber =&\frac{\phi(y,\fc_r)}{b_0'(y_\fc)}\int^{y}_{-\infty}\frac{b_0'(y_\fc)-b_0'(y')}{(b_0(y')-\fc_r)^2}dy'-\phi(y,\fc_r)\frac{1}{b_0'(y_\fc)}\frac{1}{b_0(y')-\fc_r}\Big|^{y}_{-\infty}\\
 \nonumber =&\frac{\phi(y,\fc_r)}{b_0'(y_\fc)}\int^{y}_{-\infty}\frac{b_0'(y_\fc)-b_0'(y')}{(b_0(y')-\fc_r)^2}dy'-\phi(y,\fc_r)\frac{1}{b_0'(y_\fc)}\frac{1}{b_0(y)-\fc_r}\\
 \nonumber =&\frac{\phi(y,\fc_r)}{b_0'(y_\fc)}\int^{y}_{-\infty}\frac{b_0'(y_\fc)-b_0'(y')}{(b_0(y')-\fc_r)^2}dy'-\frac{\phi_1(y,\fc_r)}{b_0'(y_\fc)}.
  \end{align}
%Here we use the fact that
%\begin{align*}
%  \lim_{y'\to-\infty}\phi(y,\fc_r)\frac{1}{b_0'(y_\fc)}\frac{1}{b_0(y')-\fc_r}=0.
%\end{align*}
For the first term on the right hand side of \eqref{eq-phi-I}, as $b_0''(y_\fc)=0$, we have
\begin{align*}
  \frac{\phi(y,\fc_r)}{b_0'(y_\fc)}\int^{y}_{-\infty}\frac{b_0'(y_\fc)-b_0'(y')}{(b_0(y')-\fc_r)^2}dy' =\frac{\phi(y,\fc_r)}{b_0'(y_\fc)}\int^{y}_{-\infty}\frac{(y'-y_\fc)^2}{(b_0(y')-\fc_r)^2}\frac{\frac{b_0'(y_\fc)-b_0'(y')}{y'-y_\fc}-b_0''(y_\fc)}{y'-y_\fc}dy'.
\end{align*}
We can see that the integrand is not singular, so $I$ is well defined at $y_\fc$ and then on $\mathbb R$. Then by using the same argument in Lemma \ref{lem-eigen}, one can prove that $\varphi\in H^1_y$ if and only if $\mathcal J_1(\fc_r)=0$. As conclusion, $\fc_r\in \mathbb R$ is an embedded eigenvalue of $\mathcal R$ if and only if 
\begin{align*}
  \mathcal J_1^2(\fc_r)+\mathcal J_2^2(\fc_r)=0.
\end{align*}
Thus we proved the lemma. 
\end{proof}
{{}
\begin{proof}[Proof of Proposition \ref{prop-repr}]
First, we provide the rigorous deduction of the identity \eqref{eq-rep-Psi}. Given $\mathfrak w_{in}\in C_c^{\infty}(\mathbb{R})$, $M$, and $\gamma$. Recall \eqref{eq: LinearEuler-Psi}. By the standard energy method, one can show that there exists $c^*>0$ such that
\begin{equation}\label{eq-grow-est1}
   \begin{aligned}    
  \mathfrak w(t,y)\le C e^{c^*t}, \quad\pa_t\mathfrak w(t,y)\le C e^{c^*t},\text{ for }t\ge0,y\in\mathbb R,\\
  \Psi(t,y)\le C e^{c^*t}, \quad\pa_t\Psi(t,y)\le C e^{c^*t},\text{ for }t\ge0,y\in\mathbb R.     
   \end{aligned}
 \end{equation} 
Applying the Fourier-Laplace transform on the evolution equation of the stream function
\begin{align}\label{eq: LinearEuler-Psi2}
  \pa_t\Psi(t,y)+i\mathcal L\Psi(t,y)=0,\quad\Psi(0,y)=\Psi_{in}(y)
\end{align}
we have
\begin{align*}
  \tilde\fc \Psi^*(\tilde\fc,y)+i\mathcal L\Psi^*(\tilde\fc,y)=\Psi_{in}(y),\text{ for }\tilde\fc_r> c^*,
\end{align*}
where $\tilde\fc=\tilde\fc_r+i\tilde\fc_i$ and
\begin{align*}
  \Psi^*(\tilde\fc,y)=\int^{+\infty}_0e^{-\tilde\fc t}\Psi(t,y)dt.
\end{align*}
It follows that
\begin{align}\label{eq-resolvent1}
  \Psi^*(\tilde\fc,y)=\left(\tilde\fc+i\mathcal L \right)^{-1}\Psi_{in}(y),\text{ for }\tilde\fc_r> c^*.
\end{align}
We remark that here $\tilde\fc$ is different to the notation $\fc$ used in \eqref{eq-rep-Psi}. Here we introduce this new notation $\tilde\fc$ to be consistent with the classical theory of the Fourier-Laplace transform.

By the inverse Fourier-Laplace transform (which is known by various names, the Bromwich integral, and Mellin's inverse formula), we have (see, \cite{Terras1985, HP1957})
\begin{equation}\label{eq-inver-1}
  \frac{1}{2\pi i}\lim_{T\to+\infty} \int^{c^*+iT}_{c^*-iT} e^{\tilde\fc t}\Psi^*(\tilde\fc,y) d\tilde\fc=\left\{
    \begin{array}{ll}
      \Psi(t,y),&\text{ for }t>0,\\
      \frac{1}{2}\Psi_{in}(y),&\text{ for }t=0,\\
      0,&\text{ for }t<0.
    \end{array}
  \right.
\end{equation}
So, at $t=0$, the inverse Fourier-Laplace transform fails to recover the original function. To address this issue, we introduce $\underline{\Psi}(t',y)$ which satisfies
\begin{align}\label{eq: LinearEuler-Psi3}
  \pa_{t'}\underline{\Psi}(t',y)-i\mathcal L\underline{\Psi}(t',y)=0,\quad\Psi(0,y)=\Psi_{in}(y).
\end{align}
One can regard $\underline{\Psi}(t',y)=\Psi(-t',y)$ for $t'\ge0$, and \eqref{eq: LinearEuler-Psi3} as the time-backward extension of \eqref{eq: LinearEuler-Psi2}. It is clear that \eqref{eq-grow-est1} also holds for $\underline{\Psi}(t',y)$. Similar to $\Psi(t,y)$, we also have 
\begin{align}\label{eq-resolvent2}
  \underline{\Psi}^*(\tilde\fc',y)=\left(\tilde\fc'-i\mathcal L \right)^{-1}\Psi_{in}(y),\text{ for }\tilde\fc_r'> c^*.
\end{align}
and
\begin{equation}\label{eq-inver-2}
  \frac{1}{2\pi i}\lim_{T\to+\infty} \int^{c^*+iT}_{c^*-iT} e^{\tilde\fc' t'}\underline{\Psi}^*(\tilde\fc',y) d\tilde\fc'=\left\{
    \begin{array}{ll}
      \underline{\Psi}(t',y),&\text{ for }t'>0,\\
      \frac{1}{2}\Psi_{in}(y),&\text{ for }t'=0,\\
      0,&\text{ for }t'<0.
    \end{array}
  \right.
\end{equation}
Let $t=-t'$ and $\tilde \fc=-\tilde\fc'$, we write \eqref{eq-inver-2} as
\begin{equation}\label{eq-inver-3}
  \frac{1}{2\pi i}\lim_{T\to+\infty} \int^{-c^*+iT}_{-c^*-iT} e^{\tilde\fc t}\underline{\Psi}^*(-\tilde\fc,y) d\tilde\fc=\left\{
    \begin{array}{ll}
      0,&\text{ for }t>0\\
      \frac{1}{2}\Psi_{in}(y),&\text{ for }t=0,\\
      \Psi(t,y),&\text{ for }t<0.
    \end{array}
  \right.
\end{equation}
Combing \eqref{eq-resolvent1}, \eqref{eq-inver-1}, \eqref{eq-resolvent2}, and \eqref{eq-inver-3}, we have for $t\in\mathbb R$ that
\begin{equation}
  \begin{aligned}    
   \Psi(t,y)=&\frac{1}{2\pi i}\lim_{T\to+\infty} \int^{c^*+iT}_{c^*-iT} e^{\tilde\fc t}\left(\tilde\fc+i\mathcal L \right)^{-1}\Psi_{in}(y) d\tilde\fc\\
  &-\frac{1}{2\pi i}\lim_{T\to+\infty} \int^{-c^*+iT}_{-c^*-iT} e^{\tilde\fc t}\left(\tilde\fc+i\mathcal L \right)^{-1}\Psi_{in}(y) d\tilde\fc   
  \end{aligned}
\end{equation}
Let $\fc=i\tilde \fc$, we have
\begin{equation}
  \begin{aligned}    
   \Psi(t,y)=&\frac{1}{2\pi i}\lim_{T\to+\infty} \int^{T}_{-T} e^{-i \left(\fc_r-ic^*\right) t}\left(\fc_r-ic^*-\mathcal L \right)^{-1}\Psi_{in}(y) d\fc_r\\
  &-\frac{1}{2\pi i}\lim_{T\to+\infty} \int^{T}_{-T} e^{-i \left(\fc_r+ic^*\right) t}\left(\fc_r+ic^*-\mathcal L \right)^{-1}\Psi_{in}(y) d\fc_r,
  \end{aligned}
\end{equation}
which is consistent with the form of \eqref{eq-rep-Psi}.

By our assumption that $\mathcal R$ has no eigenvalues, we have $\mathcal L$ has no eigenvalues. Moreover, by a classical compactness argument, it holds that under such an assumption, $\mathbb C\setminus\mathbb R$ is in the resolvent set of $\mathcal L$, namely $\mathbb C\setminus\mathbb R\subset \mathbb C\setminus \s(\mathcal L)$. Thus, $\left(\fc-\mathcal L \right)^{-1}$ is an analytic operator-value function of $\fc$ on $\mathbb C\setminus\mathbb R$. Recall that $\Phi(y,\fc)=-i(\fc-\mathcal{L})^{-1}\Psi_{in}\in H_y^2(\mathbb{R})$ which satisfies \eqref{eq-Phi-ori}, and $\mathfrak w_{in}\in C_c^{\infty}(\mathbb{R})$. By applying the same technique as in Lemma \ref{lem-E1E2}, one can easily check that
\begin{align*}
  \left\|\Phi(y,\fc)\right\|_{H^1_y}\le C \frac{1}{|\fc_r|+1},
\end{align*}
where $C$ is uniformly in $\fc_i$.

It follows that for any $c^{**}>0$, we have
\begin{align*}
  &\lim_{T\to+\infty} \int^{T}_{-T} e^{-i \left(\fc_r+ic^{**}\right) t}\left(\fc_r+ic^{**}-\mathcal L \right)^{-1}\Psi_{in}(y) d\fc_r\\
  =&\lim_{T\to+\infty} \int^{T}_{-T} e^{-i \left(\fc_r+ic^*\right) t}\left(\fc_r+ic^*-\mathcal L \right)^{-1}\Psi_{in}(y) d\fc_r\\
  &+\lim_{T\to+\infty} i\int^{c^{**}}_{c^*} e^{-i \left(T+i\fc_i\right) t}\left(T+i\fc_i-\mathcal L \right)^{-1}\Psi_{in}(y) d\fc_i\\
  &+\lim_{T\to+\infty} i\int^{c^*}_{c^{**}} e^{-i \left(-T+i\fc_i\right) t}\left(-T+i\fc_i-\mathcal L \right)^{-1}\Psi_{in}(y) d\fc_i\\
  =&\lim_{T\to+\infty} \int^{T}_{-T} e^{-i \left(\fc_r+ic^*\right) t}\left(\fc_r+ic^*-\mathcal L \right)^{-1}\Psi_{in}(y) d\fc_r.
\end{align*}
Therefore,
\begin{equation}\label{eq-rep-Psi2}
  \begin{aligned}    
   \Psi(t,y)=&\lim_{\fc_i\to0+}\frac{1}{2\pi i}\lim_{T\to+\infty} \int^{T}_{-T} e^{-i \left(\fc_r-i\fc_i\right) t}\left(\fc_r-i\fc_i-\mathcal L \right)^{-1}\Psi_{in}(y) d\fc_r\\
  &-\lim_{\fc_i\to0+}\frac{1}{2\pi i}\lim_{T\to+\infty} \int^{T}_{-T} e^{-i \left(\fc_r+i\fc_i\right) t}\left(\fc_r+i\fc_i-\mathcal L \right)^{-1}\Psi_{in}(y) d\fc_r.
  \end{aligned}
\end{equation}

We will prove that \eqref{eq-rep-Psi2} is equivalent to \eqref{eq-rep-Psi-lim}. 

We first show that for given $\mathfrak w_{in}\in C_c^{\infty}(\mathbb{R})$, $M$, $\gamma$, and fixed $y\in\mathbb R$, there exists $\varepsilon_4>0$ small enough such that for any $0<|\fc_i|\le \varepsilon_4$, $\left|\mu(\mathfrak w_{in},\fc)\right|$ has a uniform upper bound. 

We write
\begin{align*}
  \mathcal D(\fc)=&\int_{-\infty}^{+\infty}\frac{1}{\phi^2(y',\fc)}dy'\\
  =&\int^{+\infty}_{-\infty}\frac{1}{\big(b_0(y')-\fc_r-i\fc_i\big)^2}dy' 
  +\int^{+\infty}_{-\infty}\frac{1}{\big(b_0(y')-\fc_r-i\fc_i\big)^2} \left(\frac{1}{\phi_{1}^2(y',\fc_r)}-1\right)dy' \\
  &+\int^{+\infty}_{-\infty}\frac{1}{\big(b_0(y')-\fc_r-i\fc_i\big)^2}\frac{1}{\phi_{1}^2(y',\fc_r)} \left(\frac{1}{\phi_{2}^2(y',\fc)}-1\right)dy'\\
  \eqdef&\tilde I_1+\tilde I_2+\tilde I_3.
\end{align*}
From \eqref{eq-phi2-est} we have $\left|\phi_2(y,\fc)-1\right|\le C|\fc_i|^{\frac{1}{2}}|y-y_\fc|^{\frac{3}{2}}$, and then
\begin{align*}
  \left|\tilde I_3\right|\le&\int^{+\infty}_{-\infty}\frac{1}{\left|b_0(y')-\fc_r\right|^2}\frac{1}{\phi_{1}^2(y',\fc_r)} \frac{\left|\phi_{2}+1\right|\left|\phi_{2}-1\right|}{{\phi_{2}^\pm}^2(y',\fc_r)} dy'\\
  \le&C\int_{|y'-y_\fc|\le 1}\frac{1}{\left|b_0(y')-\fc_r\right|^2}\frac{1}{\phi_{1}^2(y',\fc_r)} \frac{ |\fc_i|^\frac{1}{2}|y-y_\fc|^{\frac{3}{2}}}{{\phi_{2}^\pm}^2(y',\fc_r)} dy'\\
  &+C \int_{|y'-y_\fc|>1}\frac{1}{\left|b_0(y')-\fc_r\right|^2}\frac{1}{\phi_{1}^2(y',\fc_r)} \frac{|\fc_i|}{{\phi_{2}^\pm}^2(y',\fc_r)} dy'
    \le C \left(|\fc_i|^\frac{1}{2}+|\fc_i|\right).
\end{align*}

We write
\begin{align*}
  \tilde I_2=&\int^{+\infty}_{-\infty}\frac{(b_0(y')-\fc_r)^2-\fc_i^2}{\big((b_0(y')-\fc_r)^2+\fc_i^2\big)^2}  \frac{(1+\phi_{1}(y',\fc_r))(1-\phi_{1}(y',\fc_r))}{\phi_{1}^2(y',\fc_r)} dy'\\
  &+i\int^{+\infty}_{-\infty}\frac{2\fc_i(b_0(y')-\fc_r) }{\big((b_0(y')-\fc_r)^2+\fc_i^2\big)^2}  \frac{(1+\phi_{1}(y',\fc_r))(1-\phi_{1}(y',\fc_r))}{\phi_{1}^2(y',\fc_r)} dy'\\
  \eqdef&\tilde I_{2,r}+i\tilde I_{2,i}.
\end{align*}
It holds that
\begin{align*}
  (b_0(y')-\fc_r)^2-\fc_i^2\ge0 \text{ for }|y'-y_\fc|\ge C|\fc_i|.
\end{align*}
Then we have
\begin{align*}
  \tilde I_{2,r}\le &\int_{|y'-y_\fc|\ge C|\fc_i|}\frac{(b_0(y')-\fc_r)^2-\fc_i^2}{\big((b_0(y')-\fc_r)^2+\fc_i^2\big)^2}  \frac{(1+\phi_{1}(y',\fc_r))(1-\phi_{1}(y',\fc_r))}{\phi_{1}^2(y',\fc_r)} dy'\\
  &+\int_{|y'-y_\fc|< C|\fc_i|}\frac{\left|(b_0(y')-\fc_r)^2-\fc_i^2\right|}{\big((b_0(y')-\fc_r)^2+\fc_i^2\big)^2}  \frac{(1+\phi_{1}(y',\fc_r))(1-\phi_{1}(y',\fc_r))}{\phi_{1}^2(y',\fc_r)} dy'\\
  \le& -C(1-|\fc_i|).
\end{align*}
Thus by taking $\varepsilon_4$ small enough, we have for $0<|\fc_i|\le \varepsilon_4$ that
\begin{align*}
  \tilde I_{2,r}+\left|\tilde I_3\right|\le -\widetilde C,
\end{align*}
where $\widetilde C>0$ is a constant independent of $\fc_r$ and $\fc_i$.

We write
\begin{align*}
  \tilde I_1=\int^{+\infty}_{-\infty} \frac{(v-\fc_r) \pa_v^2(b_0^{-1})(v) }{ (v-\fc_r)^2+\fc_i^2 }dv +i\int^{+\infty}_{-\infty} \frac{\fc_i \pa_v^2(b_0^{-1})(v) }{ (v-\fc_r)^2+\fc_i^2 }dv\eqdef \tilde I_{1,r}+i\tilde I_{1,i}.
\end{align*}
By the properties of the Poisson kernel and Hilbert transform, we have
\begin{align*}
  \tilde I_{1,r}=\int^{+\infty}_{-\infty} \frac{\fc_i \mathcal H \left(\pa_v^2(b_0^{-1})\right) (v) }{ (v-\fc_r)^2+\fc_i^2 }dv.
\end{align*}
As $\mathcal H \left(\pa_v^2(b_0^{-1})\right) (v)\in H^1_v(\mathbb R)$, and the Poisson kernel is an approximate identity,   there exists a constant $C_{\fc_r}$ independent of $\fc_i$ such that $|\tilde I_{1,r}|\le \frac{1}{2}\tilde C$ for $|\fc_i|\le \varepsilon_4$ and $|\fc_r|\ge C_{\fc_r}$. It follows that
\begin{align*}
  \left|\mathcal D(\fc)\right|\ge \frac{1}{2}\tilde C,\text{ for }|\fc_i|\le \varepsilon_4,|y_\fc|\ge C_{\fc_r}.
\end{align*}
From Lemma \ref{lem-lim-eigen} and Lemma \ref{lem-est-J1J2J3}, and the fact that $\mathcal R$ has no eigenvalue for $0<|\fc_i|\le\gamma$, we can extend $\mathcal D(\fc)$ to a continues function $\mathcal D^+(\fc)$ on the domain $D_\fc^+=\left\{\fc\big||\fc_r|\le C_{\fc_r}, 0\le \fc_i\le \varepsilon_4\right\}$. As $D_\fc^+$ is compact, and $\mathcal D^+(\fc)\neq0$ for all $\fc\in D_\fc^+$, there exists $\delta_*>0$ such that 
\begin{align*}
  \min_{\fc\in D_\fc^+}\left|\mathcal D^+(\fc)\right|\ge \delta_*.
\end{align*}
Using the same argument on the domain $D_\fc^-=\left\{\fc\big||\fc_r|\le C_{\fc_r}, 0\ge \fc_i\ge -\varepsilon_4\right\}$. Then we deduce that
\begin{align*}
  \left|\mathcal D(\fc)\right|\ge \delta_* \text{ for }|\fc_i|\le \varepsilon_4.
\end{align*}

Recall \eqref{eq-nume-mu} and \eqref{eq-I3-c} that
\begin{align*}
  \int_{-\infty}^{+\infty}\frac{\int_{y_{\fc}}^{y'}\mathfrak w_{in}(y'')\phi_1(y'',\fc)dy''}{(b_0(y')-\fc)^2\phi_1^2(y',\fc)}dy'=\mathcal I_1+\mathcal I_2+\mathcal I_{3,r}+i\mathcal I_{3,i}.
\end{align*}
It is clear that $\left|\mathcal I_{3,i}\right|\le C\left\|\mathfrak w_{in}\right\|_{L^\infty}$. By Proposition \ref{prop-phi}, we have
\begin{align*}
  \left|\frac{\int_{y_{\fc}}^{y'}\mathfrak w_{in}(y'')\big(\phi_1(y'',\fc)-1\big)dy''}{(b_0(y')-\fc)^2\phi_1^2(y',\fc)}\right|\le C \frac{\left|y'-y_\fc\right|^{\frac{1}{2}}\left\|\mathfrak w_{in}\right\|_{L^2}}{1+(y'-y_\fc)^2}.
\end{align*}
It follows that
\begin{align*}
  \left|\mathcal I_1\right|\le C \left\|\mathfrak w_{in}\right\|_{L^2}.
\end{align*}
For the same reason, we also have
\begin{align*}
  \left|\mathcal I_2\right|\le C \left\|\mathfrak w_{in}\right\|_{L^2}.
\end{align*}

We write 
\begin{align*}
  \mathcal I_{3,r}=&\int_{-\infty}^{\infty} \frac{v-\fc_r}{(v-\fc_r)^2+\fc_i^2} \frac{ \mathfrak w_{in}(b_0^{-1}(v))}{\left(b'_0(b_0^{-1}(v))\right)^2} d v\\
  &-\int_{-\infty}^{\infty} \frac{v-\fc_r}{(v-\fc_r)^2+\fc_i^2}\frac{b_0''(b_0^{-1}(v))\int_{b_0^{-1}(\fc_r)}^{b_0^{-1}(v)}\mathfrak w_{in}(y'')dy''}{\left(b'_0(b_0^{-1}(v))\right)^3}dv\\
  =&\int_{-\infty}^{\infty} \frac{\fc_i}{(v-\fc_r)^2+\fc_i^2} \mathcal H \left(\frac{ \mathfrak w_{in}(b_0^{-1}(v))}{\left(b'_0(b_0^{-1}(v))\right)^2}\right) d v\\
  &-\int_{-\infty}^{\infty} \frac{v-\fc_r}{(v-\fc_r)^2+\fc_i^2}\frac{b_0''(b_0^{-1}(v))\int_{b_0^{-1}(\fc_r)}^{b_0^{-1}(v)}\mathfrak w_{in}(y'')dy''}{\left(b'_0(b_0^{-1}(v))\right)^3}dv\\
  \eqdef&\mathcal I_{3,r}^{(1)}+\mathcal I_{3,r}^{(2)}.
\end{align*}
One can easily check that
\begin{align*}
  \left|\mathcal I_{3,r}^{(1)}\right| \le C \left\|\mathfrak w_{in}\right\|_{H^1},\quad  \left|\mathcal I_{3,r}^{(2)}\right| \le C \left\|\mathfrak w_{in}\right\|_{L^\infty}\left\| b''_0\right\|_{L^1}.
\end{align*}

Then we have 
\begin{align*}
  \left|\int_{-\infty}^{+\infty}\frac{\int_{y_{\fc}}^{y'}\mathfrak w_{in}(y'')\phi_1(y'',\fc)dy''}{(b_0(y')-\fc)^2\phi_1^2(y',\fc)}dy'\right|\le C\left\|\mathfrak w_{in}\right\|_{L^\infty},
\end{align*}
and then
\begin{align*}
  \left|\mu(\mathfrak w_{in},\fc)\right|\le C,
\end{align*}
where $C$ is a constant independent of $\fc_r$ and $\fc_i$.

Next, we show that $\Phi_{h,l}(y,\fc)\chi_{y_\fc>y}+\Phi_{h,r}(y,\fc)\chi_{y_\fc<y}$ has an integrable control function. 

From \eqref{eq-decay-part2}, we have for $y\le y_\fc-1$ that
\begin{align*}
  \left|\phi(y,\fc)\int_{+\infty}^y\frac{1}{\phi^2(y',\fc)}dy'\right|\le C e^{-C_4|y-y_{\fc}|}.
\end{align*}
For $y_\fc>y> y_\fc-1$, we have
\begin{align*}
  \phi(y,\fc)\int_{+\infty}^y\frac{1}{\phi^2(y',\fc)}dy'=&\phi(y,\fc)\int_{+\infty}^{y_\fc-1}\frac{1}{\phi^2(y',\fc)}dy'\\
  &+\phi(y,\fc)\int_{y_\fc-1}^y\frac{1}{\big(b_0(y')-\fc_r-i\fc_i\big)^2} \left(\frac{1}{\phi^2_1(y',\fc)}-1\right)dy'\\
  &+\phi(y,\fc)\int_{y_\fc-1}^y\frac{1}{\big(b_0(y')-\fc_r-i\fc_i\big)^2}  dy'.
\end{align*}
For the third term, we write
\begin{align*}
  &\phi(y,\fc)\int_{y_\fc-1}^y\frac{1}{\big(b_0(y')-\fc_r-i\fc_i\big)^2}  dy'
  =\phi(y,\fc)\int_{b_0 (y_\fc-1)}^{b_0 (y)}\frac{1}{\big(v-\fc_r-i\fc_i\big)^2} \frac{1}{b'_0(b_0^{-1}(v))} dv\\
  =&\phi(y,\fc) \left(\frac{1}{\big(b_0 (y_\fc-1)-\fc_r-i\fc_i\big)} \frac{1}{b'_0(y_\fc-1)}-\frac{1}{\big(b_0 (y)-\fc_r-i\fc_i\big)}\frac{1}{b'_0(y)}\right)\\
  &+\phi(y,\fc)\int_{b_0 (y_\fc-1)}^{b_0 (y)}\frac{1}{\big(v-\fc_r-i\fc_i\big)} \pa_v\frac{1}{b'_0(b_0^{-1}(v))} dv.
\end{align*}
Then one can easily check that
\begin{align*}
  \left|\phi(y,\fc)\int_{+\infty}^y\frac{1}{\phi^2(y',\fc)}dy'\right|\le C,\text{ for }y_\fc>y> y_\fc-1,
\end{align*}
and then
\begin{align*}
  \left|\phi(y,\fc)\int_{+\infty}^y\frac{1}{\phi^2(y',\fc)}dy'\right|\le C e^{-C_4|y-y_{\fc}|}\text{ for }y< y_\fc.
\end{align*}
In this way, we get
\begin{align*}
  \left|\Phi_{h,l}(y,\fc)\chi_{y_\fc>y}+\Phi_{h,r}(y,\fc)\chi_{y_\fc<y}\right|\le C e^{-C_4|y-y_{\fc}|}, \forall y_\fc\in\mathbb R,\ 0<|\fc_i|\le \varepsilon_4.
\end{align*}

From \eqref{eq-decay-part1}, we can see that
\begin{align*}
  \left|\Phi_{i,l}(y,\fc)\chi_{y_\fc>y}+\Phi_{i,r}(y,\fc)\chi_{y_\fc<y}\right|\le C \frac{\left\|\mathfrak w_{in}\right\|_{H^1}}{1+|y-y_{\fc}|}.
\end{align*}
However, $\frac{1}{1+|y-y_{\fc}|}$ is not an integrable function. Thus we could not use the Lebesgue-dominated convergence theorem directly. To overcome this difficulty, we introduce an auxiliary function $\mathring\Phi(y,\mathfrak c)$ which solves
\begin{align*}
  \pa_{yy}\mathring\Phi(y,\mathfrak c)-\mathring\Phi(y,\mathfrak c)= i\frac{\mathfrak w_{in}(y)}{b_0(y)-\mathfrak c}.
\end{align*}
By using the fundamental solution of $\pa_{yy}-1$, we write
\begin{align*}
  \mathring\Phi(y,\fc_r+i\fc_i)=&\frac{i}{2}\int_{\mathbb R} e^{-|y-y'|}\frac{\mathfrak w_{in}(y')}{b_0(y')-\fc_r-i\fc_i}dy'\\
  =&\frac{i}{2}\int_{\mathbb R} e^{-|y-b_0^{-1}(v)|}\frac{\mathfrak w_{in}(b_0^{-1}(v))}{v-\fc_r-i\fc_i}\frac{1}{b'_0(b_0^{-1}(v))}d v.
\end{align*}
It holds for all $y\in\mathbb R$ that
\begin{equation}\label{eq-mr-Phi}
  \begin{aligned}    
  \lim_{\fc_i\to0\pm}\mathring\Phi(y,\fc_r+i\fc_i)=&\text{P.V.}\frac{i}{2}\int_{\mathbb R} e^{-|y-y'|}\frac{\mathfrak w_{in}(y')}{b_0(y')-\fc_r}dy'\mp \frac{\pi}{2}e^{-|y-y_\fc|}\mathfrak w_{in}(y_\fc)\frac{1}{b'_0(y_\fc)}\\
  \eqdef& \mathring\Phi^{\pm}(y,\fc_r).    
  \end{aligned}
\end{equation}
By applying the same technique as in Lemma \ref{lem-E1E2}, one can see that  $\left\|\mathring\Phi(y,\mathfrak c)\right\|_{H^1_y}$ has an uniformly bound.  As $\mathfrak w_{in}$ has compact support, there exists a constant $Y$ such that $\mathfrak w_{in}(y)=0$ for $|y|\ge Y$, and
\begin{align*}
  \left\|\mathring\Phi(y,\mathfrak c)\right\|_{H^1_y}\le C \frac{\left\|\mathfrak w_{in}\right\|_{L^2_y}}{|y_\fc|-Y}.
\end{align*}

We can see that both $\Phi(y,\fc)$ given in \eqref{eq-Phi-exp} and $\mathring\Phi(y,\mathfrak c)$ decay in $y$ at infinity. It holds that 
\begin{align*}
  \left(\pa_{yy}-1-\frac{b_0''(y)}{b_0(y')-\fc}\right)\left(\Phi(y,\fc)-\mathring\Phi(y,\mathfrak c)\right)=\frac{b_0''(y)\mathring\Phi(y,\mathfrak c)}{b_0(y')-\fc},
\end{align*}
and
\begin{align*}
  &\Phi(y,\fc)-\mathring\Phi(y,\mathfrak c)\\
  =&i\phi(y,\fc)\int_{-\infty}^y\frac{\int_{y_\fc}^{y'}b_0''(y'')\mathring\Phi(y'',\fc)\phi_{1}(y'',\fc) dy''}{\phi^2(y',\fc)}dy'-i\mu(ib_0''\mathring\Phi,\fc)\phi(y,\fc)\int_{-\infty}^y\frac{1}{\phi^2(y',\fc)}dy'.
\end{align*}
Then similar to \eqref{eq-decay-part1}, we have for $y<y_\fc$ and $y_\fc\ge 2Y$ that
\begin{align*}
  \left|\Phi(y,\fc)-\mathring\Phi(y,\mathfrak c)\right|\le& C  \frac{\left\|b_0''\mathring\Phi(\cdot,\fc)\right\|_{H^1}}{1+|y-y_\fc|}\le  C \frac{1}{(1+|y-y_\fc|) \left|y_\fc-Y\right|}.
\end{align*}
In the same way, we also have for $y>y_\fc$ and $y_\fc\le -2Y$ that
\begin{align*}
  \left|\Phi(y,\fc)-\mathring\Phi(y,\mathfrak c)\right|\le& C\frac{1}{(1+|y-y_\fc|)\left|y_\fc+Y\right|}.
\end{align*}

Then we rewrite \eqref{eq-rep-Psi2} as
\begin{equation}\label{eq-rep-Psi3}
  \begin{aligned}    
   \Psi(t,y)=&\lim_{\fc_i\to0+}\frac{1}{2\pi}\lim_{T\to+\infty} \int^{T}_{-T} e^{-i \left(\fc_r-i\fc_i\right) t}\Phi(y,\fc_r-i\fc_i) -\mathring\Phi(y,\fc_r-i\fc_i) d\fc_r\\
  &-\lim_{\fc_i\to0+}\frac{1}{2\pi }\lim_{T\to+\infty} \int^{T}_{-T} e^{-i \left(\fc_r+i\fc_i\right) t}\Phi(y,\fc_r+i\fc_i) -\mathring\Phi(y,\fc_r+i\fc_i) d\fc_r\\
  &+\lim_{\fc_i\to0+}\frac{1}{2\pi}\lim_{T\to+\infty} \int^{T}_{-T} e^{-i \left(\fc_r-i\fc_i\right) t} \mathring\Phi(y,\fc_r-i\fc_i) d\fc_r\\
  &-\lim_{\fc_i\to0+}\frac{1}{2\pi }\lim_{T\to+\infty} \int^{T}_{-T} e^{-i \left(\fc_r+i\fc_i\right) t} \mathring\Phi(y,\fc_r+i\fc_i) d\fc_r.
  \end{aligned}
\end{equation}
For each $\fc_i=c^{**}>0$, let $\fc=i\tilde \fc$, we deduce that
\begin{align*}
  &\frac{1}{2\pi}\lim_{T\to+\infty} \int^{T}_{-T} e^{-i \left(\fc_r-i\fc_i\right) t} \mathring\Phi(y,\fc_r-i\fc_i) d\fc_r\\
  =&\frac{1}{2\pi} \frac{i}{2}\lim_{T\to+\infty} \int^{T}_{-T} e^{-i \left(\fc_r-i\fc_i\right) t} \int_{\mathbb R} e^{-|y-y'|}\frac{\mathfrak w_{in}(y')}{b_0(y')-\fc_r+i\fc_i}dy' d\fc_r\\
  =&\frac{1}{2\pi} \frac{i}{2}\lim_{T\to+\infty} \int_{\mathbb R}e^{-|y-y'|}\mathfrak w_{in}(y')\int^{b_0(y')+T}_{b_0(y')-T}   \frac{e^{-i \left(\fc_r-i\fc_i\right) t}}{b_0(y')-\fc_r+i\fc_i}d\fc_rdy'\\
  &-\frac{1}{2\pi} \frac{i}{2}\lim_{T\to+\infty} \int_{\mathbb R}e^{-|y-y'|}\mathfrak w_{in}(y')\int^{b_0(y')+T}_{T}   \frac{e^{-i \left(\fc_r-i\fc_i\right) t}}{b_0(y')-\fc_r+i\fc_i}d\fc_rdy'\\
  &-\frac{1}{2\pi} \frac{i}{2}\lim_{T\to+\infty} \int_{\mathbb R}e^{-|y-y'|}\mathfrak w_{in}(y')\int^{b_0(y')-T}_{-T}   \frac{e^{-i \left(\fc_r-i\fc_i\right) t}}{b_0(y')-\fc_r+i\fc_i} d\fc_rdy'\\
  =&\frac{1}{2\pi} \frac{i}{2}\lim_{T\to+\infty} \int_{\mathbb R}e^{-|y-y'|}\mathfrak w_{in}(y')e^{-ib_0(y')t} \int^{T-ic^{**}}_{-T-ic^{**}}   \frac{e^{-i \fc t}}{-\fc}d\fc dy'\\
  =&\frac{1}{2\pi i} \frac{1}{2}\lim_{T\to+\infty} \int_{\mathbb R}e^{-|y-y'|}\mathfrak w_{in}(y')e^{-ib_0(y')t}\int^{c^{**}+iT}_{c^{**}-iT}   \frac{e^{\tilde\fc t}}{\tilde\fc} d \tilde\fc dy'\\
  =&\frac{1}{2} \int_{\mathbb R}e^{-|y-y'|}\mathfrak w_{in}(y') e^{-ib_0(y')t}H(t) dy',
\end{align*}
where $H(t)=\frac{1}{2}(1+\text{sgn}(t))$ is the Heaviside step function. In this way, one can get for each $\fc_i=c^{**}>0$ that
\begin{align*}
   &\frac{1}{2\pi}\lim_{T\to+\infty} \int^{T}_{-T} e^{-i \left(\fc_r-i\fc_i\right) t} \mathring\Phi(y,\fc_r-i\fc_i) d\fc_r\\
  &- \frac{1}{2\pi}\lim_{T\to+\infty} \int^{T}_{-T} e^{-i \left(\fc_r+i\fc_i\right) t} \mathring\Phi(y,\fc_r+i\fc_i) d\fc_r
  =\frac{1}{2} \int_{\mathbb R}e^{-|y-y'|}\mathfrak w_{in}(y')e^{-ib_0(y')t}dy'=\mathring \psi(t,y).
\end{align*}
Indeed, $\mathring \psi$ is the stream function associated with the vorticity $\mathring w$ that solves
\begin{align*}
  \pa_t \mathring w(t,y)+ i b_0(y) \mathring w(t,y)  =0,\qquad w(0,y)=\mathfrak w_{in}(y).
\end{align*}
The above calculation is the inverse Fourier-Laplace transform of the Fourier-Laplace transform of $\mathring \psi$ which is the same to \eqref{eq-rep-Psi2}. Besides, for the real part of $\mathring\Phi^{\pm}(y,\fc_r)$ given in \eqref{eq-mr-Phi}, we have
\begin{equation}\label{eq-lim-mr-Phi}
  \begin{aligned}    
    &\frac{1}{2\pi} \int^{+\infty}_{-\infty} e^{-i \fc_r t} \left( - \mathfrak R \mathring\Phi^{-}(y,\fc_r)\right)  d\fc_r-\frac{1}{2\pi }\int^{+\infty}_{-\infty} e^{-i \fc_r t}\left( - \mathfrak R \mathring\Phi^{+}(y,\fc_r)\right) d\fc_r\\
    =&-\frac{1}{2}\int^{+\infty}_{-\infty}e^{-i \fc_r t} e^{-|y-y_\fc|}\mathfrak w_{in}(y_\fc)\frac{1}{b'_0(y_\fc)} d\fc_r=-\frac{1}{2} \int_{\mathbb R}e^{-|y-y'|}\mathfrak w_{in}(y')e^{-ib_0(y')t}dy'.
  \end{aligned}
\end{equation}

Then by the Lebesgue-dominated convergence theorem, we derive from \eqref{eq-rep-Psi3}, Lemma \ref{lem:limit}, \eqref{eq-mr-Phi}, and \eqref{eq-lim-mr-Phi} that
\begin{equation}\label{eq-rep-Psi4}
  \begin{aligned}    
   \Psi(t,y)=
   %&\lim_{\fc_i\to0+}\frac{1}{2\pi} \int^{+\infty}_{-\infty} e^{-i \left(\fc_r-i\fc_i\right) t}\Phi(y,\fc_r-i\fc_i) -\mathring\Phi(y,\fc_r-i\fc_i) d\fc_r\\
  %&-\lim_{\fc_i\to0+}\frac{1}{2\pi }\int^{+\infty}_{-\infty} e^{-i \left(\fc_r+i\fc_i\right) t}\Phi(y,\fc_r+i\fc_i) -\mathring\Phi(y,\fc_r+i\fc_i) d\fc_r\\
  %&+\frac{1}{2} \int_{\mathbb R}e^{-|y-y'|}\mathfrak w_{in}(y')e^{-ib_0(y')t}dy'\\
  &\frac{1}{2\pi} \int^{+\infty}_{-\infty}\lim_{\fc_i\to0+} e^{-i \left(\fc_r-i\fc_i\right) t}\Phi(y,\fc_r-i\fc_i) -\mathring\Phi(y,\fc_r-i\fc_i) d\fc_r\\
  &-\frac{1}{2\pi }\int^{+\infty}_{-\infty}\lim_{\fc_i\to0+} e^{-i \left(\fc_r+i\fc_i\right) t}\Phi(y,\fc_r+i\fc_i) -\mathring\Phi(y,\fc_r+i\fc_i) d\fc_r\\
  &+\frac{1}{2} \int_{\mathbb R}e^{-|y-y'|}\mathfrak w_{in}(y')e^{-ib_0(y')t}dy'\\
  =&-\frac{1}{\pi}\int_{-\infty}^{b_0(y)} e^{-i\fc_r t} \phi(y,\fc_r)\frac{\mathcal J_1\mathcal J_4+\mathcal J_2\mathcal J_3}{\mathcal J_1^2+\mathcal J_2^2}\int_{+\infty}^y\frac{1}{\phi^2(y',\fc_r)}dy'd\fc_r\\
      &-\frac{1}{\pi}\int_{b_0(y)}^{+\infty} e^{-i\fc_r t}\phi(y,\fc_r) \frac{\mathcal J_1\mathcal J_4+\mathcal J_2\mathcal J_3}{\mathcal J_1^2+\mathcal J_2^2}\int_{-\infty}^y\frac{1}{\phi^2(y',\fc_r)}dy'd\fc_r.
  \end{aligned} 
\end{equation}
This finishes the proof of this proposition.
\end{proof}
}
\section{Enhanced dissipation}\label{appendix-C}
In this section, we prove the enhanced dissipation of the solution to \eqref{eq:b-perturbation-linear}. We have the following proposition. 
\begin{proposition}
Let $\om$ be the solution of \eqref{eq:b-perturbation-linear} with initial data $\om_{in}$, then there exist $t_0$ such that for $0\leq t\leq t_0$,
\beno
\|P_{\neq }\om\|_{L^2}\leq Ce^{C\g t}\|P_{\neq}\om_{in}\|_{L^2},
\eeno
and for $t\geq t_0$, 
\beno
\|P_{\neq}\om\|_{L^2}\leq Ce^{C\g t_0}e^{-c\nu^{\f13}t}\|P_{\neq}\om_{in}\|_{L^2}.
\eeno 
Here the constants $c, C$ are independent of $t,\nu,\g$. 
\end{proposition}
\begin{proof}
We use the idea in the proof of Proposition \ref{prop-lin-upper}. 
Let $z=x-ty$, $f(t,z,y)=\om(t,x,y)$, and $\hat{f}_k(t,\xi)=\f{1}{4\pi^2}\int_{\mathbb{T}\times \mathbb{R}}f(t,z,y)e^{-ikz-i\xi y}dzdy$ then 
\begin{align*}
\pa_t\hat{f}_k(t,\xi)&+ik \mathcal{F}_{(z,y)\to (k,\xi)}\big((b-y)f(t,z,y)\big)+\nu (k^2+(\xi-kt)^2)\hat{f}_k(t,\xi)\\
&-\int_{\mathbb{R}}M\gamma^2(\xi-\eta)e^{-(\nu t+\frac{\gamma^2}{4})|\xi-\eta|^2}\frac{k\hat f_k(t,\eta)}{(\eta-kt)^2+k^2} d\eta =0.
\end{align*}
Let $A_k(t,\xi)=e^{\arctan(\f{\xi}{k}-t)}e^{\arctan((\nu k^2)^{\f13}(\f{\xi}{k}-t))}$, then $e^{-\pi}\leq A_k(t,\xi)\leq e^{\pi}$ and 
\begin{align*}
\f{\pa_tA_k(t,\xi)}{A_k(t,\xi)}=-\f{1}{(\f{\xi}{k}-t)^2+1}-\f{(\nu k^2)^{\f13}}{(\nu k^2)^{\f23}(\f{\xi}{k}-t)^2+1}.
\end{align*}
Then we have
\begin{align*}
&\f12\f{d}{dt}\|A_k(t,\xi)\hat{f}_k(t,\xi)\|_{L^2}^2
+ \nu\|(k^2+(\xi-kt)^2)^{\f12}A_k(t,\xi)\hat{f}_k(t,\xi)\|_{L^2}^2\\
&\quad+\left\|\sqrt{\f{k^2}{(\xi-kt)^2+k^2}}A_k(t,\xi)\hat{f}_k(t,\xi)\right\|_{L^2}^2
+\left\|\sqrt{\f{(\nu k^2)^{\f13}}{(\nu k^2)^{\f23}(\f{\xi}{k}-t)^2+1}}A_k(t,\xi)\hat{f}_k(t,\xi)\right\|_{L^2}^2\\
&\leq C\|b-y\|_{L^{\infty}}\||k|^{\f12}A_k(t,\xi)\hat{f}_k(t,\xi)\|_{L^2}^2\\
&\quad+\left|\int_{\mathbb{R}^2}M\gamma^2(\xi-\eta)e^{-(\nu t+\frac{\gamma^2}{4})|\xi-\eta|^2}\frac{A_k(t,\xi)k\hat f_k(t,\eta)}{(\eta-kt)^2+k^2}A_k(t,\xi)\hat f_k(t,\xi) d\eta d\xi\right|\\
&\leq CM\g^2\||k|^{\f12}A_k(t,\xi)\hat{f}_k(t,\xi)\|_{L^2}^2\\
&\quad+C\f{M}{|k|^{\f12}}\left\|\gamma^2\eta e^{-(\nu t+\frac{\gamma^2}{4})|\eta|^2}\right\|_{L^2_{\eta}}\left\|\sqrt{\f{k^2}{(\xi-kt)^2+k^2}}A\hat{f}_k(t,\xi)\right\|_{L^2}\|A\hat f_k(t,\xi)\|_{L^2}.
\end{align*}
Here we use the fact that 
\beno
\left\|\sqrt{\f{1}{(\xi-kt)^2+k^2}}\right\|_{L^2_{\xi}}\leq \f{1}{|k|^{\f12}}. 
\eeno
We also have 
\begin{align*}
  \left\|\gamma^2\eta e^{-(\nu t+\frac{\gamma^2}{4})|\eta|^2}\right\|_{L^2_{\eta}}\le C \frac{\gamma^2}{(\nu t+\frac{\gamma^2}{4})^{\frac{3}{4}}},
\end{align*}
and
\begin{align*}
&\nu\|(k^2+(\xi-kt)^2)^{\f12}A_k(t,\xi)\hat{f}_k(t,\xi)\|_{L^2}^2
+\left\|\sqrt{\f{(\nu k^2)^{\f13}}{(\nu k^2)^{\f23}(\f{\xi}{k}-t)^2+1}}A_k(t,\xi)\hat{f}_k(t,\xi)\right\|_{L^2}^2\\
&\geq c\left\|(\nu k^2)^{\f16}A_k(t,\xi)\hat{f}_k(t,\xi)\right\|_{L^2}^2+c\left\|(\nu k^2)^{\f12}A_k(t,\xi)\hat{f}_k(t,\xi)\right\|_{L^2}^2+c\left\|(\nu k^2)^{\f14}A_k(t,\xi)\hat{f}_k(t,\xi)\right\|_{L^2}^2.
\end{align*}
Thus we have for $0\leq t\leq t_0$ and $\g^2\ll \nu^{\f12}$
\begin{align}\label{eq:|k|large}
\f12\f{d}{dt}\|A_k(t,\xi)\hat{f}_k(t,\xi)\|_{L^2}^2+c\left\|(\nu k^2)^{\f16}A_k(t,\xi)\hat{f}_k(t,\xi)\right\|_{L^2}^2
\leq C \f{M^2\g}{|k|} \|A\hat f_k(t,\xi)\|_{L^2}^2,
\end{align}
and for $t\geq t_0$
\begin{align*}
&\f12\f{d}{dt}\|A_k(t,\xi)\hat{f}_k(t,\xi)\|_{L^2}^2+c\left\|(\nu k^2)^{\f16}A_k(t,\xi)\hat{f}_k(t,\xi)\right\|_{L^2}^2+\f12\left\|\sqrt{\f{k^2}{(\xi-kt)^2+k^2}}A\hat{f}_k(t,\xi)\right\|_{L^2}^2\\
&\leq C M\f{\g^2\nu^{-\f16}}{(\nu t_0+\frac{\gamma^2}{4})^{\f34}} \nu^{\f16}\|A\hat f_k(t,\xi)\|_{L^2}\left\|\sqrt{\f{k^2}{(\xi-kt)^2+k^2}}A\hat{f}_k(t,\xi)\right\|_{L^2},
\end{align*}
The proposition follows directly by taking $t_0\geq CM^{\f43}\nu^{-\f13-\f{16}{9}\d_0}$. 

Note that by the estimate \eqref{eq:|k|large}, for $k\gg v^{-\f{2}{5}\d_0}$, it holds that
\beno
\|\hat{f}_k(t,\xi)\|_{L^2}^2\leq \|\hat{f}(0,\xi)\|_{L^2}^2,\quad \text{for}\ t\leq t_0,
\eeno
which means that the growth can only happen for small $k$. 
\end{proof}

\end{appendix}

\section*{Acknowledgements}
The work of N. M. is supported by NSF grant DMS-1716466 and by Tamkeen under the NYU Abu Dhabi Research Institute grant of the center SITE. 

\bibliographystyle{siam.bst} 
\bibliography{unstableshearflow.bib}

\end{document}